\documentclass[11pt]{amsart}
\RequirePackage{amsthm,amsmath,amsfonts,amssymb}
\usepackage{mathrsfs}
\RequirePackage{graphicx}
\usepackage[top=3cm,bottom=2cm,left=3cm,right=3cm,marginparwidth=1.75cm]{geometry}
\usepackage[english]{babel}
\usepackage[utf8]{inputenc}
\usepackage[T1]{fontenc}
\usepackage{wrapfig}
\usepackage{physics}
\usepackage{ dsfont }

\usepackage{subcaption}
\usepackage[dvipsnames]{xcolor}
\usepackage[colorlinks=true, allcolors=BlueViolet, pagebackref=false]{hyperref}
\usepackage[shortlabels]{enumitem}
\usepackage[all,cmtip]{xy}
\usepackage{tikz-cd}
\usepackage{pgfplots}
\pgfplotsset{compat=1.17}
\usetikzlibrary{intersections}
\usepackage{relsize}
\usepackage{subcaption}
\usepackage{dsfont}

\usepackage{xcolor, import}
\graphicspath{{Figures/}}

\usepackage{cleveref}
\usepackage{autonum}
\usepackage{xspace} 
\usepackage{cancel}
\usepackage[normalem]{ulem}
\usepackage{fancyvrb}

\numberwithin{equation}{section}

\theoremstyle{plain}
\newtheorem{theorem}{Theorem}[section]
\newtheorem{proposition}[theorem]{Proposition}
\newtheorem{corollary}[theorem]{Corollary}
\newtheorem{lemma}[theorem]{Lemma}
\newtheorem{definition}[theorem]{Definition}

\newtheorem{assumption}{Assumption}

\theoremstyle{remark}

\newtheorem{claim}[theorem]{Claim}
\newtheorem{remark}[theorem]{Remark}

\crefname{equation}{Equation}{Equations}
\crefname{multline}{Equation}{Equations}
\crefname{figure}{Figure}{Figures}
\crefname{subfigure}{Figure}{Figures}
\crefname{question}{Question}{Question}
\crefname{section}{Section}{Sections}
\crefname{subsection}{Subsection}{Subsections}
\crefname{lemma}{Lemma}{Lemmas}
\crefname{proposition}{Proposition}{Propositions}
\crefname{theorem}{Theorem}{Theorems}
\crefname{innercustomthm}{Theorem}{Theorems}
\crefname{mainthm}{Theorem}{Theorems}
\crefname{corollary}{Corollary}{Corollaries}
\crefname{definition}{Definition}{Definitions}
\crefname{remark}{Remark}{Remarks}
\crefname{proposition}{Proposition}{Proposition}
\crefname{corollary}{Corollary}{Corollaries}
\crefname{example}{Example}{Examples}
\crefname{claim}{Claim}{Claim}
\crefname{conjecture}{Conjecture}{Conjecture}
\crefname{assumption}{Assumption}{Assumption}
\crefname{enumi}{}{}

\newcommand{\R}{\mathbb{R}}

\newcommand{\N}{\mathbb{N}}
\newcommand{\PP}{\mathbb{P}}

\newcommand{\E}{\mathbb{E}}

\newcommand{\HX}{\mathcal{H}_X}

\newcommand{\de}{\partial}

\newcommand{\Mlv}{M}
\newcommand{\dmlv}{\mathrm{D}_{\Mlv}}
\newcommand{\inter}[1]{%
  {\kern0pt#1}^{\mathrm{o}}%
}

\newcommand{\mathbbm}[1]{\mathds{#1}}

\newcommand{\f}{\varphi}
\renewcommand{\a}{\alpha}
\newcommand{\e}{\varepsilon}

\renewcommand{\grad}{\mathrm{grad}}
\newcommand{\Vect}{\mathfrak{X}}
\newcommand{\g}{\grad}
\newcommand{\hess}{\mathrm{Hess}}
\newcommand{\h}{\mathrm{H}{\mathrm{ess}}^\sharp}
\renewcommand{\H}{\hess}
\newcommand{\tDelta}{\tilde{\Delta}}
\newcommand{\n}{\mathbf {n}}

\newtheorem*{theorem*}{\bf Theorem}

\def\randin{%
  \mathchoice%
    {\raisebox{-.35ex}{$\displaystyle{^\subset}$}\mkern-11.5mu\raisebox{+.45ex}{$\displaystyle{_\subset}$}}
    {\mkern+1mu\raisebox{-.27ex}{$\textstyle{^\subset}$}\mkern-11.7mu\raisebox{+.45ex}{$\textstyle{_\subset}$}}
    {\raisebox{.35ex}{$\scriptstyle\subset$}\mkern-14mu\raisebox{-.15ex}{$\scriptstyle\subset$}}
    {\raisebox{.3ex}{$\scriptscriptstyle\subset$}\mkern-13.5mu\raisebox{-.10ex}{$\scriptscriptstyle\subset$}}
}

\newcommand{\zkrok}{\emph{z-KROK}\xspace}

\newcommand\restr[2]{{
  \left.\kern-\nulldelimiterspace 
  #1 
  \vphantom{\big|} 
  \right|_{#2} 
  }}

\DeclareMathOperator{\vol}{\mathcal{H}}

\newcommand{\seg}{\underline}

\newcommand{\transv}{\pitchfork}
\newcommand{\mC}{\mathcal{C}}
\newcommand{\m}[1]{\mathcal{#1}}
\newcommand{\be}{\begin{equation}}
\newcommand{\ee}{\end{equation}}

\newcommand{\bega}{\begin{equation}\begin{aligned}}
\newcommand{\eega}{\end{aligned}\end{equation}}
\newcommand{\kop}{\left\{}
\newcommand{\pok}{\right\}}
\newcommand{\tyu}{\left(}
\newcommand{\uyt}{\right)}
\newcommand{\qwe}{\left[}
\newcommand{\ewq}{\right]}

\newcommand{\CZ}{\text{CZ}}
\newcommand{\Prob}{(\Omega, \mathfrak{S},\PP)}

\begin{document}
\title[Differentiable Nodal Volumes]{Nodal Volumes as Differentiable\\ Functionals of Gaussian fields}
\author{Giovanni Peccati and Michele Stecconi}
\date{\today}

\begin{abstract}
We characterize the absolute continuity of the law and the Malliavin-Sobolev regularity of random nodal volumes associated with smooth Gaussian fields on generic $\mathcal{C}^2$ manifolds with arbitrary dimension. Our results extend and generalize the seminal contribution by Angst and Poly (2020) about stationary fields on Euclidean spaces and cover, in particular, the case of two-dimensional manifolds, possibly with boundary and corners. The main tools exploited in the proofs include the use of Gaussian measures on Banach spaces, Morse theory, and the characterization of Malliavin-Sobolev spaces in terms of ray absolute continuity. Several examples are analyzed in detail. 

\end{abstract}



\maketitle
\setcounter{tocdepth}{1}
\tableofcontents
\section{Introduction}\label{sec:intro}
\textit{For the rest of the paper, and unless otherwise specified, every random element is assumed to be defined on an adequate common probability space $\Prob$, with $\mathbb{E}$ indicating expectation with respect to $\mathbb{P}$.}

\subsection{Overview and motivation} The aim of this work is to develop new geometric and probabilistic tools for studying fine properties of \emph{random nodal volumes} (that is, measures of vanishing loci) associated with smooth Gaussian random functions on compact Riemannian manifolds, with special emphasis on their differentiability and on the non-singularity of their laws. One of our main achievements (see Theorem \ref{t:mainintro}, as well as Section \ref{ss:motivintro}, for a discussion) is an almost exhaustive characterization of random nodal volumes as elements of the {\it Malliavin-Sobolev spaces} $\mathbb{D}^{1,p}$ (see Section \ref{ss:malliavin} below, as well as \cite[Section 5.2]{bogachev} and \cite{nourdinpeccatibook, nualartbook}), when $p\in \{1,2\}$. As demonstrated in the subsequent discussion, the approach developed in the present work 
significantly extends  (by substantially different methods) 
the seminal contribution by J. Angst and G. Poly \cite{PolyAngst}, where the absolute continuity of the laws of random nodal volumes was studied for the first (and unique!) time by using tools from the Malliavin calculus of variations.


The main findings of \cite{PolyAngst} can be summarized as follows. Let $m\geq 2$, and let $M$ be either the $m$-dimensional torus $(\mathbb S^1)^m = \mathbb{T}^m$ or a closed rectangle in $\R^m $ of positive Lebesgue measure. We consider a centered Gaussian random field $f = \{f(x) : x\in M\}$ such that $f$ is a.s. of class $\m C^2$, stationary and unit variance. The {\rm nodal volume} of $f$ is the random quantity $V(f) := \m H^{m-1}(f^{-1}(0) \cap M)$, where $\m H^{m-1}$ indicates the $(m-1)$--dimensional Haussdorff measure. The next statement, proved in \cite{PolyAngst}, yields the non-singularity and Malliavin-Sobolev smoothness of $V(f)$ under some further assumptions on $f$. 

\begin{theorem}[Theorems 1 and 2 in \cite{PolyAngst}]\label{t:paintro} {In addition to the previous requirements, assume that the $m$-dimensional vector $\nabla f(0)$ has a density. Then, $V(f)$ is not a.s. equal to a constant. Moreover, if $m \geq 3$ one has that: {\rm (i)} $V(f)$ belongs to the Malliavin-Sobolev space $\mathbb{D}^{1,\eta}$ for all $\eta \in [1,\frac{m+1}{3} )$, and {\rm (ii)} its law has a non-zero component that is absolutely continuous with respect to the Lebesgue measure.}
\end{theorem}

Note that Point (i) and (ii) of the previous statement do not cover the case $m=2$, which seems to be outside the scope of the proof techniques exploited in \cite{PolyAngst}, see \cref{rem:last}. Classically, when coupled with stationarity, the requirement that $\nabla f(0)$ has a density implies that $0$ is a.s. a regular value of $f$, and therefore $f^{-1}(0)$ is a $(m-1)$-dimensional $\mC^2$ (random) submanifold of $M$. The strategy adopted in \cite{PolyAngst} for proving the above statement hinges upon an elegant representation of $V(f)$ as a Riemann integral involving some smooth transformation of $f$ and its derivatives up to the order two (see \cite[Proposition 7 and Section 4.5.1]{PolyAngst}). Such a closed formula can then be used to study the Malliavin-Sobolev regularity of $V(f)$ by leveraging the fact that each space $\mathbb{D}^{1,\eta}$ is the closure in an appropriate norm of the class of smooth cylindrical functions (see e.g. \cite[Section 1.2]{nualartbook} and \cite[Section 2.3]{nourdinpeccatibook}, as well as Section \ref{ss:malliavin} below). Once the Malliavin smoothness of $V(f)$ is established, the non-singularity of its law follows by a classical criterion due to N. Bouleau (see \cite[Theorem 4]{PolyAngst}). 

The principal aim of the present work is to address the following general problem:

\smallskip
\noindent\textsc{Problem A.} \textit{ For $m\geq 2$, let $M$ be a $m$-dimensional Riemannian manifold, let $X$ be a smooth Gaussian field defined on $M$, and write  $ V(X) := \mathcal{H}^{m-1}( X^{-1}(0))$ for the nodal volume of $X$. Under which conditions on $X$ and $M$ is the law of $V(X)$ non-singular with respect to the Lebesgue measure? Under which conditions is $V(X)$ an element of some Malliavin-Sobolev space $\mathbb{D}^{1,\eta}$, with $\eta\geq 1$?   }


\begin{remark}\label{rem:last}
\begin{enumerate}
\item[(a)] To the best of our understanding, it would not be possible to attack Problem A by a direct application of the approach developed in \cite{PolyAngst} (or in the companion paper \cite{Jubin}). In particular, the bounds on the integrands associated with exact Kac-Rice formulae derived in \cite[Section 4.2.2]{PolyAngst} are ineffective in dimension $m=2$, even in the case of a stationary field defined on the plane.  We will see that our geometric approach allows one to completely bypass such a difficulty.

\item[(b)] To keep the length of our work within bounds, and in view of their central role in the geometric analysis of random fields (see e.g. \cite{AdlerTaylor, AzaisWscheborbook}), we decided in the present paper to only focus on nodal volumes of real-valued Gaussian fields. As our discussion will demonstrate, our approach would naturally extend to more general (and, possibly, multidimensional) geometric functionals, like e.g. {volumes of level set intersections} \cite{Dalmao2021, notarnicolaAHL} and {multivariate valuations of level set measures} \cite{NPV, PV}.

\end{enumerate}
\end{remark}

The crucial idea developed in the sections to follow is that one can address Problem A in full generality (covering in particular the two-dimensional and non-stationary case), by directly exploiting the equivalent characterization of the space $\mathbb{D}^{1,p}$ ($p \geq 1$) as the collection of those random variables $\Psi = \Psi(X)\in L^p(\PP)$ verifying the three properties
\begin{enumerate}[(a)]
    \item $\Psi$ is {\it ray absolutely continuous}, meaning that, for every element $h$ of the Cameron-Martin space of $X$, the mapping $t\mapsto \Psi(X+th)$ admits an absolutely continuous modification; \label{p:rac}
\item $\Psi$ admits a {\it Fr\'echet-type stochastic derivative}, written $\dmlv \Psi$; 
\item $\dmlv \Psi$ is integrable with respect to an appropriate norm (whose definition depends on the exponent $p$).
\end{enumerate}
By construction, one has that $\mathbb{D}^{1,q} \subset \mathbb{D}^{1,p}$, for all $1\leq p<q$. The reader is referred to Definition \ref{def:Dspace} below, as well as \cite[Definitions 5.2.3 and 5.2.4]{bogachev}, for details; we observe that the fact that properties (a)--(c) fully characterize $\mathbb{D}^{1,p}$ is a deep result in infinite-dimensional Gaussian analysis (proved e.g. in \cite[Thm. 5.7.2]{bogachev}), mirroring the usual characterization of classical Sobolev spaces in terms of ``absolute continuity on lines'', see e.g. \cite[Section 11.3]{leonibook}. Interestingly, our analysis will show that the study of the non-singularity 
 of the law of $V(X)$ can be disentangled from the study of its Malliavin-Sobolev regularity, see Section \ref{sec:nons} (in particular, Remark \ref{r:whatisused}) for a full discussion of this point. We now present one of the general statements proved in our work --- which represents a substantial extension of Theorem \ref{t:paintro} and a general solution to Problem A. For the rest of this section, we will use some notions and notation from differential geometry; see Section \ref{ss:riemannelements} below, and the references therein, for relevant definitions.
 
\begin{theorem}\label{t:mainintro} Let $M$ be a $\mC^2$ compact Riemannian manifold with dimension $m\geq 2$ (possibly with boundary $\partial M$) and let $X = \{X(p) : p\in M\}$ be a centered Gaussian field which is $\mathbb{P}$-a.s. of class $\mathcal{C}^2(M)$. Write $V(X) := \mathcal{H}^{m-1}( X^{-1}(0))$ for the nodal volume of $X$, and assume moreover that, for all $p\in M$ and all non-zero $v\in T_pM$, the Gaussian vector $(X(p), d_pX(v) ) \randin \mathbb{R}^2$ admits a density. Then, $X^{-1}(0)$ is almost surely a neat $\mC^2$ hypersurface and the following holds:
\begin{enumerate}[\rm (i)]
    \item For all $m\geq 2$, $V(X)$ is square-integrable.  Moreover, provided $V(X)$ is not $\mathbb{P}$-a.s. equal to a constant, one has that the law of $V(X)$ and the Lebesgue measure are not mutually singular, and the topological support of the law of $V(X)$ is a (possibly unbounded) interval. Finally, if the topological support of the law of $X$ is given by the whole set $\mathcal{C}^2(M)$, then there exist $P\in (0,1)$ and a probability density $\pi(x)$ with support equal to $[0,+\infty)$ such that, for all Borel subset $I\subseteq \R$,
    \begin{equation}\label{e:magnificentlaw}
\mathbb{P}(V(X)\in I) = P\,  \delta_0(I) + (1-P) \int_I \pi(x)\, dx,
    \end{equation}
    where $\delta_0$ is the Dirac mass in zero.
\item One has that $V(X) \in \mathbb{D}^{1,1}$ for $m\in \{2,3\}$, and $V(X) \in \mathbb{D}^{1,2}$ for all $m\geq 4$.
\item When $m=3$, a sufficient condition ensuring that $V(X) \in \mathbb{D}^{1,2}$ is that, for all $p\in M$ and all non-zero $v\in T_pM$, the Gaussian vector $(X(p), d_pX(v), {\rm Hess}_p X(v,v) )$, with values in $ \mathbb{R}^3$, admits a density.
 \item When $m=2$, assume in addition that for every pair $p,q\in \partial M$, there exist $u\in T_p M,  v\in T_q M$ such that the vectors $(X(p), d_pX(u) )$ and $(X(q), d_qX(v) )$ are not fully correlated;
 then, $V(X) \notin \mathbb{D}^{1,2}$, provided there exists no curve $Z\subset M$ such that $Z$ and $X^{-1}(0)$ are $\mathcal{C}^1$-isotopic in $M$ with probability 1.

 
\end{enumerate}
\end{theorem}

\begin{remark}\label{r:marketing}\begin{enumerate}[\rm (a)]
\item A hypersurface of a manifold with boundary is \emph{neat} if it intersects the boundary in a transverse way, see \cref{def:neat} from \cite[Section 1.4]{Hirsch}. The fact that $X^{-1}(0)$ is a neat $\mC^2$ hypersurface is ensured by \cref{prop:neat} and follows from Bulinskaya \cref{lem:bulinskaya}, by standard arguments. 
     \item Point (i) in Theorem \ref{t:mainintro} combines the contents of Theorem \ref{thm:mainabs}, Proposition \ref{prop:interval}, Corollary \ref{cor:volL2} and Theorem \ref{thm:sing} below; Point (ii) follows from Theorems \ref{thm:D12} and \ref{thm:D11}; Point (iii) is again a consequence of Theorem \ref{thm:D12}; finally, Point (iv) can be derived from Theorem \ref{thm:D12boundary}. We observe that, in contrast to Theorem \ref{t:paintro}, our findings cover the case of two-dimensional manifolds, and do not involve any notion of stationarity.
     \item As shown in Theorem \ref{thm:D12corner}, the content of Point (ii) (in the case $m\geq 3$) and Point (iii) extends to the case of a manifold $M$ \emph{with corners}, therefore including the situation in which $M$ is a closed rectangle in $\R^m$. As a consequence, when applied to the case where $m\in \{3,4,5\}$ and $M$ is such a rectangle (or $M = \mathbb{T}^m$), the conclusions of Theorem \ref{t:mainintro} and of its extension to the cornered case, refine the content of Theorem \ref{t:paintro}-(i) (since, for $m\leq 5$, one has that $\frac{m+1}{3} \leq 2$).   
     \item The additional condition appearing at Point (iii), ensuring that $V(X) \in \mathbb{D}^{1,2}$ when $m=3$, is an artifact of our use of Kac-Rice formulae for checking the integrability of norms of stochastic derivatives --- see the proof of Theorem \ref{thm:D12} below. It seems plausible that such a restriction can eventually be lifted, and that the conclusion holds in full generality.  
      \item As already observed, since the Gaussian field $f$ considered in Theorem \ref{t:paintro} is stationary and has positive variance, one can automatically conclude that the random variable $V(f)$ is not $\mathbb{P}$-a.s. constant. As discussed in \cite[Section 3]{PolyAngst}, such a result is a consequence of Bochner's theorem which, in this case, allows one to directly characterize the topological support of $f$ in terms of the topological support of its spectral measure (see \cite[Lemma 2]{PolyAngst}). At the level of generality of Theorem \ref{t:mainintro}-(i), no equivalent of Bochner's theorem is available and one has to explicitly assume that $V(X)$ is non-constant. We observe that the non-constancy of $V(X)$ is implied by the following property: {\it the topological support of $X$ contains two functions $g_1,g_2 \in \mathcal{C}^2(M)$, such that $0$ is a regular value of $g_1$ and $g_2$, and $V(g_1) \neq V(g_2)$.}  
Moreover, notice that if the zero set of a centered Gaussian field is almost surely empty, then it must be of the form $\gamma f$, for some $\gamma\sim \m N(0,1)$ and $f\in \mC^2(M)$. Under the hypotheses of \cref{t:mainintro}, such situation is excluded, so that there exist always at least one function $g_2$ in the topological support of $X$ such that $V(g_2)>0$.
Consequently, if the topological support of $X$ contains some $g_1\in \mC^2(M)$ such that $g_1^{-1}(0)=\emptyset$, then $V(X)$ is non-constant. 
       \item Consider the case $m=2$, and denote by $Y$ the random variable counting the number of connected components of the set $X^{-1}(0)$. Then, if $Y$ is not degenerate (that is, the distribution of $Y$ charges at least two integers with positive probability), one has that there is no curve $Z\subset M$ such that $X^{-1}(0)$ is isotopic to $Z$ (see \cite[p. 178]{Hirsch} or \cref{def:c1iso}) with probability one. 
     \item As discussed e.g. in \cite{PolyAngst}, proving that a certain set of random variables is included in some Malliavin-Sobolev class $\mathbb{D}^{1,\eta}$, opens the way to many further investigations, such as: probabilistic representations of densities (see \cite[Proposition 2.1.1]{nualartbook} or \cite[Theorem 3.1]{nourdinviens}), concentration inequalities (see \cite[Theorems 4.1 and 4.3]{nourdinviens} or \cite[Proposition 2.1.2]{nualartbook}), non-singularity and smoothness of joint laws \cite[Theorems 2.1.1 and 2.1.4] {nualartbook}, upper and lower bounds on variances \cite[Section 2.11]{nourdinpeccatibook}, Malliavin-Stein bounds \cite[Theorem 5.1.3]{nourdinpeccatibook}, second-order Poincar\'e estimates \cite[Theorem 5.3.3]{nourdinpeccatibook}, and so on. While a full analysis of these developments is outside the scope of the present paper and will be undertaken elsewhere, in Section \ref{ss:densityintro}, to motivate the reader, we will illustrate how the Malliavin-Sobolev regularity of nodal volumes can be used to deduce an explicit representation for the density $\pi$ appearing in \eqref{e:magnificentlaw}.  
    \end{enumerate}
    
\end{remark}

\subsection{A short literature review} Local geometric quantities associated with nodal sets of Gaussian (and non-Gaussian) smooth random fields have been recently the object of an intense study --- often in connection with outstanding open problems in differential geometry, like e.g. \emph{Yau's conjecture} (see e.g. \cite{logunovsurvey, yauconjecture}) or \emph{Berry's random wave conjecture} (see e.g. \cite{Abert2018, Berry1977, Canzani2020, ingremeauBerry}). The recent survey \cite{Wigman2022} contains a detailed overview of the literature, to which we refer the reader for further references and motivation. 

The following list displays a sample of contributions that are directly relevant to the present work. 

\begin{itemize}
\item As already mentioned, reference \cite{PolyAngst} contains the first study of the absolute continuity of the law of nodal volumes (for stationary fields on rectangles or on flat tori) by using Malliavin calculus. Some of the pivotal formulae in \cite{PolyAngst} have been subsequently extended by B. Jubin to a Riemannian setting in \cite{Jubin}, where they are used to study the continuity of nodal volumes with respect to appropriate function space topologies.

\smallskip

\item The main findings of \cite{PolyAngst} and of the present work have a {\it static} nature, that is, they provide information about the distribution (and smoothness) of the nodal volume associated with some fixed random field. On the other hand, most existing results about the law of nodal volumes (and other local geometric quantities) have an {\it asymptotic} character, that is, they take the form of {probabilistic limit theorems} (such as laws of large numbers, as well as central and non-central limit theorems) obtained either by letting some ``energy parameter'' diverge to infinity, or by suitably expanding the domain of definition of the field. It is a remarkable fact that many central and non-central results deduced in this way make use of \emph{Wiener chaos expansions}, which are in turn one of the fundamental building blocks of Malliavin-Sobolev spaces (see e.g. \cite[Section 1.1]{nualartbook} and \cite[Section 2.7]{nourdinpeccatibook}). The reader is referred to the survey \cite{rossisurvey}, as well as to \cite{benatarMW, kratzleon, MRossiWigman2020, MW1, MW2, NourdinPeccatiRossi2019, notarnicolaAHL, PV}, for a collection of representative references -- connected in particular to Berry's seminal work \cite{Berry2002a} on {cancellation phenomena}; see \cite{nourdinpeccatibook} for a detailed analysis of the probabilistic techniques behind these results.

\smallskip

\item In Section \ref{s:L2} and Section \ref{s:prelims} we wil use \emph{Kac-Rice formulae} in order to study, respectively, the square-integrability of random nodal volumes, and the integrability of stochastic derivatives with respect of appropriate Sobolev-type norms. While the content of Section \ref{s:L2} makes direct use of the recent findings from \cite{gasstec2023} (see also \cite{letendre2023multijet}), the results of Section \ref{s:prelims} are technically more demanding, and require a careful study of the two-points Kac-Rice density associated to the stochastic derivative. 
As discussed at length in \cite{Wigman2022}, Kac-Rice formulae have played a fundamental role in the proof of first- and second-order asymptotic results for nodal volumes of Gaussian random fields, and in particular for establishing tight upper and lower bounds on variances --- see e.g. \cite{Berry2002a, Canzani2020, Wig2013arith, Wigman_2010} for some outstanding examples. The reader is referred to the monographs \cite{AdlerTaylor, AzaisWscheborbook, berzin2022kacrice} (see also \cite{KRStec}) for a thorough introduction to this topic.
\end{itemize}

\subsection{Standard assumptions} 
To simplify the discussion, we notice that most results concerning Gaussian fields contained in the present paper (see, in particular, \cref{t:mainintro}) are stated under the following assumption.
\begin{assumption}\label{ass:1}
The set $M$ is a compact Riemannian manifold of dimension $m\geq 2$, possibly with boundary, and of class $\mC^2$. Also, $X = \{X(p) : p\in M\}$ is a centered Gaussian field defined on $M$, which is $\mathbb{P}$-a.s. of class $\mathcal{C}^2(M)$. Moreover, for all $p\in M$ and all non zero $v\in T_pM$, the two-dimensional Gaussian vector $(X(p), d_pX(v) ) $ admits a density.
\end{assumption}

In the context of Assumption \ref{ass:1}, the boundary of $M$ is denoted by $\partial M$, whereas the internal manifold is written $\inter{M}$; in particular, $M=\de M \sqcup \inter{M}$.

\subsection{First consequences and representative examples}

\subsubsection{Density formulae }\label{ss:densityintro} Theorem \ref{t:mainintro} provides sufficient conditions ensuring that random nodal volumes belong to some space $\mathbb{D}^{1,\eta}$. The next statement, which is a consequence of Theorem \ref{thm:sing}, shows that, when $\eta=2$, one can infer an explicit form for the density $\pi(x)$ appearing in \eqref{e:magnificentlaw}. In what follows, we write $\mathcal{H}_X$ to indicate the Cameron-Martin space associated with a Gaussian process $X$, and denote by ${\rm D}_{\mathscr{M}}F$ the Malliavin derivative of some $F\in \mathbb{D}^{1,\eta}$ (which is a random element with values in $\mathcal{H}_X^*\simeq \mathcal{H}_X $); see Section \ref{sec:ACDB} for definitions. We also use the symbol $L^{-1}$ to indicate the  \textit{pseudo-inverse of the generator of the Ornstein-Uhlenbeck semigroup}, as defined e.g. in \cite[Definition 2.8.10]{nourdinpeccatibook}.

\begin{theorem}\label{t:introdensitivan} Let Assumption \ref{ass:1} prevail and assume in addition that the topological support of the law of $X$ is the full space $\mathcal{C}^2(M)$ and that $V(X) \in \mathbb{D}^{1,2}$. Then, writing $c = \mathbb{E}[V(X)]$ and $\bar{V} = V(X)-c$, one has that the density $\pi$ appearing in \eqref{e:magnificentlaw} has the form 
    \begin{equation}\label{e:densitivan}
    \pi(x+c) = \frac{\mathbb{E}|\bar{V}|}{2 g(x)} \exp\left\{ -\int_0^x \frac{y\,dy}{g(y)}  \right\}, 
\end{equation} 
for $dx$-a.e. $x\in [-c,+\infty)$, where the function $g$ has the following properties:
\begin{enumerate}[\rm (a)]

\item $g(x) > 0$, for $dx$-almost every $x\in (-c,\infty)$;

\item on the interval $(-c,\infty)$, $g$ is a version of the mapping 
\begin{equation}\label{e:npkernel}
x\mapsto \mathbb{E}\big[ \langle {\rm D}_{\mathscr{M}}V, - {\rm D}_{\mathscr{M}}L^{-1}V\rangle_{\mathcal{H}_X} \, |\, \bar{V} = x  \big], 
\end{equation}
and is uniquely defined up to sets of Lebesgue measure zero.
\item $\mathbb{E}[ g(\bar{V})]<\infty$.

\end{enumerate}
\end{theorem}

As made clear by the proof of Theorem \ref{thm:sing}, formula \eqref{e:densitivan} should be regarded as a variation of the celebrated \textit{ Nourdin-Viens formula} from Malliavin calculus (see \cite{nourdinviens} and \cite[Section 10]{nourdinpeccatibook}). We observe that the conditional expectation appearing in \eqref{e:npkernel} plays a pivotal role in the so-called \textit{Malliavin-Stein approach} to probabilistic approximations, as described e.g. in \cite[Chapter 5]{nourdinpeccatibook}; in particular, the Malliavin-Stein theory implies that, if such a conditional expectation is close to a constant, then the law of $\bar{V}$ is close to a centered Gaussian distribution (see e.g. \cite[Theorem 5.1.3]{nourdinpeccatibook} for a representative statement).  By virtue of Point (b) in Theorem \ref{t:introdensitivan}, one could in principle use our new explicit representation for the differential of nodal volumes, as stated in \eqref{eq:firstvar} below, to deduce bounds on the function $g$ appearing in \eqref{e:densitivan}, which would provide estimates on the density $\pi$ --- see e.g. \cite[Corollary 3.5]{nourdinviens}. We leave this important issue open for future research. 


%
%

\subsubsection{Some standard models} We will now demonstrate that one can directly apply Theorem \ref{t:mainintro} to several classical models of random fields on manifolds.
\begin{enumerate}[\rm (1)]

 \item For $m\geq 2$, let $M$ be either a closed rectangle of $\R^m$ or the torus $\mathbb{T}^m$, and let $f$ be a centered stationary non-zero Gaussian field of class $\mathcal{C}^2(M)$ such that $\nabla f(0)$ has a density. Then, as already recalled, \cite[Theorem 1]{PolyAngst} (whose conclusions for $m\geq 3$ is contained in Theorem \ref{t:paintro} as a special case) implies that $V(f)$ is a.s. not equal to a constant. As a consequence, the conclusion of Point (i) and Point (ii) of Theorem \ref{t:mainintro} (in the case $X=f$) directly apply. One can also check that, for a field $f$ as above, the non-degeneracy conditions stated at Point (iii) of Theorem \ref{t:mainintro} are verified, yielding that $V(f)\in \mathbb{D}^{1,2}$ also for $m=3$. Outstanding examples of Gaussian stationary random fields verifying the assumptions above are \textit{arithmetic random waves} on tori of any dimension $m\geq 2$ (see e.g. \cite{apv, Wig2013arith, Marinucci2016, ORW2008, rozensheinIMRN, Rudnick2008}), or (an appropriate restriction of) \textit{Berry's random wave model} on $\R^m$ (see e.g. \cite{Berry1977, Berry2002a, Canzani2020, Dalmao2021, NPV, NourdinPeccatiRossi2019, PV}), or \emph{Bargmann-Fock field} on $\R^m$ (see \cite{2022Beliaev_fluctuations,DuminilVanneauville_2023,2017_Beffara_Gayet,Dalmao2021,rivera:hal-03320870,MTTRPS_AIF_stec,stec2022GeometrySpin}). In the latter case, in particular, the field has full support, so that Point (i) of \cref{t:mainintro} applies in its stronger form and Point (iv) establishes that the nodal length of the (restriction to $M$ of) the two-dimensional Bargmann-Fock field is not in $\mathbb{D}^{1,2}$.
 
\item Consider a stationary random field $f$ on $\R^2$ verifying the non-degeneracy assumptions stated at Point (1), and assume that $f$ has the \textit{ moving average representation} $f = q\star W$, where $q(x)$ is a square-integrable hermitian kernel and $W$ is a homogeneous white noise on the plane. Then, using \cite[Corollary 1.4]{gasstec2023} (see also \cite[Theorems 1.2 and 1.3]{beliaev2022central}) yields the following conclusion: if $q$ and its derivatives up to the order 2 decay at least as fast as $|x|^{-\beta}$ for some $\beta>6$, then the number of connected components of the nodal set $f^{-1}(0)\cap M$ has a strictly positive variance for every sufficiently large rectangle $M\subset \R^2$. Since the non-degeneracy assumptions stated in Theorem \ref{t:mainintro}-(iv) are trivially satisfied by $f$, one can directly apply Remark \ref{r:marketing}-(f) and deduce that $V(X) \in \mathbb{D}^{1,1}$ and $V(X) \notin \mathbb{D}^{1,2}$ when $M$ is large enough.

\item Let $M = \mathbb{S}^2$ and let $X$ be a \emph{random spherical harmonic} with eigenvalue $-l(l+1)$, $l\geq 1$, as defined e.g. in \cite{CammarotaM2018,  marinucci_peccati_2011, MRossiWigman2020, NazarovSodin2009, stec2022GeometrySpin, Wigman_2010}. Then, combining Theorem \ref{t:mainintro} and Remark \ref{r:marketing}-(f) with the results of \cite{Wigman_2010} (variance of the nodal length) and \cite{NazarovSodin2009} (variance of the number of nodal components), one infers that, for $l$ large enough, the law of $V(X)$ is not singular with respect to the Lebesgue measure, $V(X) \in \mathbb{D}^{1,1}$ and $V(X) \notin \mathbb{D}^{1,2}$ (notice that the non-degeneracy properties necessary to apply Theorem \ref{t:mainintro} are easily verified also in this case).

\item Let $P_d$ be the real homogeneous Kostlan polynomial of degree $d$ in $m+1$ variables, that is the Gaussian random field on $\R^{m+1}$ with covariance function $\E\kop  P_d(x) P_d(y)\pok=(x^Ty)^d$. This field is central in random algebraic geometry --- see \cite{GaWe1,GaWe2,GaWe3,LETENDRE2016,kostlan:93,FyLeLu,LeLuStat16,buerg:07,RootsKost_ancona_letendre,2020_Ancona_JEMS_Exprar,2023Ancona_completeint,2022BreKeneLer,MTTRPS_AIF_stec}, the survey paper \cite{anantharaman:hal-03527166} and the reference therein --- in that the set $Z_{m,d}$ defined by the equation $P_d=0$ in the projective space $\R \mathsf P^m$ is a particularly natural model of a random algebraic hypersurface of degree $d$, supported on the set of all smooth ones.
Equivalently, one can study the subset of the sphere $\mathbb S^m$ given by the same equation, which is the nodal set of the restriction of $P_d$ to $m$-sphere $\psi_d:=P_d|_{\mathbb S^m}$. This field can be easily seen to satisfy \cref{ass:1}, for all $d\ge 1$ and $m\ge 2$, so that \cref{t:mainintro} applies to the random variable $m$-volume $\vol^{m-1}(Z_{m,d})=\frac12 V(\psi_d)$ of $Z_d$, where $\R \mathsf P^n$ is endowed with the its standard metric. Moreover, since for all $d\ge 2$ there are at least two isotopy classes of smooth real algebraic hypersurfaces, it follows from Point (iv) of \cref{t:mainintro} that the length of $Z_{2,d}$ is in $\mathbb{D}^{1,1}$ but not in $\mathbb{D}^{1,2}$.

\item Let $(M,g)$ be a compact, smooth Riemannian manifold of dimension $m\geq 2$, let $X = \varphi_\lambda$ be a \textit{monochromatic random wave} of parameter $\lambda >0$ (as defined e.g. in \cite{Canzani2020, zelditchmonochromatic}), and assume that $x\in M$ is a point of \textit{ isotropic scaling} (see \cite[Section 1.2]{Canzani2020}). Then, according e.g. to \cite[Proof of Theorem 2]{Canzani2020}, once one chooses coordinates around $x$ ensuring that $g_x = {\rm Id}$ and letting $\lambda \to \infty$, one has that the \textit{ monochromatic pullback random wave} $\{ \varphi^x_{\lambda}(u) : u\in T_xM\}$ (see \cite[formula (8)]{Canzani2020} for definitions) converges in distribution to the standard Berry's random wave model on $\R^m$ (as a random element with values in the space $\mathcal{C}^k$, for every fixed $k\geq 1$). One can therefore suitably adapt the computations contained in \cite[Sections 3.3 and 3.4]{DNPR} to deduce from Point (1) above that, if $B\subset T_x M$ is a large enough ball contained in the tangent space $T_x M$, then $V(\varphi_\lambda^x) = \mathcal{H}^{m-1}((\varphi_\lambda^x)^{-1}(0)\cap B)$ has a strictly positive variance for $\lambda$ large enough, and the conclusions of Points (i) and (ii) of Theorem \ref{t:mainintro} apply.

\item The results outlined at Point (4) characterize the smoothness of nodal volumes associated with the monochromatic random wave $\varphi_\lambda$ restricted to a ball containing a point $x$ of isotropic scaling. We stress, however, that the results of the present paper could be used to investigate the regularity of the law of the total nodal volume $V(\varphi_\lambda) = \mathcal{H}^{m-1}(\varphi_\lambda^{-1}(0))$, possibly under some appropriate additional assumptions on the geometry of $M$. 
In principle, such an analysis should exploit the fact that 
for all $p\in M$, for all $v\in T_pM$ and $\lambda$ large enough, the Gaussian vector $(\varphi_\lambda (p), d_p \varphi_\lambda(v) ) $ admits a density (which is a necessary condition for Theorem \ref{t:mainintro} to be applicable).
In particular, this follows from \cite[Proposition 3.2]{zelditchmonochromatic} in the cases of a compact Riemannian manifold $M$ that is either \emph{Zoll} (for instance, the sphere) or \emph{aperiodic}, see \cite[Section 1.1]{zelditchmonochromatic} for definitions.
A proper investigation of this point will be undertaken elsewhere. 

\end{enumerate}

\subsubsection{Sharpness of Theorem \ref{t:mainintro}:  the example of linear fields} Let $d\in \N$ and $\gamma^T=(\gamma_0,\dots,\gamma_d)$ be a normal Gaussian vector in $\R^{d+1}$. Define a Gaussian field $\hat X$ of class $\mC^\infty(\R^{d+1})$ by the expression
\be 
\hat{X}(p)=\gamma^T p=\sum_{i=0}^d\gamma_i p_i,
\ee
for every $p\in \R^{d+1}$. Then, for every $\mC^2$ compact embedded submanifold with boundary $M\subset \R^{d+1}\smallsetminus \{0\}$, the restriction $X:=\hat X|_{M}$ satisfies the general assumptions of \cref{t:mainintro} (but not necessarily the additional requirements of (iii) and (iv)). The nodal set is thus the intersection $Z=M\cap L$ of $M$ with the random hyperplane $L=\gamma^\perp$, that is uniformly distributed on the (dual) projective space 
$\mathsf{P}^d$. Varying $M$ in the previous construction leads to several examples, allowing one to probe the assumptions of Theorem \ref{t:mainintro}.
\begin{enumerate} 
\item Let $M=\mathbb S^m$ and $d=m\ge 2$. Then, the nodal set of $X$ is always a sphere of dimension $m-1$, so that the nodal volume $V(X)$ is a constant random variable, which is trvially in $\mathbb{D}^{1,\eta}$ for all $\eta\in [1,+\infty]$. 
Notice that for every $p\in \mathbb{S}^m$ and $v\in T_p\mathbb{S}^m$, the random vector $(X(p),d_pX(v),\H_p(v,v))$ is degenerate: when $m=3$, this indicates that the additional non-degeneracy assumption stated at point (iii) of \cref{t:mainintro}, is not necessary for the conclusion to hold --- see also point (d) of \cref{r:marketing}.
\item It is important to notice that, when $m=2$, the example at Point (1) above is not in contradiction with point (iv) of \cref{t:mainintro}, since in that case all realizations of the zero set are $\mC^1$ isotopic to a fixed sphere. If we replace the sphere $\mathbb{S}^m$ with an ellipsoid $\m S$ having distinct semi-axis lengths, we can see point (iv) of \cref{t:mainintro} in play again, in that the isotopy class of the nodal set $Z=\m S\cap L$ (an ellipsoid of lower dimension) is constant, but the nodal volume $V(X)$ is not. Nevertheless, one can check that in this case $V(X)\in \mathbb D^{1,2}$ and its law is absolutely continuous.
\item For $m=2$, let us consider once again Point (iv) of \cref{t:mainintro}. We want to show that, without the non-correlation assumption stated therein, one can build examples of submanifolds $M$ such that $X=\hat X|_{M}$ has an a.s. constant nodal volume $V(X)$ (which is therefore in $\mathbb{D}^{1,\eta}$ for all $\eta\geq 1$), irrespective of the constancy of the isotopy type of its nodal set; in the case of a constant isotopy type, this shows in particular that the non-correlation assumption in \cref{t:mainintro}-(iv) is not necessary for its conclusion to hold.
To see this, let $\f\colon \mathbb{S}^1\to\,\, ( -1,1 )$ be any smooth function such that $\f(-\theta)=-\f(\theta)$. Consider the two-sphere $\mathbb{S}^2$ in $\R^3=\R\times \R^2$ and define a manifold (with boundary) $M:=\kop (t,re^{i\theta})\in \mathbb{S}^2: t\le \f(\theta)\pok$, where the second entry of the vector $(t, r\theta)$ is expressed in polar coordinates, see \cref{fig:curly}. The field $X=\hat X|_M$ does not satisfy the correlation assumption stated at Point (iv) of Theorem \ref{t:mainintro}, because both $M$ and $X$ are invariant by the map $p\mapsto -p$, in such a way that, at antipodal boundary points, the field and its derivatives are fully correlated. Also, for any choice of $\f$ one has that the nodal set $Z=M\cap L$ is an anti-symmetric subset of the circle $\mathbb{S}^2\cap L$, namely, $p\in Z\iff -p \notin Z$, for all $p\notin \de Z$, which implies that $\mathcal{H}^{1}(Z) = V(X) = \pi$, and therefore $V(X)\in \mathbb D^{1,\eta}$, $\eta\geq 1$. Now, if $\f=0$, then $M$ is a half sphere, $Z=M\cap L$ is a half equator, and $Z$ has constant isotopy type in $M$. On the other hand, if $\f$ has $2+4n$ zeroes\footnote{The number of zeros of an anti-symmetric function on $\mathbb{S}^1$, if zero is a regular value, must be of the form $2+4n$, for some $n\in \N$.}, then $Z$ has $1+2n$ connected components if $L$ is horizontal (see \cref{fig:curly}-(A)), and one connected component if $L$ is vertical (see \cref{fig:curly}-(B)). Since the law of the Gaussian vector $\gamma$ has full support in $\R^{d+1}$, one immediately deduces that the probability that $L$ has $1+ 2n$ connected components and the probability that $L$ has exactly one connected component are both strictly positive. As a consequence, the isotopy type of $Z$ is not a.s. constant, despite $V(X)$ being a (trivial) element of $\mathbb{D}^{1,2}$. See also point (f) of \cref{r:cmspace}.
\end{enumerate}
\begin{figure}
\centering
\begin{subfigure}{.45\linewidth}
  \centering
  \includegraphics[width=.9\linewidth]{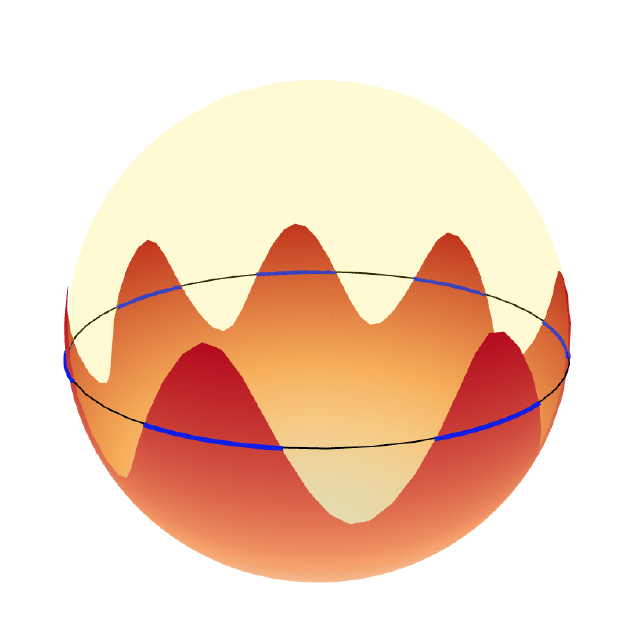}
  \vspace{15pt}
 \caption{If $L$ is \emph{horizontal}, then $Z=M\cap L$ is the disjoint union of $7$ segments of circles.}
  \label{fig:curly2}
\end{subfigure}
\hspace{.05\linewidth}
\begin{subfigure}{.45\linewidth}
  \centering
  \includegraphics[width=0.98\linewidth]{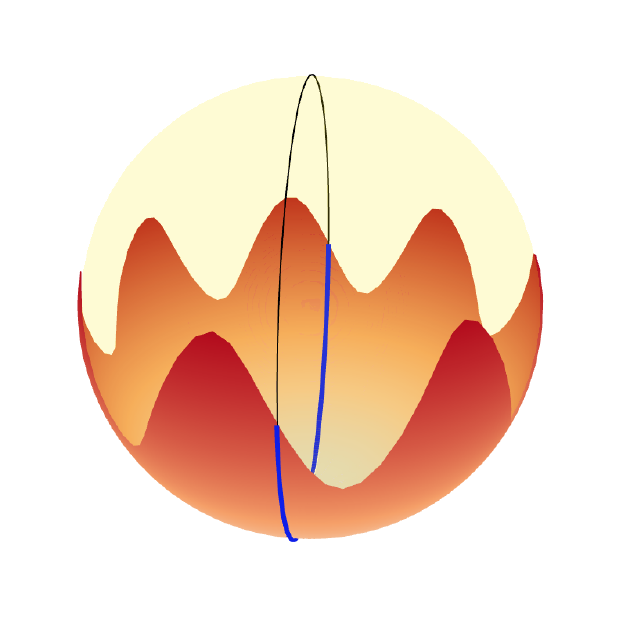}
  \caption{If $L$ is \emph{vertical}, then $Z=M\cap L$ is a semi-circle.}
  \label{fig:curly1}
\end{subfigure}%
\caption{The pictures shows two nodal lines of the Gaussian field $X=\hat{
X}|_M$ on the manifold $M:=\kop (t,re^{i\theta})\in \mathbb{S}^2: t\le \frac{2}{7}\sin(7\theta)\pok$, where the second entry of the vector $(t, r\theta)$ is expressed in polar coordinates. They both are of the form $Z=M\cap L$, with $L=X^{-1}(0)$ being a plane in $\R^2$}
\label{fig:curly}
\end{figure}

The next section contains a discussion of the main methodological innovations established in our work and serves as a plan for the rest of the paper. 

\subsection{Main technical contributions and structure of the paper}\label{ss:motivintro}

\begin{enumerate}[(i)]
    \item {\it Criteria for absolute continuity on Banach spaces}. Section 3 contains general criteria --- stated in Theorem \ref{t:banach} and Corollary \ref{cor:abs} --- ensuring that the distribution of a random variable of the form $V = V(X)$, where $X$ is a Banach space-valued Gaussian random element, is not singular with respect to the Lebesgue measure. Our conditions require, in particular, that the Fr\'echet differential $d_f V$ is not zero for some $f$ in the topological support of the law of $X$. As discussed below, our strategy of proof is close in spirit to the arguments rehearsed in \cite{enchevstroock}.  

    
    \item {\it Explicit formulae for differentials of nodal volumes}. The main achievement of Section \ref{sec:nodaldiff} is the following new explicit formula for the differential of nodal volumes. Let $(M,g)$ be a $m$-dimensional $\mathcal{C}^2$ Riemannian manifold (possibly with boundary) and let $\mathcal{U}$ denote the open subset of $\mathcal{C}^2(M)$ composed of those $f$ having zero as a regular value; then, setting $V(f) = \mathcal{H}^{m-1} (f^{-1}(0))$, one has that $V\in \mathcal{C}^1(\mathcal{U})$ and, for all $f\in \m{U}$, 
    \be\label{eq:firstvar} 
\langle d_fV,h\rangle=-\int_{Z}h\cdot \frac{\tDelta f}{\|df\|^{2}} dZ+\int_{\de Z}h\cdot \frac{ g(\n,\nu) }{\|d(f|_{\de M})\|} d\de Z,
\ee
where $Z:=f^{-1}(0)$, $\nu:=\|\g f \|^{-1}\g f$, $\tDelta f:=\Delta f-\H f(\nu,\nu)$ and the scalar product on the right-hand side corresponds to the duality pairing notation. Formula \eqref{eq:firstvar} (whose proof is based on classical variational formulae discussed e.g. in \cite{li2012geometric}) is our main tool for directly applying the absolute continuity criteria discussed above, as well as for studying the Malliavin-Sobolev regularity of nodal volumes. As discussed in Remark \ref{r:diffheur}, when specialized to a Euclidean setting (in which $f$ is e.g. a smooth stationary Gaussian field on some rectangle $M$, having zero as a regular value with probability one), relation \eqref{eq:firstvar} can be formally deduced by differentiating the classical relation
\begin{equation}\label{e:volumecondirac}
\mathcal{H}^{m-1}(f^{-1}(0)) = \int_M \delta_0(f(x)) \, \|\nabla\, f(x)\|_{\mathbb{R}^m}\, dx,
\end{equation}
where the right-hand side of the previous equation has to be regarded as the appropriate limit of integrals where $\delta_0$ is replaced by a smooth kernel. See e.g. \cite{DNPR, marinucci2023laguerre, MRossiWigman2020, NourdinPeccatiRossi2019} for a sample of references directly using (or referring to) \eqref{e:volumecondirac}. 


    \item {\it Transversality of random curves in function spaces}. In Section \ref{sec:nondegeneracyconditions}, we focus on the characterization of {\it random segments} of the form $[X, X+h]:= \{ X+ th : t\in [0,1] \}$ (as special cases of more general $\mathcal{C}^2(M)$-valued curves $\{f_t\}$), where $X$ is a smooth Gaussian field satisfying Assumption \ref{ass:1} and $h$ is a fixed element of the associated Cameron-Martin space. Our main achievement --- see Theorem \ref{thm:transcurve} --- is the proof that a.s. $[X, X+h]$ only intersects \textit{transversely} the exceptional subset of $\mathcal{C}^2(M)$ composed of functions having zero as a critical value, in such a way that the intersections are finite in number and correspond to mappings for which zero is a Morse critical value. Such a result is the key element to verify the property of ray absolute continuity that is necessary to establish the Malliavin-Sobolev regularity of nodal volumes. See Section \ref{def:pseudotransv} for the appropriate notion of transversality, as well as Theorem \ref{thm:bestia} for some special results needed in the two-dimensional case. 


\item {\it Regularity of nodal volumes along transverse curves via a transfer principle}. One of the crucial consequences of the findings of Section \ref{sec:nondegeneracyconditions} (see, in particular, Corollary \ref{cor:TransvToMorse}), is that the study of the local regularity of mappings of the type $t\mapsto V(f_t)$, where $V$ is the nodal volume and $\{f_t\}$ is a transverse curve, can be reduced to the local analysis of functions with the form $t\mapsto V(T-t)$, where $T$ is a Morse mapping on an adequate ancillary manifold with the same dimension. Our analysis (see e.g. Lemmas \ref{lem:upasyMorse} and \ref{lem:notd12}) reveals that the local regularity of these mappings strongly depends on the dimension on the manifold $M$. The results of Section \ref{sec:nondegeneracyconditions} and Section \ref{sec:levelMorse} are brought together in Section \ref{sec:transverse}, which also contains the proof of the non-singularity of the law of nodal volumes under \cref{ass:1} and an additional non-constancy requirement. 

    
    \item {\it Malliavin-Sobolev regularity via Kac-Rice formulae.} As already discussed, Sections \ref{s:prelims} and \ref{s:mallsob} establish the necessary integrability conditions for nodal volumes, and for the associated norms of weak derivatives, by repeatedly using Kac-Rice formulae -- as applied in particular to \eqref{eq:firstvar}. Such an approach has to be contrasted with the techniques developed in \cite{PolyAngst}, where such a task is accomplished through the study of the local regularity of the integrand appearing in an (already evoked) \textit{exact Rice formula} for nodal volumes.  

    
    \item {\it Structure of the law of nodal volumes under a full support assumption}. As anticipated, Section \ref{s:singular} contains a full proof of Equation \eqref{e:magnificentlaw}. In Remark \ref{r:bestremarkever} it is argued that the main argument in the proof addresses a conjecture formulated in \cite[Section 4]{PolyAngst}.

\item {\it Constraints on the index of simultaneous critical points}. 
For a Gaussian field $X$ satisfying \cref{ass:1}, the event that a random segment $[X,X+h]$ contains functions $f$ with multiple critical zeros might have positive probability, although \cref{ass:1} ensures that all those critical points are Morse. For instance, if $X$ is a periodic function, on the torus $\mathbb T^m$ or on a domain $D\subset \R^m$ with small enough period, we prove (\cref{thm:bestia}) that almost surely, if such event occours, there are precise constraints on the Morse index of the critical points of $f$. It is surprising that this is true under the sole \cref{ass:1}, so in particular, without making any assumption on the correlation of the random variables $X(p)$ and $X(q)$ at distinct points $p\neq q\in M$.
\item {\it Negligible events in Gaussian spaces}.
We signal that \cref{lem:blind}, despite its very technical appearence, provides a very efficient strategy to prove that certain events have zero Gaussian probability, especially those arising from a differential condition, and involving random segments. Indeed, it works when more standard transversality arguments --- like those on which \cref{lem:Ptransv}, \cref{lem:bulinskaya} (Bulinskaya) and  \cref{thm:transcurve} rely on --- are not available. It is a key ingredient in the proof of \cref{thm:bestia} and in the proof of \cref{thm:sing}, to prove the validity of \cref{e:magnificentlaw} in point (i) of \cref{t:mainintro}.
\end{enumerate}

The forthcoming Section \ref{sec:notations} contains a list of preliminary definitions and notational conventions, whereas the formal definition of the Malliavin-Sobolev space $\mathbb{D}^{1,p}$ is deferred to Section \ref{ss:malliavin}.
\subsection{Acknowledgments}
The authors are thankful to J. Angst and G. Poly for useful discussions.
This research is supported by the Luxembourg National Research Fund (Grant: 021/16236290/HDSA).
\section{Notations}\label{sec:notations}

The following list contains some recurring conventions adopted in our work.
\begin{enumerate}[(i)]


\item A \emph{random element} (see \cite{Billingsley}) of the topological space $T$ (or \emph{with values} in $T$) is a measurable mapping $X\colon \Omega\to T$, defined on $\Prob$. In this case, one writes
    \be\label{eq:randin}
    X\randin T
    \ee 
    and denote by $[X]=\PP X^{-1}$ the (push-forward) Borel probability measure on $T$ induced by $X$. We will use the notation
\be 
\PP\{X\in U\}:=
\PP X^{-1}(U)
\ee 
to indicate the probability that $X\in U$, for some Borel measurable subset $U\subset T$, and write (as usual)
\be 
\E\{f(X)\}:=\int_{T}f(t)d[X](t),
\ee
to denote the expectation of the random variable $f(X)$, where $f\colon T\to \R^k$ is a measurable mapping such that the above integral is well-defined.
We will sometimes write that $X$ is a \emph{random variable}, a \emph{random vector} or a \emph{random mapping}, respectively, when $T$ is the real line, a vector space, or a space of functions $\mC^r(M,N)$, 
respectively. Finally, if $T_1\subset T$ is a measurable subset such that $\PP\{X\in T_1\}=1$, then we will also write $X\randin T_1,$ by a slight abuse of notation.
\item The space of $\mC^r$ functions between two manifolds $M$ and $N$ is denoted by $\mC^r(M,N)$. We simply write $\mC^r(M)=\mC^r(M,\R)$ when $N=\R$.
If $E\to M$ is a vector bundle, we denote the space of its $\mC^r$ sections by $\Gamma^r(E)$ (see e.g. \cite[Chapter 10]{leesmooth}). Both spaces $\mC^r(M,N)$ and $\Gamma^r(E)$ are regarded as topological spaces endowed with the weak Whitney's topology (see \cite[p. 34]{Hirsch}). In particular, the space of $\mC^r$ vector fields on $M$ will be denoted by $\Vect^r(M):=\Gamma^r(TM)$.
\item The sentence: ``$X$ has the property $\mathcal{P}$ almost surely'' (abbreviated ``a.s.'') means that the set  $S=\{t\in T : t \text{ has the property }\mathcal{P}\}$ contains a Borel set of $[X]$-measure $1$. It follows, in particular, that the set $S$ is $[X]$-measurable, i.e. it belongs to the $\sigma$-algebra obtained from the completion of the measure space $(T,\mathcal{B}(T),[X])$.
 \item We write $\#(S)$ for the cardinality of the set $S$.
\item We use the symbol $A\transv B$ to say that objects $A$ and $B$ are in transverse position, in the usual sense of differential topology (as in \cite[p. 74]{Hirsch} and \cref{def:transv}).

\end{enumerate}
\section{Absolute Continuity and Differentiability on Banach Spaces}\label{sec:ACDB}

{In this section, we use some basic elements of the theory Gaussian measures on Banach spaces. The reader is referred to \cite[Chapters 2 and 3]{bogachev} for a comprehensive discussion; see also \cite[Section 4]{hairerSPDES} for a succinct presentation. From now on, the notational conventions from Section \ref{sec:notations} are adopted without further notice.}

\subsection{Some general statements} {Let $E$ be a Banach space, denote by $E^*$ its dual and write $\langle \lambda,g\rangle$, $\lambda \in E^*, \, g\in E$, to indicate the usual duality pairing}. We endow $E$ with a centered Gaussian probability measure $\mu$ (see \cite[Definition 2.1]{bogachev}), and we assume that $\mu = [X]$, that is,  $X\randin E$ is a random element whose push-forward measure on $E$ coincides with $\mu$ (see Section \ref{sec:notations}--(ii) and subsequent discussion). Accordingly, we will use the notation $\HX$ to indicate the {\it Cameron-Martin space} associated with $\mu$, as defined e.g. in \cite[p. 44]{bogachev}. We recall that $\HX$ is a Hilbert space with the properties that the inclusion $\HX\subset E$ is injective and continuous, and the random element $X$ such that $\mu=[X]$ can be chosen in such a way that 
\begin{equation} 
X=\sum_{n\in \N}\gamma_n h_n, \label{e:seriexp}
\end{equation}
for an arbitrary orthonormal Hilbert basis $(h_n)_n$ of $\HX$ and any family $(\gamma_n)_n$ of i.i.d. random variables $\gamma_n\sim \mathcal{N}(0,1)$, where the series converges in $E$, a.s.-$\mathbb{P}$. See e.g. \cite[Theorem 3.5.1]{bogachev}, as well as the forthcoming Remark \ref{r:cmspace}, for concrete examples.

\smallskip

In the sequel, we will assume that $\mu$ is {\it non-degenerate}, meaning that the topological support of $\mu$ is the whole space $E$, that is: for each open subset $ A \subset E$ one has that $\mu(A)>0$ (see e.g. \cite[Definition 3.6.2]{bogachev}). According to \cite[Theorem 3.6.1]{bogachev}, this property is equivalent to the fact that $\HX$ is dense in $E$. 

 \begin{remark}\label{r:fullsupport} {\rm Following e.g. \cite[Definition A.3.14]{bogachev}, we recall that the topological support of the measure $\mu$ --- henceforth written ${\rm supp}(\mu)$ ---  is defined to be the smallest closed subset of $E$ such that $\mu(\, {\rm supp}(\mu) \, ) = 1$. Since (again by virtue \cite[Theorem 3.6.1]{bogachev}) the set ${\rm supp}(\mu)$ coincides with the closure of $\mathcal{H}_X$ in $E$, one can bypass the non-degeneracy assumption on $\mu$ by considering instead the restriction of $\mu$ to the Banach (sub)space ${\rm supp}(\mu)$. 
}
\end{remark}

Now fix a Borel measurable mapping $V\colon E \to \R$, and assume that there exists an open subset $\mathcal{U}_0 \subset E$ such that $V$ is of class $\mC^1$ on $\mathcal{U}_0$, that is, there exists a continuous mapping $dV:\mathcal{U}_0\to E^*$ such that, as $g\to 0$ in $E$,
\begin{equation}\label{e:banachgrad}
V(f+g)=V(f)+\langle d_fV,g\rangle+o\tyu\|g\|\uyt,
\end{equation}
where we used the duality pairing notation introduced above. The following statement provides a straightforward criterion ensuring that there exists a truncation of the random variable $V(X)$ whose law is absolutely continuous.

\begin{theorem}\label{t:banach}
Let the above setting and assumptions prevail (in particular, $X$ is a $E$-valued random element with non-degenerate distribution $\mu$), and assume that $d_fV\neq0$ for some $f\in \m{U}_0$. Then, there exists an open neighborhood $\m O\subset\m{U}_0$ of $f$ such that the law of the random variable $V1_{\m{
O}} := V(X)1_{\m{
O}}(X)$ is absolutely continuous with respect to the Lebesgue measure.
\end{theorem}
\begin{proof}
By density and continuity, we can find $h_0\in \mathcal{U}_0$ and $h\in \HX$ such that $\langle d_{h_0}V,h\rangle \neq 0$. Without loss of generality, we can assume that $|h|=1$. The rest of the proof is partially based on a construction reminiscent of the arguments rehearsed in \cite[Section 2]{enchevstroock}. Define the closed Banach subspace $E_0:=\overline{\{h\}^\perp}$ (closure of $\{h\}^\perp$ in $E$). Then, using e.g. \cite[first Lemma of Section 2]{enchevstroock}, one sees that there is a splitting $E =E_0\oplus \R h$ such that the random element $X$ can be written as $X=Y+h\gamma$, where $Y\randin E_0$ is a Gaussian random element with values in $E_0$ (with Cameron-Martin space $\m{H}_Y=\{h\}^\perp$) and $\gamma\sim \m{N}(0,1)$, in such a way that $Y$ and $\gamma$ are independent. Let $\m{O}:=\m{O}_{0}\oplus (a,b)h$ be an open neighborhood of $h_0$ such that $0\notin \{\langle d_fV, h\rangle :f\in \m{O}\}$, where $\m{O}_0\subset E_0$ is open and $a,b\in \R$. For any $g\in \m{O}_0,$ then, the function $t\mapsto V(g+th)$ is a $\mC^1$ diffeomorphism between two intervals: $(a,b)\to (V(g+ah),V(g+bh))$, so that, if $Z\subset \R$ has zero Lebesgue measure, then 
$1_Z\tyu V(g+t h)\uyt=0,$ for a.e. $t\in (a,b).$
 As a consequence, 
\bega 
\tyu\tyu V1_{\m{O}}\uyt_*\mu\uyt\kop Z\pok
&=
\E\kop 1_{\m{O}_0}(Y)1_{(a,b)}(\gamma)1_Z\tyu V(Y+\gamma h)\uyt\pok
\\&=\E\kop 1_{\m{O}_0}(Y)\int_{a}^b 1_Z\tyu V(Y+t h)\uyt\rho(t)dt\pok=0.
\eega
\end{proof}

\begin{remark} To keep the notational complexity of the present paper within bounds, in the following we will only deal with Gaussian random elements with values in Banach spaces, so that Theorem \ref{t:banach} will be enough for our purposes. One should note, however, that the above-displayed arguments can be made to work in the case where $E$ is a non-complete normed vector space or non-normed Fr\'{e}chet space. 
\end{remark}

The following statement is a direct consequence of Theorem \ref{t:banach}. It is a general criterion allowing one to deduce that the law of a given function of the Gaussian element $X$ has an absolutely continuous component.

\begin{corollary}\label{cor:abs}
Under the assumptions of this section, consider a mapping $V:E\to \R$ that is differentiable and non-constant on some connected open subset of $E$. Then, the law of the random variable $V(X)$ and the Lebesgue measure are not mutually singular.
\end{corollary}

\begin{remark}\label{r:cmspace}{\rm Let $M$ be a $\mC^2$, compact $m$-dimensional manifold. Most of the Banach spaces encountered in the present work will have the form of closed subsets of the class $\m C^2(M)$, that we endow with the weak Whitney's topology --- see Section \ref{sec:notations}--(iii). As explained e.g. in \cite[p. 35]{Hirsch}, the set $\m C^2(M)$ (and therefore any of its closed subsets) can be equipped with the structure of a Banach space. Now consider a centered Gaussian measure $\mu$ on $\m C^2(M)$, and let $\mu = [X]$ for some random field $X $ with covariance $K_X$ (note that, necessarily, $K_X\in C^{2,2}(M\times M)$). According e.g. to the discussion contained in \cite[Section 2.4]{bogachev} or \cite[Section 4 and Appendix A]{dtgrf}, the Cameron-Martin space $\mathcal{H}_X$ admits the following standard characterization:
 \begin{enumerate}[(a)]
 \item Let $D$ be the linear space generated by the collection of Dirac masses $\{\delta_p : p\in M\}$ and define $\m H$ to be Hilbert space obtained as the closure of (an appropriate quotient of) $D$ with respect to the bilinear extension of the inner product $\langle \delta_p, \delta_q\rangle = K_X(p,q).$ Then, $\mathcal{H}_X$ is the subset of $\mathcal{C}^2(M)$ given by the mappings $\Phi_h: p\mapsto \Phi_h(p) := \langle h, \delta_p\rangle$, $h\in \m H$, with inner product $\langle \Phi_h, \Phi_g\rangle_{\mathcal{H}_X}= \langle h,g \rangle$.

 \item If $G\in \m C^2(M)^*$ has the form
 $$
 G(f) = \int_M f(p)\, \theta(dp) = \int_M \delta_p(f)\, \theta(dp), \quad f\in \m C^2(M),
 $$
for some finite signed Borel measure $\theta$, then $G\in \HX^*\simeq \HX$ and 
$$
\| G \|^2_{\HX} = \int_M \int_M \langle \delta_p, \delta_q \rangle \,\theta(dp)\theta(dq) = \int_M \int_M K_X(p,q) \, \theta(dp)\theta(dq) .
$$
A full characterization of $C^2(M)^*$ (not needed in our work) is provided in \cite[Theorem 44]{dtgrf}. 
\end{enumerate}
 }
\end{remark}

Several consequences of Corollary \ref{cor:abs} are discussed in Section \ref{sec:nons}.

\subsection{Malliavin-Sobolev spaces}\label{ss:malliavin}
 This is the right place for introducing the \emph{Malliavin-Sobolev} \emph{spaces} $\mathbb{D}^{1,p},$ for all $p\geq 1$. The next definition corresponds to the content of \cite[Definitions 5.2.3 and 5.2.4]{bogachev}, with one exception: at Point (i) below, we decided to state the property of ``ray absolute continuity'' in a form that is more suitable for the framework of our paper. For the sake of completeness, the equivalence between the two formulations is proved in Appendix \ref{apx:ray}. Finally, to be in line with the notation adopted in the standard references \cite[Section 1.2]{nualartbook} and \cite[Section 2.3]{nourdinpeccatibook}, we use the symbol $\mathbb{D}^{1,p}(\mu)$ instead of writing $\mathbb{D}^{p,1}(\mu)$ (inverting $1,p$ to $p,1$) as in \cite{bogachev}.
\begin{definition}\label{def:Dspace}
Let $E$ be a Banach space equipped with a centered Gaussian measure $\mu=[X]$ having full support, and let $\HX\subset E$ be the associated Cameron-Martin Hilbert space. Let $1\le p <\infty.$ We write $\mathbb{D}^{1,p}(\mu)$ to indicate the {\rm Malliavin-Sobolev space} of all functions $V\in L^p(\mu),$ with the following properties:
\begin{enumerate}[\rm (i)]
    \item {\rm ($V$ is {ray absolutely continuous})} For every $h\in \HX,$ there is a measurable set $N_h\subset E,$ with $\mu(N_h)=0,$ such that for all $x\in E \smallsetminus N_h,$ the function $t\mapsto V(x+th)$ coincides $dt$-almost everywhere with an absolutely continuous function $t\mapsto \f(t)$.
    \item {\rm ($V$ is \emph{stochastically Fr\'echet differentiable})} There exists measurable mapping $${\dmlv V:E\to \HX^*}$$ such that for any $h\in \HX,$ the expression 
    \be 
\mu\kop x\in E \colon \left|\frac{V(x+th)-V(x)}{t}-\langle \dmlv V(x),h\rangle \right|>\e\pok 
    \ee
    tends to zero as $t\to 0,$ for any choice of $\e>0$, i.e., there is convergence in probability of the divided difference. The mapping $\dmlv V$ is called the \emph{stochastic derivative} of $V.$
    \item The stochastic derivative is $p$-integrable: $\E\|\dmlv V\|_{\HX}^p<\infty.$
\end{enumerate}
If $V\in \mathbb{D}^{1,p}(\mu),$ then the random variable $\dmlv V(X) \randin \HX^*$ is called the \emph{Malliavin derivative} of $V$, and $V$ is said to be \emph{in the domain of the Malliavin derivative}. We equip $\mathbb{D}^{1,p}(\mu)$ with the norm:
\be 
\|V\|_{\mathbb{D}^{1,p}}:=\E\kop|V|^p\pok^\frac1p+\E\kop\|\dmlv V\|_{\HX}^p\pok^{\frac1 p}.
\ee
\end{definition}


\begin{remark}
    It is clear from the above definition that if $V\in \mathbb{D}^{1,p}(\mu)$ and $V$ is  of class $\mC^1$ on an open subset $\m U\subset E,$ then for $\mu$-a.e.  $x\in \m U$ one has that
    \be 
\dmlv V(x)=d_xV|_{\HX}.
    \ee
We recall that $\mathcal{H}_X$ is continuously embedded in $E$.   
    
\end{remark}

\begin{remark}{\rm 
By \cite[Thm. 5.7.2]{bogachev}, the normed spaces $\mathbb{D}^{1,p}(\mu)$ defined above are complete and contain the subset of smooth cylindrical functions as a dense subset, where the smooth cylindrical functions $V\colon E \to \R$ are those of the form $V=F\circ L,$ where $L : E\to \R^N$ is a continuous linear map for some $N\in \N$ and $F\in \m{C}^{\infty}(\R^N)$ has bounded derivatives of all orders. Thus, \cref{def:Dspace} coincides with the most common definition of the Malliavin Sobolev spaces $\mathbb{D}^{1,p}(\mu)$ as the completion of the spaces of smooth cylindrical functions with respect to the norm $\|\cdot\|_{\mathbb{D}^{1,p}}.$ See e.g. \cite[Section 1.2]{nualartbook} and \cite[Section 2.3]{nourdinpeccatibook}. } 
\end{remark}
\section{Nodal Volumes as Differentiable Mappings: Proof of \eqref{eq:firstvar}}\label{sec:nodaldiff}
\subsection{Some heuristic considerations} The main achievement of the present section is the derivation of formula \eqref{eq:firstvar} (see Theorem \ref{thm:firstvar}). Such a result yields an explicit representation of the Fr\'echet differential $d_f V$ (as defined in \eqref{e:banachgrad}) in the case where $E$ is a closed subset of the space $\m{C}^2(M)$ associated with a $m$-dimensional compact $\mC^2$ Riemannian manifold $M$, and $V(f) = \mathcal{H}^{m-1}(f^{-1}(0))$ is the {\it nodal volume} of a mapping $f\in \m{C}^2(M)$ such that $0$ is a regular value of $f$ (see Definitions \ref{d:nodalv} and \ref{def:regval}).

\smallskip

\begin{remark}\label{r:diffheur} Fix $m\geq 2$. In the case where $M$ is a bounded domain of $\mathbb{R}^m$ with a smooth boundary, and $f : M\rightarrow \mathbb{R}$ is a $\m C^2$ mapping such that $0$ is a regular value of $f$, formula \eqref{eq:firstvar} takes the following simple form: for the mapping $g\mapsto V(g) = \mathcal{H}^{m-1}(g^{-1}(0))$ and for any $h\in \mathcal{C}^2(M)$, writing $Z = f^{-1}(0)$, one has that
\begin{eqnarray}\label{e:heu}
\langle d_f V, h\rangle &= &\int_Z \frac{h(x)}{\|\nabla f(x)\|^2}\cdot \left\{\nu\, {\rm Hess } f(x) \nu^T - \Delta f(x) \right\} \, \mathcal{H}^{m-1}(dx) \\
&&+ \int_{Z\cap \partial M} h(x) \cdot \frac{\langle \nabla f(x), {\bf n}(x)\rangle}{\|\nabla (f|_{\de M})\|\|\nabla f(x)\| } \mathcal{H}^{m-2}(dx), \notag
\end{eqnarray}
where $\nu = \nabla\, f(x) / \|\nabla\, f(x)\| $, and ${\bf n}(x)$ is the outward pointing normal at $x\in \partial M$. It is interesting that formula \eqref{e:heu} can be deduced by {\it formally} differentiating the right-hand side of \eqref{e:volumecondirac} under the integral sign. Writing $A_x (g) := \delta_0(g(x))$ and $B_x(g) := \|\nabla g(x)\| $, one has indeed the formal chain rule
\begin{eqnarray}\label{e:heu2}
\langle d_f V, h\rangle &= &\int_M \langle d_f A_x , h \rangle \|\nabla f(x)\| dx + \int_M \delta_0(f(x)) \langle d_f B_x, h \rangle dx := (I)+(II).  
\end{eqnarray}
Now, standard computations yield that
$$
\langle d_f B_x, h \rangle = \frac{\langle \nabla f(x) , \nabla h(x)\rangle}{\|\nabla f(x)\|},
$$
 whereas a formal integration by parts based on the divergence theorem leads to
 $$
 (I) = - \int_M \delta_0(f(x)) \left\{ \nabla \bullet \frac{h(x) \nabla f(x) }{\|\nabla f(x)\|} \right\} dx + \int_{\partial M} \delta_0(f(x)) h(x)  \frac{\langle \nabla f(x), {\bf n}(x)\rangle}{\|\nabla f(x)\|}  dx, 
 $$
where the symbol $\bullet$ indicates a scalar product in $\mathbb{R}^m$; the fact that $(I)+(II)$ equals the right-hand side of \eqref{e:heu} now follows by explicitly computing the term between curly brackets inside the first integral.
\end{remark}

\subsection{Elements of Riemannian geometry}\label{ss:riemannelements}
Our main references for this section are the monographs \cite{AdlerTaylor, Hirsch, leesmooth, leeriemann, li2012geometric}, to which we refer the reader for any unexplained notion and results. For $m\geq 1$, we consider an $m\ge 1$ dimensional $\mC^2$ {compact} Riemannian manifold $M$, possibly with boundary $\de M$, endowed with a metric $g:TM\otimes TM\to \R$ of class $\mC^2$, with associated norm $\|\cdot\|=\sqrt{g(\cdot,\cdot)}$. 
We can think that $M$ is a stratified manifold (see \cite[Section 8.1]{AdlerTaylor}) with two strata: the interior $\inter M$ of dimension $m=:\dim M$ and the boundary $\de M$ of dimension $m-1$, with the tangent bundle of $\inter M$ that extends to a vector bundle $TM$ over all $M$. There is a $\mC^1$ section $\n:\de M\to TM$ defined as the unit (i.e., $\|\n\|=1$) outward normal vector to $\de M$, so that $TM|_{\de M}=T{\de M}\oplus_\perp \mathrm{span}\{\n\}$; 
see e.g. \cite[p. 391]{leesmooth}.

\subsubsection{Volumes}
We use the notation $\int_M f dM=\int_M f(p) dM(p)$ for the integration of a Borel function $f:M\to \R$ with respect to the \emph{volume measure} of $(M,g)$. The integral is defined as the linear functional such that when $f$ is supported in the domain of a chart $\f\colon O\to \R^m$, 
\be 
\int_M f dM:=\int_{\f(O)}f(\f^{-1}(x))\sqrt{\det G(x)}dx,
\ee 
where $G_{ij}(x)=g\tyu\frac{\de \f^{-1}}{\de x^i}(x), \frac{\de \f^{-1}}{\de x^j}(x)\uyt$. 
In particular, if $Z\subset M$ is a submanifold of dimension $k$, endowed with the metric induced by $g$ and $\vol^{k}$ denotes the $k$-dimensional Hausdorff measure on $M$ with respect to the geodesic distance induced by $g$ (see e.g. \cite[pp. 168-169]{AdlerTaylor}), then $\int_Z f dZ=\int_M f(p)1_Z(p) \vol^k(dp)$, so that the \emph{volume}, or \emph{$k$-volume}, of $Z$ is
\be 
\vol^k(Z)=\int_Z1 dZ.
\ee
Plainly, the $1$-volume is the \emph{length} and the $2$-volume is the \emph{surface area}.
\subsubsection{Gradient, Hessian and Laplacian}
The metric induces an isomorphism of vector bundles $\sharp_g\colon T^*M\to TM$, which \emph{raises the indices}
 and whose inverse is denoted by $\flat_g$. We will adopt the standard notation 
 \be 
 \a^\sharp:=\sharp_g (\a),
 \ee 
 for any $\a\in T^*_pM$, so that for all $v\in T_pM$, we have  $g(\a^\sharp,v)=\langle\a,v\rangle$.
 The \emph{gradient} of a differentiable function $f\in\mC^{2}(M)$ is the vector field $\grad{f}={d f}^\sharp\in \Vect^1(M)$ such that
 \be 
 g\tyu \g f(p), v\uyt=\langle d_pf, v\rangle
 \ee 
 Let $\nabla$ denotes the Levi-Civita connection of the metric $g$. The \emph{Hessian} of $f\in\mC^{2}(M)$ at $p\in M$ is the symmetric bilinear form $\hess_p f\colon T_pM\times T_pM\to\R$ such that
\be 
\hess_p f(v,w)=\langle \nabla_v(d_pf),w\rangle,
\ee
and defines a continuous tensor $\H f$.
The corresponding linear operator is $\h f =\sharp_g\circ \H f\colon TM\to TM$, that is the operator such that $g(\h f(v),w)=\H f(v,w)$ for all $v,w\in TM$.
Since $\nabla\sharp_g=\sharp_g\nabla,$ it follows that $\h f=\nabla \g f.$
Taking the trace of the Hessian, with respect to the metric, one obtains the Laplace-Beltrami operator $\Delta\colon \mC^{2}(M)\to \mC^0(M)$, or just the \emph{Laplacian}, of the metric $g$: 
\be 
\Delta f=\mathrm{tr}_g( \H f)=\mathrm{tr}(\h f).
\ee
\subsubsection{Mean curvature}
Given $f\in \mC^2(M)$, such that $\{p\in M:f(p)=0,d_pf=0\}=\emptyset$, then $Z=f^{-1}(0)$ is a $\mC^{2}$ hypersurface, cooriented by the normal vector field $\nu=\|\g f\|^{-1}\g f$. The \emph{second fundamental form} of $Z$ at $p\in Z$ is the bilinear form $\mathrm{II}_p=\|d_pf\|^{-1}\H_p f|_{TZ}$.  Indeed, for any $v,w\in T_pZ$ we have
\be \label{eq:II}
\mathrm{II}(v,w)\overset{\mathrm{def}}{=}g( v,\nabla_w \nu) = \frac{\H f(v,w)}{\|df\|}
-g( v,\nu)\frac{\H f(\nu,w)}{\|df\|},
\ee
but the second term vanishes because $T_pZ=\nu(p)^\perp.$
Finally, the \emph{mean curvature} of $Z$ is the trace of $\mathrm{II}$ with respect to the restriction of $g$ to $TZ$, that is the function $H_Z\in\mC^0(Z)$ 
\be\label{eq:mean}
H_Z=\text{tr}_g\tyu {\mathrm{II}}\uyt=\|df\|^{-1}\tyu \Delta f-\H f(\nu,\nu)\uyt.
\ee
\subsection{Nodal volume and regular values: main results}
The next definition formally introduces the main object of our paper. 
\begin{definition}[Nodal volumes]\label{d:nodalv}
Let the above notation and assumptions prevail. We define the \emph{nodal volume} as the mapping $V\colon \mC^2(M)\to \R \, \colon f\mapsto V(f)$ such that 
\be \label{e:defnodvol}
V(f):=\vol^{m-1}(f^{-1}(0)).
\ee
\end{definition}
Observe that the definition of $V$ is well given because $f^{-1}(0)$ is a closed (and, therefore, Borel) set. In general, $V$ is a mapping possessing some degree of regularity only when restricted to functions $f$ such that $f^{-1}(0)$ is a submanifold -- the latter property is encoded in the fundamental notion of \emph{regular value}. 
\begin{definition}[Regular Values]\label{def:regval}
For $m\geq 1$, let $M$ be a $m$-dimensional $\mC^2$ Riemannian manifold with boundary, let $f\in \mC^2(M)$ and $r\in\R$. If $\de M=\emptyset$, then we say that $r$ is a \emph{regular value} of $f$, and write $f\transv r$, if for every $p\in M$, the following implication holds: $f(p)=r\implies d_pf\neq 0$.
In general, we say that $r$ is a \emph{regular value} of $f$, and write $f\transv r$, if and only if $f|_{\inter M}\transv r$ and $f|_{\de M}\transv r$. 
Moreover, $r$ is said to be a \emph{critical value} of $f$ if it is not a regular value. 
From now on, we write
\begin{equation}\label{e:callu}
\m U:=\kop f\in \mC^2(M): f\transv 0\pok.
\end{equation}
\end{definition}
\begin{definition}[Neat Hypersurfaces]\label{def:neat}
Following \cite[Section 1.4]{Hirsch}, we say that a subset $Z\subset M$ is a $\mC^2$ \emph{neat submanifold} of $M$, having codimension $k$, if for every point $p\in Z$, there is a $\mC^2$ chart $\f\colon O\to \R^{m-1}\times [0,+\infty)$ of $M$ around $p$, such that $Z\cap O=\f^{-1}(\R^{m-k})$. 
As customary, we will say: \emph{hypersurface}, in place of: \emph{submanifold of codimension one}.
\end{definition}

In particular, a $\mC^2$ neat hypersurface $Z\subset M$ is a $\mC^2$ manifold with boundary $\de Z=Z\cap \de M$, and such that for every $x\in \de Z$, we have $T_xZ\cap T_x\de M=T_x\de Z$. We refer to \cite[Section 1.4]{Hirsch} for more details. 
The following statement explains the relation between the two above definitions; its proofs can be found in \cite[Theorem 1.4.1]{Hirsch}. 

\begin{proposition}\label{prop:neat}
The class $\m U$ defined in \eqref{e:callu} is an open subset of $\mC^2(M)$. Moreover, if $f\in \m U$, then $Z:=f^{-1}(0)$ is a $\mC^2$ neat hypersurface of $M$.
\end{proposition}
In short, all subsets that are not neat hypersurfaces are somewhat degenerate for our study. In this sense, the functions in the class $\m U$ are non-degenerate, in that the equation $f=0$ defines a neat hypersurface.
The next result is the main achievement of the section, as well as the linchpin of the entire paper. The proof is deferred to Section \ref{ss:proofirstvar}.

 \begin{theorem}[Nodal volumes as differentiable mappings]\label{thm:firstvar}
Assume that $M$ is a $\mC^2$ Riemannian manifold with boundary $\de M$ and let $\n\in \Gamma^\infty(TM|_{\de M})$ 
 denote the outward normal vector to the boundary. Define $\m U\subset \mC^2(M)$ as in \eqref{e:callu}. Then, the following conclusions hold:
 \begin{enumerate}[\rm (1)]
     \item  The nodal volume mapping $V$ defined in \eqref{e:defnodvol} is such that $V\in \mC^1(\m{U})$.  
\item Fix $f\in \m{U}$, let $Z:=f^{-1}(0)$, and define $\nu:=\|\g f \|^{-1}\g f$ and $\tDelta f:=\Delta f-\H f(\nu,\nu)$. Then, formula \eqref{eq:firstvar} holds for all $h\in \mC^2(M)$, that is:
\be\label{eq:firstvar1} 
\langle d_fV,h\rangle=-\int_{Z}h\cdot \frac{\tDelta f}{\|df\|^{2}} dZ+\int_{\de Z}h\cdot \frac{ g(\n,\nu) }{\|d(f|_{\de M})\|} d\de Z,
\ee
where the existence of the integrals on the right-hand side is part of the conclusion.
 \end{enumerate}
 
\end{theorem}

For the next statement, we assume that $X\randin \m C^2(M)$ is a centered Gaussian element with values in $\m C^2(M)$ and law $\mu$. We write $\mathbb{E}[X(p)X(q)] := K_X(p,q)$ for the covariance of $X$ and denote by $E\subset \m C^2(M)$ the topological support of $\mu$.

\begin{corollary}[Cameron-Martin norms]\label{cor:expDV}
Under the above notation and assumptions, let $\HX\subset E$ denote the Cameron-Martin space of $X$, and assume that
$$
\mathbb{P}[X\in \mathcal{U}] = \mu(E\cap \mathcal{U}) = 1.
$$
Let $Z=Z_X=X^{-1}(0)$ and set $\nu = \nu_X=\|\g X\|^{-1}\g X$. Then, the nodal volume mapping $V$ defined in \eqref{e:defnodvol} is such that $V\in \m C^1(\mathcal{U}_0)$, where $\m U_0 :=E\cap \m U $, and consequently $d_XV|_{\HX}\randin \HX^*\cong \HX$. Moreover, a.s.-$\mathbb{P}$,
\bega 
 \| d_XV\|^2_{\HX} = &\int_Z\int_Z \frac{\tilde{\Delta} X(p) }{\|d_pX\|^2}\frac{\tilde{\Delta} X(q) }{\|d_qX\|^2}K_X (p,q)dZ(p)dZ(q)
 \\
 &+
 \int_{\de Z}\int_{\de Z} \frac{g_p(\n,\nu)g_q(\n,\nu)  }{\|d_p(X|_{\de M})\|\|d_q(X|_{\de M})\|}K_X (p,q)d\de Z(p)d\de Z(q).
\eega
\end{corollary}
\begin{proof} The proof is a direct consequence of the content of Remark \eqref{r:cmspace}-(b), in the special case $G = d_X V$ and $\theta(dp) = \| d_p X\|^{-2} 1_Z(p)  dZ+ 1_{\partial Z}(p)g_p(\n, \nu)\| d_p X|_{\partial M} \|^{-1} d\partial Z$.

\end{proof}

\begin{remark}{\rm In the previous statement, the random element $\Omega \ni \omega \mapsto d_X V (\omega)$ is defined as follows: $d_XV ( \omega) = d_{X(\omega)} V$, whenever $\omega$ is such that $X(\omega) \in \m U_0$, and $d_X V ( \omega)= 0$ otherwise. } \end{remark}

\medskip

Before proving Theorem \ref{thm:firstvar} in Section \ref{ss:proofirstvar}, one needs to study in some detail the first-order variations of nodal volumes. The necessary technical statements are gathered together in Sections \ref{ss:firstvar} and \ref{ss:changevar}.

\subsection{First variation of the volume}\label{ss:firstvar}
The mean curvature of a (compact and oriented) hypersurface is known to be the derivative of the $d-1$ dimensional volume of the surface (see \cite{li2012geometric}), in the sense that if we move $Z$ in the direction of $\nu$ for a small time $t$ inside a closed manifold $M$, then, the volume varies as $\vol^{m-1}(Z_t)=t\int_Z H_Z dZ+o(t).$ 
Let us report the most general formula, including the case when the variation has a tangential component $T\in\Vect^1(Z)$ and $Z\subset M$ is a submanifold with boundary $\de Z$ without any assumption on the relative positions of $\de Z$ and $\de M$.
    Assume that $Z_t=\phi_t(Z)$, where $\phi\colon (-\e,\e)\times Z\to M$ is a $\mC^1$ map such that $\phi_t=\phi(t,\cdot)$ is an immersion for all $t$. Let us define the \emph{first variation} of such family of parametrized hypersurfaces as $\delta{Z}:=\frac{d \phi_t}{dt}\big|_{t=0}$, which is a $\mC^1$ section of $TM|_{Z}$.
Then we can decompose $\delta Z=T+\dot Z\nu$ into a tangential vector field $T\in \Vect^1(Z)$ and a section of the normal bundle determined by a function $\dot{Z}\in \mC^1(Z)$. Recall that in this case $T_zZ^\perp$ is just a line spanned by $\nu(z)=\|\g f(z)\|^{-1}\g f(z)$. The following formula holds (see \cite[Section 1, page 4]{li2012geometric}).
\be\label{eq:li} 
\frac{d}{dt}\Big|_0\vol_{m-1}\tyu Z_t\uyt =\int_Z  \dot{Z}\cdot H_Z +\mathrm{div}(T) dZ.
\ee
We will need the following formulation of the above formula.
\begin{proposition}\label{prop:boundarypuff} Let the assumptions above prevail. If, moreover, $\de Z_t\subset \de M$, but $T Z_0|_{\de Z_0}\neq T\de M$, then
\bega\label{eq:firstvargeometric}
\frac{d}{dt}\Big|_0\vol_{m-1}\tyu Z_t\uyt &=
\int_Z  g( \delta{Z},\nu) \cdot H_Z dZ
-
\int_{\de Z}g( \delta{Z},\nu) 
 \frac{g(\n,\nu)}{\sqrt{1-g(\n,\nu)^2}} d\de Z,
\eega 
where $H_Z$ is the mean curvature of $Z$ in the direction $\nu$, defined as in \cref{eq:mean}.
\end{proposition}
\begin{proof}
By \eqref{eq:li} and Stokes-divergence theorem, we have that
 \bega
\frac{d}{dt}\Big|_0\vol_{m-1}\tyu Z_t\uyt &
=\int_Z g( \delta{Z},\nu) \cdot H_Z dZ+\int_{\de Z}g(\n_\de,T) d\de Z,
\eega 
where $\n_\de\in \Gamma^\infty(N\de Z)$ is the outward unit normal vector of $\de Z$. We have that for $x\in\de Z$, $\n=\n(x)$ belongs to the space $(T_x\de Z)^\perp$ spanned by the orthonormal basis given by $\n_\de=\n_\de(x)$ and $\nu=\nu(x)$, 
 Moreover, by construction, both $\n_\de$ and $\n$
 are pointing outside of $\de M$, thus $g(\n,\n_\de)=\sqrt{1-g(\n,\nu)^2}>0$. 
Under the hypothesis that $\de Z_t\subset \de M$ for all $t$, we must have that $\delta Z(x)\in T_x\de M=\n^\perp$, so that, evaluating at $x$,
 \be 
 g(\n_\de,T)=g(\n_\de,\delta Z)
 =
 -\frac{g(\n,\nu)}{\sqrt{1-g(\n,\nu)^2}}g(\nu, \delta Z).
 \ee
 \end{proof}
 \subsection{Change of variation formula}\label{ss:changevar}
Let $f\in \m U\subset \mC^2(M)$ and $Z=f^{-1}(0)$. In order to apply the formula of \cref{prop:boundarypuff} we need to represent the family of hypersurfaces $Z_t=(f+th)^{-1}(0)$ as a small translation of $Z$ in the direction of a vector field $\delta Z:Z\to TM$.
To this end we look for a time dependent vector field $W\in\mC^1(\R\times M,TM)$ such that, if $P^t$ is the flow of $W_t$, then $P^t(Z)=Z_t$. 
This is equivalent to look for a solution $W_t$ of the equation:
\be 
(f+th)(P^t(z))=0, 
\ee
By differentiating with respect to $t$, we get the equivalent condition that: 
\be\label{eq:dW} 
\langle d_pf+td_ph,W_t(p)\rangle+h(p)=0, \text{ for all $p\in Z_t$ and $t\in (-\e,\e)$}.
\ee
The second equation is solved in a neighborhood of $Z$ by the vector field 
\be\label{eq:Wt} 
W_t=-h\|\g (f+th)\|^{-2}\g (f+th),
\ee 
which is of class $\mC^1$ in that neighborhood. Using a partition of unity, $W_t$ can be extended to a $\mC^1$ vector field $W_t$ on the whole manifold $M$. 
If the manifold is closed, i.e. $\de M=\emptyset$, then $P^t$ is defined for all times and we have $P^t(Z)=Z_t$ for small enough $t$. In this case, we can apply \cref{prop:boundarypuff} directly with 
\be\label{eq:changeofvariation}
\delta{Z}=W_0|_Z= -h\frac{\g f}{\|\g f\|^2}.
\ee
and deduce the formula for $d_fV$ from \cref{eq:mean}.
In the general case, we cannot argue that $Z_t=P_t(Z)$ because the flow $P_t$ is not defined for all times, since the trajectories could fall outside the boundary. In this case however, \cref{eq:dW} holds. 
\begin{lemma}\label{lem:flowP}
Let $f\in \m U\subset \mC^2(M)$ and $Z=f^{-1}(0)$. Let $P(t,x)=P^t(x)$ be the flow of a time dependent vector field $W\in\mC^1(\R\times M,TM)$, that satisfies $\eqref{eq:Wt}$ and $g(\n,W_t|_{\de M})= 0$ for $|t|\le \e$. Then, $P:(-\e,\e)\times M\to M$ is a $\mC^1$ map and, for all $t\in (-\e,\e)$, the map $P^t\colon M\to M$ is a $\mC^1$ diffeomorphism such that $P^t(Z)=Z_t$.
\end{lemma}
\begin{proof}
Since $g(\n,W^{\de}_t)= 0$, the vector field $W_t^n$ is tangent to the boundary $\de M$. This ensures, by standard O.D.E. theory (see \cite[Section 6.2]{Hirsch}), that there exists a flow $P_t^n\colon M\to M$ of $W_t^n$ defined for all times $t\in \R$. Since $W^n\in \mC^1(\R\times M,TM)$, by the property of smooth dependence on initial data of O.D.E., we have $P^n\in \mC^1(\R\times M,M)$. 
\end{proof}
\subsection{Proof of Theorem \ref{thm:firstvar}}\label{ss:proofirstvar}

Let $Z_t=\{f+th=0\}$ and $Z_0=Z$.
We take the time-dependent vector field 
$\inter{W}_t\in \mC^1(\R\times M, TM)$ is defined so that $\inter{W}_t=-h\|\g (f+th)\|^{-2}\g (f+th)$ in a neighborhood $\m Q$ of $Z$ which contains all $Z_t$ for $|t|\le t_0$ small enough. Therefore, for every $p\in Z_t$, we have that 
\be \label{eq:sodfnvoivprimo}
\langle d_pf+td_ph,\inter{W}_t(p)\rangle+h(p)=0, \text{ for all $p\in Z_t$}.
\ee
Observe that, applying the same argument to $f|_{\de M}$ yields a time-dependant vector field $W^{\de} \in \mC^1(\de M\times \R, T\de M)$ such that
\be \label{eq:sodfnvoiv}
\langle d_pf+td_ph,W^{\de}_t(p)\rangle+h(p)=0, \text{ for all $p\in \de Z_t$}.
\ee
Let $\rho_n\in\mC^{\infty}(M,[0,1])$ be a sequence of smooth functions such that $\rho_n(\de M)=1$ and such that $\sup_K \rho_n\to_{n\to
\infty} 0$ for any $K\subset \inter{M}$ compact subset. This can be achieved by taking $\rho_n$ supported in a tubular neighborhood of $\de M$ of radius $\frac{1}{n}$. It is also convenient to extend $\n$ and $W^{\de}_t$ to vector fields $\n\in \Vect^\infty(M)$ and $W^\de_t\in \Vect^1(M)$. We can assume that \cref{eq:sodfnvoiv} remains true for all $p$ in a fixed neighborhood $\m Q_\de$ containing $\de Z_t$ for all $|t|\le t_0$. Define
\be 
W_{t}^n:=\inter{W}_t (1-\rho_n) +W^\de_t \rho_n \in \Vect^1(M),
\ee
Observe that $g(\n,W_t^n)=g(\n,W^{\de}_t)=0$. 
By construction, there is $n_0\in\N$ big enough that $Z_t\cap \mathrm{supp}(\rho_n)\subset \m Q_\de$, for all $|t|\le t_0$ and $n\ge n_0$. Then, for $p\in Z_t$, both \cref{eq:sodfnvoivprimo} and \cref{eq:sodfnvoiv} hold and hence, by linearity, we can apply \cref{lem:flowP} to the vector field $W^n_t$,
hence $P_t^n(Z)= Z_t$ for all $t\in [-t_0,t_0]$ and $n\ge n_0$. Define a parametrization of $Z_t$ as the $\mC^1$ map $\phi\colon \R\times Z\to M$ such that $\phi(t,z)=P_t^n(z)$, then an application of \cref{prop:boundarypuff} 
with $\delta Z=W_0^n$, yields
\bega
\frac{d}{dt}\Big|_0&\vol_{m-1} \tyu Z_t\uyt =\int_Z g( W_0^n,\nu) \cdot H_Z dZ
-
\int_{\de Z}g(W_0^n,\nu)
 \frac{g(\n,\nu)}{\sqrt{1-g(\n,\nu)^2}} d\de Z(x)=
\\
&=
\int_Z \tyu \frac{df(\inter{W}_0)(1-\rho_n)}{\|d f\|}+g( \nu,W^\de_0) \rho_n \uyt \cdot H_Z dZ\\
&\quad -
\int_{\de Z}\frac{df( W_0^\de)}{\|d f\|}\frac{df (\n)}{\sqrt{\|df\|^2-df(\n)^2}} d\de Z(x)
\\
&=\dots,
\eega 
for all $n\ge n_0$, thus letting $n\to \infty$, we obtain that
\bega 
\dots &= \int_Z  \frac{-h}{\|d f\|}  \cdot H_Z dZ -\int_{\de Z} \frac{-h}{\|d f\|} \frac{df (\n)}{\sqrt{\|df\|^2-df(\n)^2}} d\de Z(x)=
\\
&=
\int_Z  \frac{-h}{\|d f\|}  \cdot \frac{\tDelta f}{\|d f\|} dZ
+\int_{\de Z} \frac{h}{\|d f\|} \frac{df (\n)}{\|d(f|_{\de M})\|} d\de Z(x).
\eega
The fact that $V\in\mC^1(\m{U})$, follows because the differential map $dV\colon \m{U}\to \mC^2(M)^*$ is continuous.
\qed
\subsection{Isotopy}
\begin{definition}[$\mC^1$-isotopic]\label{def:c1iso}
Let $Z_0,Z_1\subset M$ be two $\mC^1$ submanifolds. We say that \emph{$Z_0$ and $Z_1$ are $\mC^1$-isotopic in $M$} if there is a $\mC^1$ diffeotopy of $M$ that sends $Z_0$ to $Z_1$, namely, a $\mC^1$ function $P\colon M\times [0,1]\to M$ such that: $P(\cdot,t)$ is a $\mC^1$ diffeomorphism for all $t\in [0,1]$, $P(\cdot,0)$ is the identity and $P(Z_0,1)=Z_1$.
\end{definition}
\begin{remark}
This definition agrees with that of \cite[Chapter 8, Section 1]{Hirsch}, after replacing $\mC^\infty$ with $\mC^1$, see \cite[Chapter 8, Section 1, Exercise 4]{Hirsch}. Precisely, because of the Isotopy Extension properties (see \cite[Theorems 1.5-1.9]{Hirsch}), two submanifolds $Z_0$, $Z_1$ are $\mC^1$-isotopic, according to \cref{def:c1iso}, if and only if there is a $\mC^1$ isotopy $F_t\colon Z_0\to M$ (defined as in \cite[Chapter 8, Section 1]{Hirsch}) of the embedding $F_0\colon Z_0\subset M$ in $M$, such that $F_1(Z_0)=Z_1$. 
\end{remark}
The notion of isotopy is very natural in the context of Morse theory. Given a Morse function $T\colon M\to \R$ (see \cite[Section 6.1]{Hirsch} for definitions), if there are no critical values in the interval $(a,b)$, then the level sets $T^{-1}(a)$ and $T^{-1}(b)$ can be continuosly deformed one into the other, see \cite[Theorem 3.1]{MilnorMorse}. The deformation is obtained from a construction based on the gradient flow of $T$, which naturally provides an isotopy of the same regularity as that of the differential of $T$. One can then use \cref{prop: transvprop} below to prove that the zero sets $f_a^{-1}(0)$ and $f_b^{-1}(0)$ are $\mC^1$-isotopic for any $\mC^1$ family of functions $f_t$ all contained in $\m U$ (see \cref{e:callu}). We make this discussion rigorous, with the following lemma. 
\begin{lemma}\label{lem:c1iso}
 Let $F\colon M\times [0,1]\to M$ be continuous and assume that $f_t=F(\cdot,t)$ is in $\m U=\{f\in \mC^2(M)\colon f\transv 0\}$ for all $t\in [0,1]$. Then, the zero sets $f_1^{-1}(0)$ and $f_0^{-1}(0)$ are $\mC^1$-isotopic in $M$. 
\end{lemma}
For compact manifolds and $f_t=T-t$, this is an easy consequence of the proof of \cite[Theorem 3.1]{MilnorMorse}, although it does not follow quite explicitely from the statement. The same applies to other mainstream references on Morse theory (for instance, \cite{nicolaescu2011Morse,GoreskyMacPherson,audin2013morse,banyaga2004lectures}), which indeed are more concerned with topological properties of the level sets. See \cite{MorseBoundary,laudenbach:morseBoundary} for extensions to manifolds with boundary. The general case $f_t$ can be seen as a particular case of the setting of Thom's isotopy theorem, see \cite{GoreskyMacPherson, Mathernotes} and the references therein. Such point of view has been adopted in a recent paper \cite[Theorem 4.1]{beliaev2023covariance} to prove a statement analogous to \cref{lem:c1iso}, but which postulates the existence of an isotopy that is only $\mC^0$. We report a short proof of \cref{lem:c1iso}, since all the ingredients are already in \cref{lem:flowP}.
\begin{proof}
We start by observing that being $\mC^1$-isotopic is an equivalence relation. To see the transitivity, observe that two $\mC^1$ isotopies $P,Q$ such that $P(\cdot,1)=Q(\cdot,0)$ can be attached together using a smooth function $\rho:[0,1]\to [0,1]$ such that $\rho(t)=0$ for $t$ in a neighborhood of $0$ and $\rho(t)=1$ for all $t$ in a neighborhood of $1$. Then, the function 
\be 
F(x,t)=P(x,\rho(2t))1_{\{2t\le 1\}}+Q(x,\rho(2t-1))1_{\{2t\ge 1\}}
\ee 
is still a $\mC^1$-isotopy with $F(\cdot,0)=P(\cdot,0)$ and $F(\cdot,1)=Q(\cdot,1)$.
Let us replace the curve $t\mapsto F(\cdot,t)$ in $\m U$ with one, called $\tilde F$, that is piece-wise of the form $\tilde{F}(\cdot,t)=f_i+t h_i$ for $t\in [t_i,t_{i+1}]$, for some finite sequence $0=t_0<\dots <t_n=1$ and $h_i\in \mC^2(M)$. The sole constraint on $\tilde{F}$ is that $\tilde{F}(\cdot,t)\subset \m U$ and $\tilde F(\cdot,t)=F(\cdot,t)$ for $t\in \{0,1\}$. This is possible simply because $\m U$ is an open subset of $\mC^2(M)$. Knowing that $\mC^1$-isotopy is an equivalence relation, we can reduce our study to a single piece, namely to the case when $F(\cdot,t)=f_0+th$ for all $t\in [0,1]$, which falls in the context of \cref{lem:flowP}.
From \cref{lem:flowP} and the antecedent discussion, we obtain that for every $s\in [0,1]$, there is an $\e_s>0$, such that if $|t-s|<\e_s$, then $Z_t=f_t^{-1}(0)$ is $\mC^1$-isotopic to $Z_s=f_s^{-1}(0)$. By the compactness of $[0,1]$, we deduce that we can find a finite sequence $0=t_0<t_1<\dots <t_N=1$ such that $Z_{t_i}$ is $\mC^1$-isotopic to $Z_{t_{i+1}}$ and conclude, from the transitivity, that $Z_0$ is $\mC^1$-isotopic to $Z_1$. 
\end{proof}
\section{Non-degeneracy conditions}\label{sec:nondegeneracyconditions}
In order to make a meaningful study of the nodal set $Z=X^{-1}(0)\subset M$ of a random field $X$, it is important to rule out the most degenerate situations. In fact, one can show that if $X$ is supported on the whole space $\mC^2(M)$, then $Z$ ranges over all closed subsets of $M$, see \cite{WhitneyExtension}.\footnote{ 
For any $C\subset M$ closed subset, there exists $f\in \mC^\infty(M)$ such that $C=f^{-1}(0)$. This result is generally attributed to Whitney, as a corollary of his celebrated Extension Theorem \cite{WhitneyExtension}.}
However, by Bulinskaya Lemma (see \cref{lem:bulinskaya} below), under \cref{ass:1}, $Z$ is a hypersurface with probability one. In other words, the set 
\be\label{eq:W} 
\m{W}:=\mC^2(M)\smallsetminus \m U
\ee 
of degenerate maps has zero probability: \cref{lem:bulinskaya} below says that $\PP(X\in \m W)=0$. 

We know by \cref{thm:firstvar} that, outside $\m W$, the nodal volume map $V$ is of class 
$\mC^1$, hence the study of the Malliavin differentiability of $V$ naturally reduces to a study of $V$ near $\m W$. Precisely, in this section we will be concerned about condition (i) of \cref{def:Dspace}, which requires to study the regularity of the restriction of $V$ to random segments of the form $[X,X+h]$. Therefore, we will need to investigate the intersections of such random segments with $\m W$.

The leading idea of this section, and one key idea of the whole paper, is the following heuristics: 
\emph{Let $E$ be the support of $X$. The set $\m W \cap E$ is a stratified hypersurface of $E$. In particular, there is a singular subset $\m W'$ of codimension $2$ such that $E\cap \m W\smallsetminus \m W'$ is a hypersurface.} This statement takes a precise meaning when $E$ is finite dimensional, although it is not always true under the sole \cref{ass:1}. 
The notion of codimension in this discussion is purely heuristic and it is expressed by the property, which is the content of \cref{thm:transcurve}, that a random segment $[X,X+h]$ in $E$ intersects $\m W$ in a finite set, all contained in $\m W\smallsetminus \m W'$. 
A second key idea is that the main stratum of $\m W$ is the set
\be \label{eq:Uprime}
\m{U}':=\{f\in \mC^2(M): \text{$0$ is a Morse critical value of $f$}\}\subset \m{W},\quad \m{W}':=\m{W}\smallsetminus \m{U}', 
\ee
where the definition of Morse critical value will be given below, see \cref{def:morse} in \cref{sec:degMorse}. In order to prove this, we will introduce the concept of \emph{transverse curves} 
 (see \cref{def:transv} and \cref{prop: transvprop} below), which is analogous to the standard notion of transversality in differential geometry (see \cref{def:transv} below).
 The fact that the degenerate functions in a random segment $[X,X+h]$ have only Morse critical zeroes, will allow us to completely understand when the function $t\mapsto V(X+th)$ is absolutely continuous, and thus when the condition (i) of \cref{def:Dspace} holds. This is the content of \cref{sec:levelMorse}.

There is one caveat to the previous paragraph: for technical reasons, we will also need to make sure that not too many Morse critical zeros will appear at the same time. We will discuss this in \cref{sec:bestia} and prove a very precise and general result in this direction, see \cref{thm:bestia}, which is of independent interest. Nevertheless, we stress that for us this is important only to establish Point (iv) of \cref{t:mainintro}, in the case $m=2$.
\subsection{The nodal set is a submanifold}
When $M\subset \R^m$ is an open domain with boundary, the first order Taylor polynomial of a function $f\colon M\to \R$ at $p\in M$ is identified by the pair $j^1_pf:=(f(p),d_pf)$, taking values in $M\times \R\times \R^m$. The classical Bulinskaya Lemma is phrased in terms of $j^1_pf$.
Using the language of vector bundles (for which we refer to \cite[Section 4.1]{Hirsch}), the result generalizes to one that is convenient in our setting. 
\begin{definition}
The \emph{first jet bundle} of $M$ is the vector bundle $J^1(M):=\R\times TM$. For any $f\in \mC^{1}(M)$, we call the \emph{first jet of $f$} the section $p\mapsto j^1_pf=(f(p),d_pf)$ (see  \cite[Section 2.4]{Hirsch}).
\end{definition}
Recall the definition of the set $\mathcal{U}$ given in \eqref{e:callu}. 
By \cref{prop:neat}, all maps $f\in \m U$ are non-degenerate in the sense that $f^{-1}(0)$ is a $\mC^2$ neat submanifold of $M$.
The following classical lemma gives conditions for which the set 
$
\m{W}$
of degenerate maps, defined as in \cref{eq:W}, has zero probability.
\begin{lemma}[Bulinskaya, \cite{AzaisWscheborbook}]\label{lem:bulinskaya}
Let $X\randin \mC^2(M)$ and assume that for any $p\in M$ the random vector $j^1_pX=(X(p),d_pX)\randin \R\times T_pM$ has a density $\rho|_{\R\times T_pM}$, where $\rho\colon \R\times TM\to\R$ is a locally bounded function, then 
\be 
\PP\{X\in \m U\}=\PP\{X^{-1}(0) \text{ is a $\mC^2$ neat hypersurface of $M$}\}=1.
\ee
\end{lemma}
Recall the definition of the set $\mathcal{U}$ given in \eqref{e:callu}. 
\subsection{Transversality}
We recall the following standard notions from differential geometry. For references, see \cite[Section 3.2]{Hirsch}.
\begin{definition}\label{def:transv}
Let $f\colon M\to V$ a be a $\mC^r$ map, with $r\ge 1$, between $\mC^r$ manifolds, possibly non compact and without boundary. Let $W\subset V$ be a $\mC^r$ submanifold. We say that \emph{$f$ is transverse to $W$} and write
\be 
f\transv W
\ee
if for every $p\in f^{-1}(W)$, we have that $d_pf(T_pM)+T_{f(p)}W=T_{f(p)}V$. 
\end{definition}
In particular, if $V=\R$ and $W=\{0\}$, then $f\transv 0$ if and only if $0$ is a regular value of $f$ as defined in \cref{def:regval}. A standard consequence of transversality (see \cite[Theorem 4.2]{Hirsch}) is that if $f\transv W$, then $f^{-1}(W)$ is a $\mC^r$ submanifold of $M$ having the same codimension as $W$. We will need to consider the case when $W=0_V$ is the zero section of a vector bundle $V\to M$. By passing to local charts this setting can be easily reduced to that of a $\mC^r$ map $f\colon \R^m\to \R^k$, with $W=\{0\}$. In this setting, Sard's theorem \cite{sard}, states that the set $\{u\in \R^k:f+u\transv 0\}$ has full Lebesgue measure in $\R^k$, provided that $r\ge 1+\max\{m-k,0\}$. We will use the following generalization proved in \cite{dtgrf}.
\begin{lemma}\label{lem:Ptransv}
Let $M$ be a $\mC^r$ manifold of dimension $\dim M=m$, with $\de M=\emptyset$, possibly not compact. Let $V\to M$ be a $\mC^r$ vector bundle of rank $k$ and denote by $0_V$ the zero section. Let $X\randin \Gamma^r(V)$ be a Gaussian random section of class $\mC^r$ such that for all $p\in M$, the Gaussian random vector $X(p)\randin V_p$ has a density on $V_p$. If $r\ge 1+\max\{m-k,0\}$, then $\PP\{X\transv 0_V\}=1$.
\end{lemma}
\begin{proof}
By passing to local charts, we can reduce to the case of a Gaussian field $X\randin \mC^r(M,\R^k)$. If $r=\infty$, the statement follows directly from \cite[Theorem 7]{dtgrf}. However, the proof of \cite[Theorem 7]{dtgrf} employs Sard's theorem and goes through without any modification for finite $r$ such that $r\ge 1+\max\{d-k,0\}$.
\end{proof}
\subsection{Degenerations are Morse}\label{sec:degMorse}
Let us introduce the notion of Morse critical value on a manifold with boundary, following \cite{GoreskyMacPherson}.
\begin{definition}\label{def:morse}
Let $(M,g)$ be a $\mC^2$ Riemannian manifold with boundary and let $f\colon M\to\R$ be a $\mC^2$ function and $r\in \R$ be a critical value. If $\de M=\emptyset$, we say that $r\in \R$ is a \emph{Morse critical value} of $f$ if for any point $p$ such that $f(p)=r$ and $d_pf=0$, the Hessian $\H_p f :T_pM\times T_pM\to \R$ is non-degenerate. In general, we say that $r\in \R$ is a \emph{Morse critical value} of $f$ if it is a Morse critical value of $f|_{\inter M}$ and $f|_{\de M}$ and if $0\notin df(\de M)$. The mapping $f$ is a \emph{Morse function} if every critical value of $f$ is a Morse critical value. From now on, we define $\m U'$ and $\m W'$ as in \cref{eq:Uprime}.
\end{definition}
The definition of Morse critical value corresponds with that of \cite{GoreskyMacPherson} of \textit{ Morse functions on stratified spaces}, when the strata of $M$ are $\inter M$ and $\de M$. Notice that this definition does not depend on the metric.
In the Gaussian case, in the same setting as that of \cref{lem:bulinskaya}, we have the following stronger statement. 
\begin{lemma}\label{lem:asMorse}
Let $X\randin \mC^2(M)$ be Gaussian and assume that for any $p\in M$ the Gaussian random vector 
$d_pX\randin T_p^*M$ has a density, then $\PP\{X\text{ is Morse}\}=1.$
\end{lemma}
\begin{proof}
When, $\de M=\emptyset$, the function $f$ is Morse if and only if the differential $df$ is transverse to the zero section of $T^*M$. Applying \cref{lem:Ptransv} to $dX$ 
one obtains that $\PP\{X\text{ is Morse}\}=1.$ When there is boundary $\de M$, by repeating this argument twice: one for $X|_{\inter M}$ and one for $X|_{\de M}$, one obtains that $\PP\{X\text{ is Morse}\}=\PP\{0\notin dX(\de M)\}.$ Consider $(dX)|_{\de M}$ as a section of the vector bundle $TM|_{\de M}$; one more application of \cref{lem:Ptransv} to that section, shows that $\de M\cap dX^{-1}(0)=\emptyset$ almost surely, thus we conclude. 
\end{proof}
In order to discriminate between typical segments and degenerate ones, we introduce the notion of transverse curve. The heuristic behind, confirmed by \cref{thm:transcurve} below, is that a random segment $[X,X+h]$ is almost surely a transverse curve.
\begin{definition}\label{def:pseudotransv}
Let $\m U'$ and $\m W'$ be defined as in \cref{eq:Uprime}. Let $f\in \m U'\subset \m W$. We define the \emph{tangent space to $\m W$ at $f$} as  
\be T_f\m{W}:=\{\dot{c}(0)|\ c\colon \R\to \mC^2 (M) \text{ is a $\mC^1$ curve with $c(0)=f$ and $c(\R)\subset \m W$.}\}.
\ee
Let $I\subset \R$ be an interval. We will say that a curve $c\colon I\to \mC^2 (M)$ is \emph{transverse} to $\m{W}$ at $t\in I$, and write $c\transv_t \m{W}$, if and only if: 
$c(t)\in \m{W}$ $\implies$ $c(t)\in \m U'$ and $\dot{c}(t)\notin T_{c(t)}\m{W}$. If $c\transv_t \m W$ for all $t\in I$ and if $c(\de I)\cap \m W=\emptyset$, then we write 
\be 
c\transv \m W
\ee
and we say that \emph{$c$ is transverse to $\m W$} and that \emph{$c$ is a transverse curve}. When $c(t)=f+th$, $t\in [0,1]$, we will identify $c$ with its image, the segment $[f,f+h]$.
\end{definition}

In a reasonable finite-dimensional setting, one has that $\m{W}\subset E$ is a stratified hypersurface having $\m{W}'$ as singular locus. In this situation, a curve $c$ in $E$ is transverse to $\m{W}$ in the usual sense if and only if $c\transv \m{W}$ in the sense of \cref{def:pseudotransv}. In fact, although we do not need to be that precise, even in this infinite dimensional case, we could argue that $\m{W}$ is a submanifold.

The following proposition clarifies the role of transverse curves and their close relation with Morse functions. Thanks to \cref{cor:TransvToMorse}, based on point (4) of \cref{prop: transvprop} below, we will be able to pass from a generic transverse curve or segment $\{c(t):t\in I\}$ to one of the form $\{T-t\colon t\in I\}$, where $T$ is a Morse function, while preserving the geometry of the zero sets, see \cref{cor:TransvToMorse}. 
\begin{proposition}\label{prop: transvprop}
Let $c\colon I\to \mC^2(M)$ be a $\mC^1$ curve defined on an open interval $I\subset \R$ and denote $c(t):=f_t$. The following statements are equivalent.
\begin{enumerate}[\rm (1)]
\item $c\transv \m{W};$
\item 
$c(I)\cap \m{W}\subset \m U'$ and $j^1_pf_t=0\implies \frac{d}{dt} f_t(p)\neq 0$, for all $t\in I$ and $p\in M$;
\item the function $j^1c\colon M\times I\to J^1(M)$ defined as $(p,t)\mapsto j^1_p f_t$, is transverse to the zero section $0_J:=\{0\}\times 0_{T^*M}\subset J^1(M)$;
\item (If moreover, $c$ is of class $\mC^2$) the equation $f_t(p)=0$ is regular on $M\times I$ and defines a $\mC^2$ hypersurface $N\subset M\times I$ with boundary $\de N=\de M\times I$. Moreover, the second projection $\pi_2|_{N}\colon N\to I$ is a Morse function whose critical points $(p,t)$ correspond to the pairs such that $p$ is a critical point of $f_t$ with value $0$.
\end{enumerate}
\end{proposition}
\begin{proof}
Assume that $c(t_0)\in \m W$, then both (1) and (2) imply that $c(t_0)\in \m U'$,  thus that $0$ is a Morse critical value of $f_{t_0}=c(t_0)$. Let $p_1,\dots,p_k$ be the critical points of $f_{t_0}$. 
By applying the implicit function theorem to the equation $ d_{p}f_t=0$ in a neighborhood of $(p_i,t_0)$, we deduce that there exist $k$ curves $p_i(\cdot)$ of critical points of $c(\cdot)$ in $M$, that is, such that $d_{p_i(t)}f_t=0$ for all $t$ close to $t_0$. Then, we have that $c(t)\in\m W$ if and only if there exists an $i\in \{1,\dots,k\}$ such that $f_t(p_i(t))=0$. By differentiating the latter equation at $t=t_0$, we get that $0=\frac{d}{dt}f_t(p_i(t))=\frac{d}{dt} {f}_{t}(p_i(t_0))$. 
Since the latter condition depends only on the value of $c(t_0)$ and $\dot c (t_0)$ at the point $t_0$, it follows that the curve $c(t)$ is tranverse to $\m{W}$ at $t_0$ if and only if $\frac{d}{dt}|_0 f_t(p_i(t_0)\neq 0$ for all $i\in \{1,\dots,k\}$. This argument proves that (1) is equivalent to (2).

By definition, point (3) means that if $j^1_pf_t=0$, then $\frac{d}{dt}f_t(p)\neq 0$ and $(\nabla df_t)_p:T_pM\to T^*_pM$ is surjective, i.e., $\H_pf_t$ is non degenerate. The second condition implies that $f_t\in \m{W}\smallsetminus\m{W}'$. Therefore, (2) and (3) are equivalent.

$(2)\implies (4):$ By $(2)$, we have that the differential of the function $(p,t)\mapsto f_t(p)$ cannot vanish at a point of $N=\{(p,t)\in M\times I\colon f_t(p)=0\}$, therefore $N$ is a $\mC^2$ hypersurface of $M\times I$. The $\mC^2$ regularity follows from that of the defining equation $c(t)(p)=$, which is of class $\mC^2$ because $c$ is. The tangent space of $N$ at $(p,t)\in N$ is 
\be\label{eq:Wtangent} 
T_{(p,t)}N=\ker \tyu \frac{d}{dt}f_t(p)dt+d_pf_t=0\uyt,
\ee  
so that $dt|_{T_{(p,t)}N}=0$ if and only if $d_pf_t|_N=0$, from which we deduce the characterization of critical points. Given that $N$ has codimension $1$, a point $(p,t)$ is critical for the second projection if and only if $T_{(p,t)}N=T_pM\times \{0\}$. By differentiating twice the equation of $N$ along a curve $(p(s),t(s))\in N$, such that $p(0)=p$, $t(0)=t$ and $v:=(\dot{p},\dot{t})\in T_{(p,t)}N$, we obtain that
\be   
\H_{(p,t)} \tyu \pi_2|_N\uyt(v,v)=\ddot{t}=\tyu\frac{d}{dt}f_t(p)\uyt^{-1}\H_pf_t(\dot{p},\dot{p}).
\ee
It follows that $(p,t)$ is Morse as a critical point of $\pi_2|_N$ if and only if $p$ it is a Morse critical point of $f_t$.
Repeating the previous argument backwards one can prove that $(4)\implies (2)$.
\end{proof}
\begin{corollary}\label{cor:everytransv}
Every point $f\in \m U'$ lies on a tranverse curve. Thus, $\m U'=\{c(0)\colon c\transv \m W \text{ and } c(0)\notin \m U\}$.
\end{corollary}
\begin{proof}
By point (2), we have that the curve defined as $c(t)=f+tX$ is almost surely a transverse curve on a small enough (random) interval $t\in (-\e(X),\e(X))$.
\end{proof}
We are ready to state and prove the main theorem of this section, having as an immediate consequence that almost every segment $[X,X+h]$ is transverse to $\m W$.
\begin{theorem}\label{thm:transcurve}
Let \cref{ass:1} prevail. Let $c:[0,1]\to \mC^2(M)$ be any $\mC^2$ curve. Then $\PP\{c+X\transv \m{W}\}=1$.
\end{theorem}
\begin{proof}
We have that the assignement $\psi_X(p,t)=j^1_p(X+c(t))$ defines a $\mC^1$ Gaussian random section $\psi_X \colon M\times \R\to J^1(M)$ of the vector bundle $V\to \R\times M$ obtained as a trivial extension of $J^1(M)\to M$. By hypothesis $\psi_X(p,t)$ has a density on $V_{(p,t)}$. Notice that the rank of $V$ is $m+1$, which is equal to the dimension of $M\times \R$, thus the $\mC^1$ regularity of $\psi_X$ is sufficient to apply \cref{lem:Ptransv}, from which it follows that $\psi_X\transv 0_V$ almost surely. The latter is equivalent to $X+c\transv \m W$, by point (3) of \cref{prop: transvprop}. Finally, we have that $\PP\{c(i)+X\cap \m W=\emptyset\}=1$ by \cref{lem:bulinskaya}.
\end{proof}

\begin{corollary}\label{cor:TransvToMorse}
Let \cref{ass:1} prevail. Let $\{f_t\colon t\in [0,1]\}$ be a $\mC^2$ transverse curve. Then, there exists a $\mC^2$ compact $m$-dimensional Riemannian manifold $N$ with boundary 
and a Morse function $T\colon N\to \R$, such that $T^{-1}(t)$ is $\mC^2$ isometric to $f_t^{-1}(0)$ for all $t\in [0,1]$, via the restriction of a $\mC^2$ map $\pi\colon N\to M$.\footnote{Then, $\pi^{-1}(p)$ is in bijection with $\{t\colon p\in f_t^{-1}(0)\}$.}
\end{corollary}
\begin{proof}
The statement almost follows from point (4) of \cref{prop: transvprop}, except for the fact that the manifold $N$ so defined can have corners at $N\cap \de M \times \{0,1\}$. We can extend $N$ to a manifold without corners as follows.

Consider the segments $[1,f_0]$ and $[f_1,1]$ in $\mC^2(M)$. By \cref{thm:transcurve} there are realizations $h_0,h_1$ of $X$, such that $[1,f_0]+h_0$ and $[f_1,1]+h_1$ are transverse curves. Moreover, we can choose $h_i$ so small (in $\mC^2(M)$) that, for $i\in \{0,1\}$, $h_i+1> 0$ and such that $h_i+f_i$ belongs to a ball $B_i\subset \m U$ around $f_i$. This is because $\m U$ is an open subset, see \cref{prop:neat}. For the same reason, it is possible to extend the transverse curve $(f_t)_{t\in [0,1]}$ to a $\mC^2$ curve defined for all $t\in [-2,3]$, in such a way that $\{f_t: t\in [-2,-1]\}=[1,f_0]+h_0$; $\{f_t: t\in [-1,0]\}\subset B_0$; $\{f_t: t\in [1,2]\}\subset B_1$ and $\{f_t: t\in [2,3]\}=[f_1,1]+h_1$. Such a curve is automatically tranverse, since the two new pieces are contained in $\m U$. Moreover, $f_{-2}^{-1}(0)=f_{3}^{-1}(0)=\emptyset$.

Let $N:=\{(p,t)\in M\times [-2,3]\colon f_t(p)=0\}$ and define $T:=\pi_2|_N$ and $\pi:=\pi_1|_N$, where $\pi_1$ and $\pi_2$ are the projections on the first and second factors of $M\times \R$, respectively. Then, $N$ is compact and, by \cref{prop: transvprop}, $N\subset M\times (-2,3)$ is a $\mC^2$ neat hypersurface with boundary $\de N=\de M\times (-2,3)$ and no corners. 

By construction, we have that $T^{-1}(t)=f_t^{-1}(0)\times \{0\}$, for all $t\in [0,1]$, so that $\pi$ restricts to a diffeomorphism on it. Let us endow $N$ with the Riemannian metric induced by the inclusion in the product space $M\times \R$, when the latter is given the product Riemannian metric. Then the metric induced on $T^{-1}(t)$ by the inclusion in $N$ coincides with that induced by the inclusion in $M\times \R$, therefore the restriction of $\pi$ is also an isometry of $T^{-1}(t)$ onto $f_t^{-1}(0)$.
\end{proof}
\subsection{Simultaneous critical points have the same index}\label{sec:bestia}
Given $f\in \mC^2(M)$, we denote by $\CZ_f$ the set of \emph{critical zeroes} of $f$. When $\de M=\emptyset$, this is the set 
\be 
\CZ_f=\kop p\in M\colon f(p)=0,\ d_pf=0\pok,
\ee 
and $\CZ_f=\CZ_{f|_{\inter M}}\cup \CZ_{f|_{\de M}}$ in the general case. 
We will prove in \cref{sec:lowd2} that the function $t\mapsto V(f+th)$ has a certain behavior in a neighborhood of any $t_0$ such that $f+t_0h\in \m W$. This as long as the segment $[f,f+h]$ is transverse and if the critical zeroes of $f+t_0h$ do not compensate each other, see \cref{prop:3pt}. The following result ensures that such an event has nonzero probability only under very restricting deterministic assumptions. 
\begin{theorem}\label{thm:bestia}
Let \cref{ass:1} prevail. Fix an arbitrary element $e\in E$ of the support $E\subset \mC^2(M)$ of $X$. Then, almost surely, 
we have the following: if  $C=\CZ_f\neq \emptyset$ is the set of critical zeroes of $f=X+t_0e$ (which implies that $X+t_0e\notin \m U$), for some $t_0\in \R$, then: 
\begin{enumerate}[\rm (1)]
\item\label{bestia:im} $0$ is a Morse critical value of $f$ ;
\item\label{ie} $e(p)\neq 0$ for all $p\in C$;
\item\label{bestia:is} $C\subset S$, where $S\in\{ \inter M,\de M\}$;
\item\label{bestia:ii} the symmetric forms $e(p)\H_{p}f|_S$ all have the same index $\lambda$ for every $p\in C$;
\item\label{ic} for any pair $p,q\in C$, the random vectors $j^1_{p}X|_S, j^1_{q}X|_S$ in $J^1(S,\R)$ are fully correlated.
\end{enumerate}
\end{theorem}
\begin{proof}
The proof is postponed to Appendix \ref{apx:compensation}.
\end{proof}
We will use \cref{thm:bestia} in \cref{sec:sobd2} in conjunction with \cref{prop:3pt} and with the results of \cref{sec:lowd2}, in order to prove \cref{cor:notD12}. Point (iv) of \cref{t:mainintro} follows quite directly from the latter, see the proof of \cref{thm:D12boundary}. We stress the fact that points (i)-(iii) of \cref{t:mainintro} are independent from \cref{thm:bestia}.
\section{Nodal volumes of the levels of a Morse function}\label{sec:levelMorse}
In \cref{sec:degMorse} we discussed how to pass from an arbitrary transverse curve $f_t$ or segment in $\mC^2(M)$, to one of the form $T-t$, where $T$ is a Morse function defined on another manifold with the same dimension, see point (4) of \cref{prop: transvprop}, in a such a way that $f_t^{-1}(0)$ is isometric to $(T-t)^{-1}(0)$ and thus
\be 
V(f_t)=V(T-t).
\ee 
Now, we will focus on the latter case. 
Let $M$ be a $\mC^2$ compact manifold with boundary, endowed with a Riemannian metric of class $\mC^1$. Let $T\in \mC^2(M)$ be a Morse function and assume that $0$ is a critical value of $T$. Let $Z_t:=T^{-1}(t)$ and let  $\f(t):=V(T-t)=\vol^{m-1}(Z_t).$
This section is devoted to determine the Sobolev regularity of the function $\f(t)$, as this will be exploited later in \cref{sec:transverse} to give conditions under which the nodal volume $V$ is ray-absolutely continuous, see (i) of \cref{def:Dspace}.

First, we will reduce the study to a neighborhood of a critical point of $T$ in \cref{sec:localize} and provide a general upper bound for the integration over the level sets of a Morse function, see \cref{thm:upperbound}. 
This is the main theorem of this section and we consider it to be of independent interest. Using the latter result, in \cref{sec:upasy} we deduce the behavior of $\f'$ near $t=0$, proving \cref{lem:upasyMorse}, also including the case of manifolds with boundary. In \cref{sec:lowd2}, we will obtain lower bounds in the two dimensional case, implying that $\f'$ is not square integrable in this case. This will be an important ingredient of the proof of \cref{cor:notD12} and thus of point (iv) of \cref{t:mainintro}.

\subsection{Localization}\label{sec:localize}
\subsubsection{Morse coordinates}\label{sec:Mors}
Morse Lemma gurantees that a Morse function is always locally equivalent, up to a change of coordinates, to its second order Taylor polynomial, near a critical point. This result is standard when $T\in \mC^{2+r}$, however in most of the literature, the statements provides a change of coordinates only of class $\mC^r$. For us, $r=0$, but it will be convenient to have a $\mC^1$ change of coordinates. This is possible due to the following theorem, proved by Bromberg and L\'{o}pez de Medrano \cite{bromberg1993lemme}, improving a result of Kuiper \cite{kuiper1970cr}. We report it in full generality, for the reader's interest.
\begin{theorem}[$\mC^r$ Morse Lemma]\label{lem:CrMorse}
Let $\m H$ be a separable Hilbert space and let $f\colon \m H\to \R$ be a function of class $\mC^r$, $r\ge 1$, such that $f(0)=0$ and $d_0f=0$. The following two statements are equivalent:
\begin{enumerate}
\item There exist two neighborhoods $O_1,O_2$ of $0$ in $\m H$, a $\mC^r$ diffeomorphism $\phi\colon O_1\to O_2$ and a quadratic form $Q:\m H\to \R$ such that $f\circ \phi=Q$;
\item $f$ admits the differential of order $r+1$ at $0$ and the second order differential $\H_0f$ is non-degenerate.
\end{enumerate}
\end{theorem}
\begin{proof}
    \cite[Lemma de Morse $\mC^r$]{bromberg1993lemme}.
\end{proof}\subsubsection{Local parametrization of the level set} 
Let $q_0\in \inter{M}$ be a critical point of $T$, with $T(q_0)=0$ and assume that there are no other critical points in $T^{-1}(0)$. By the $\mC^1$ version of \cref{lem:CrMorse} above, applied with $\m H=\R^m$, there exists a neighborhood $O\subset M$ of $q_0$ and a $\mC^1$ diffeomorphism $x=(x_+,x_-):O\to x(O)\subset \R^{n_+}\times \R^{n_-}$ such that $T(x)=|x_+|^2-|x_-|^2$. 
We may assume that $\overline{O}$ has a $\mC^2$ boundary and that 
\be 
B_{\sqrt{2\e}}^{n_+}\times B_{\sqrt{\e}}^{n_-}\subset x(O)\subset B_{2\sqrt{2\e}}^{n_+}\times B_{\sqrt{\e}}^{n_-}, 
\ee 
where $B_s^k$ is the open ball of radius $s$ in $\R^k$. Therefore, 
we can parametrize $Z_t\cap O$ as follows:
\be\label{eq:parametr}   
Z_t\cap O=x^{-1}\kop (u\sqrt{r^2+t},vr):u\in S^{n_+-1}, v\in S^{n_--1}, r\in [0,\sqrt{\e})\pok, \quad \forall t\in [0,\e),
\ee 
if $n_+,n_-\ge 1$. When $n_-=0,$ the function has a minimum at level $t=0$, so that $Z_t\cap O=x^{-1}(\sqrt{t}S^m)$, for all $t\ge 0$. When $n_+=0,$ we have the opposite situation: the function has a maximum at $t=0,$ which means that $Z_t\cap O=\emptyset$ for $t>0$ and $Z_t\cap O=\{q_0\}$ if $t=0$. 

We will start by considering the case when $n_+,n_-\ge 1$. Let us fix $\e>0$ and define the map $\psi_t: S^{n_+-1}\times S^{n_--1}\times [0,\e^{\frac12})\to \R^{m}$, such that 
\be\label{eq:psi} 
\psi_t(u,v,r)=(u\sqrt{r^2+t},vr).
\ee 
From now on, we will divide $Z_t$ in two parts:
\be 
Z_t^1:=Z_t\cap O \quad \text{ and } \quad Z_t^2:=Z_t\cap M\smallsetminus O
\ee 
We have that $Z_t^2$ is a $\mC^2$ neat hypersurface of $M_O:=M\smallsetminus O$, for all $t\in [0,\e^2)$. Indeed, $T$ has no critical point in its interior $M\smallsetminus \overline{O}$ and, as can be seen from the expression in the coordinates $x_+,x_-$, we have that $T|_{\de O}$ has no critical point in $\de M\cap T^{-1}([0,\e^2))$. Therefore, by \cref{thm:firstvar}, its volume $\f_2(t):=\vol^{m-1}(Z_t^2)$ is a $\mC^1$ function of $t$. Indeed, denoting $V_O$ and $\m U_O$ be the nodal volume functional and the set of regular functions relative to the manifold $M_O$, as in \cref{thm:firstvar}, then, $(T-t)|_{M_O}$ is a $\mC^1$ curve contained in $\m U_O$ and hence $\f_2(t)=V_O(T-t|_{M_O})$.
This allows us to focus on the function $\f_1(t)=\vol^{m-1}(Z^1_t)=\f(t)-\f_2(t)$, which we can study within the coordinate chart.
\subsubsection{Riemannian vs Eucliden volume}
The Morse coordinate provide an explicit coordinate expression for $T$, but the metric might not be Euclidean, so we will need to have a control on the change of $k$-volume elements. 
In the $\mC^1$ coordinate chart $x\colon O\to \R^m$, the Riemannian metric is represented by a matrix $g(x)$, depending continuously on $x$. Let $A:\R^k\to T_pM$ be a linear injection. Then the \emph{Jacobian} of $A$ is $JA:=\sqrt{\det\tyu A^Tg(x)A\uyt}$. 
Let us consider the function 
\be 
{\mathcal{J}im}_{g}(x,A):=\frac{\sqrt{\det\tyu A^Tg(x)A\uyt}}{\sqrt{\det\tyu A^TA\uyt}},
\ee
defined for all $A$ injective. Note that for $G\in GL(k),$ we have that ${\m{J}im}_g(p,A)={\m{J}im}_g(p,A G)$, therefore ${\m{J}im}_g$ depends only on $g$ and on the image of $A,$ which is a $k$-dimensional linear subspace $V\subset T_pM.$ This means that ${\m{J}im}_g$ defines a function on the Grassmannian of $k$-planes in $T_pM$. In this paper, we only care about the case $k=m-1$, for which we give the following definition.
\begin{definition}
Let $\mathsf{P}(T^*M):=\{\ell \subset T_p^*M \colon p\in M, \ \ell \text{ is a line containing the origin}\}$ be the projectivized cotangent bundle. We define
\be
\m{J}_g : \mathsf{P}(T^*M) \to (0,+\infty), \qquad
\m{J}_g(\ell):={\m{J}im}_g(p,A),
\ee
where $A\colon \R^{m-1}\to T_pM$ is a linear injection with $A(\R^m)=\ell^\perp.$
\end{definition}
\begin{lemma}
    $\m{J}_g$ is continuous.
\end{lemma}
\begin{proof}
Observe that ${\m{J}im}_g(p,A)$ is continuous on the open set $\{(p,A): A\text{ is injective}\}$.
\end{proof}
\subsubsection{The volume density in a Morse chart: case $n_-,n_+\ge 1$}
We apply the previous discussion to control the $(m-1)$-volume element of the parametrization $\psi$, defined as in \cref{eq:psi}.

Let $g(x)$ be the matrix of the Riemannian metric in the chart $x\colon O\to \R^m$. 
Let 
\bega   
A_t(u,v,r)&:=\begin{pmatrix} 
d\psi_t(\dot{r}) & d\psi_t(\dot{u}_1) & \dots & d\psi_t(\dot{u}_{n_+-1}) &d\psi_t(\dot{v}_1) & \dots & d\psi_t(\dot{v}_{n_--1})
\end{pmatrix}
\\
&=
\begin{pmatrix} 
ur(r^2+t)^{-\frac12} & (r^2+t)^{\frac12}\dot{u}_1 & \dots & (r^2+t)^{\frac12}\dot{u}_{n_+-1} &0 & \dots & 0
\\
v & 0 & \dots & 0 &r\dot{v}_1 & \dots & r\dot{v}_{n_--1} 
\end{pmatrix}
\eega
be the $m\times m-1$ matrix of $d_{(u,v,r)}\psi_t$, where $(\dot{u_i})_i$ is an orthonormal basis of $T_uS^{n_+-1}=u^\perp$ and $(\dot{v_j})_j$ is an orthonormal basis of $T_vS^{n_--1}=v^\perp.$ Therefore,
\bega\label{eq:AA} 
\sqrt{\det\tyu A_t(u,v,r)^TA_t(u,v,r)\uyt}&=\sqrt{\tyu\frac{r^2}{(r^2+t)}+1\uyt}(r^2+t)^{\frac{n_+-1}{2}}r^{n_--1}
\\
&=
(2r^2+t)^{\frac12}(r^2+t)^{\frac{n_+-2}{2}}r^{n_--1}
\\
&=
t^{\frac{m-2}{2}}(2s^2+1)^{\frac12}(s^2+1)^{\frac{n_+-2}{2}}s^{n_--1},
\eega
where in the last line we define $s$ such that $r=s\sqrt{t}.$
Define 
\be 
\ell_t(u,v,r):=\R (u\sqrt{r^2+t},-vr)\in \mathsf{P}(\R^m)
\ee 
as the line generated by the differential of $T$ at $\psi_t:=\psi_t(u,v,r)$. Then the image of $A_t:=A_t(u,v,r)$, that is the tangent space to $Z_t$ at $\psi_t(u,v,r)$, is the orthogonal to $\ell_t(u,v,r)$:
\be 
\sqrt{\det(A_t^Tg(\psi_t)A_t)}=\sqrt{\det(A_t^TA_t)} \m{J}_g(\ell_t(u,v,r))
\ee 
for all $x\in O$ and $t.$ Notice that the above function is continuous at $t=0$.

\subsubsection{Integration over $Z_t^1$}
Let $S_+:=S^{n_+-1}$ and $S_-:=S^{n_--1}$. For all $t\in[0,\e^2)$, the embedding $\psi_t\colon S_+\times S_-\times (0,\e)\to \R^m$ parametrizes the submanifold $\Sigma_t=x(Z^1_t)\smallsetminus (\sqrt{t}S_+\times \{0\})$, which has full measure in $x(Z^1_t)$. The downside is that $\Sigma_t$ is not compact. 
For any positive measurable function $h\colon \R^m\to\R$, we have 
\bega\label{eq:hSigmat} 
&\int_{Z^1_t}h\circ x dZ^1_t
=
\int_{\Sigma_t}hd\Sigma_t
\\
&=
\int_{S_+}\int_{S_-}h(\psi_t)\int_{0}^\e\sqrt{\det\tyu A_t^Tg(\psi_t)A_t\uyt}drdS_+(u)dS_-(v)
\\ 
&=\int_{0}^{\sqrt{\e}} \int_{S_+}\int_{S_-}h(\psi_t)\sqrt{\det\tyu A_t^TA_t\uyt}\m{J}_g(\ell_t)drdS_+(u)dS_-(v),
\eega
where we denoted $A_t=A_t(u,v,r)$, $\psi_t=\psi_t(u,v,r)$ and $\ell_t=\ell_t(u,v,r)$, for brevity.
Thus, recalling \cref{eq:AA} and performing the change of variable $s=r\sqrt{t}$, we get the following.
\begin{lemma}
\be\label{eq:intZ1}
\int_{Z^1_t}(h\circ x) dZ^1_t
=t^{\frac{m-1}{2}}\int_{0}^\frac{\sqrt{\e}}{\sqrt{t}} \hat h(t,s)
(2s^2+1)^{\frac12}(s^2+1)^{\frac{n_+-2}{2}}s^{n_--1}
ds,
\ee
where $\hat h\colon (0,\e)\times (0,+\infty)\to \R$ is the function 
\be\label{eq:hhat} 
\hat h(t,s)=\int_{S_+}\int_{S_-}\m{J}_g(\ell_t(u,v,s\sqrt t))h(\psi_t(u,v,s\sqrt{t}))dS_+(u)dS_-(v).
\ee
\end{lemma}
\subsubsection{The volume density in a Morse chart: case $n_-=0$ and $n_+=m$}
This case is much simpler than the previous one in that $Z_t^1=O\cap Z_t=x^{-1}(\sqrt{t}S^m)$ and, for small $t>0,$ such set is a closed embedded sphere entirely contained in the open set $O.$ The integral of a measurable function $h\colon \R^m\to \R$ over such sphere writes as
\bega\label{eq:nmeno0}
\int_{Z^1_t}(h\circ x)dZ^1_t
&=
t^{\frac{m-1}{2}}\int_{S^{m-1}}\m{J}_g(\R \sqrt t x)h(\sqrt{t}x)dS^{m-1}(x).
\eega
\subsubsection{A general upper bound}
The following theorem describes the integration along level sets of Morse functions, i.e., we study the behavior at $t\to 0$, of the quantity 
\be\label{eq:inth} 
\int_{\{T=t_0+t\}} h(p)\mathcal{H}^{m-1}(dp)
\ee 
where $h$ is a measurable function on $M$. While we believe that this result is of independent interest, our main purpose is to apply it to the function $h={\tDelta T}{\|dT\|^{-2}}$ appearing in the formula (\cref{eq:firstvar}) for the derivative of $V$, which explodes at critical points, see \cref{lem:Morsepiecewisecont}. Thus, we consider functions $h$ with a controlled behavior near the critical set of $T$, measured by the Riemannian distance function (see \cite[Section 2]{leeriemann}) of $(M,g)$.

Given $C\subset M$ a finite subset of a Riemannian manifold $(M,g)$, we denote by $\mathrm{dist}_g(p,C)=\min\{\mathrm{dist}_g(p,q)\colon q\in C\}$ the Riemannian distance from $p\in M$ to $C$. A consequence of Gauss Lemma \cite[Theorem 6.9]{leeriemann} is that for $p$ in the domain of a small enough coordinate chart $x\colon O\to \R^m$ around a point $q=x^{-1}(0)\in C$, we have that 
\be 
\frac 1A |x(p)|\le \mathrm{dist}_g(p,C)=\mathrm{dist}_g(p,q) \le A |x(p)| 
\ee
for some constant $A>0$, see \cite[Corollary 6.12]{leeriemann}. 
\begin{theorem}\label{thm:upperbound}
Let $M$ be a $\mC^2$ compact Riemannian manifold of dimension $m\ge 2$, with boundary $\de M$. Let $g$ be a $\mC^1$ Riemannian metric. Let $T\in \mC^2(M)$ be Morse and let $t_0\in\R$ be a Morse critical level of $T,$ with critical set $C\subset 
\inter{M}$ and assume that there are no other critical values in the interval $[t_0,t_0+\e]$, for some $\e>0$. 
Then, there exists a constant $A>0$ such that for any measurable function $h:M\smallsetminus C\to \R$ that satisfy 
\be T(x)\ge t_0 \implies |h(x)|\le \mathrm{dist}(x,C)^{-k}\ee 
for some $k\in \N$, we have that for all $t\in (0,\e]$
    \be 
\int_{\{T=t_0 +t\}} 
|h(p)|\mathcal{H}^{m-1}(dp)\le A\left\{
\begin{aligned}
    1,\quad &\text{if $k\le m-2$;}
    \\
    -\log t,\quad &\text{if $k= m-1$;}
    \\
    \frac{1}{t^{\frac{1+k-m}{2}}},\quad &\text{if $k\ge m$.}
\end{aligned}
\right.
    \ee
    Moreover, if $k\le m-2$, and $h$ is continuous, then the mapping $t\mapsto \int_{\{T=t_0 +t\}} h(x)d\mathcal{H}^{m-1}(x)$ is well defined and continuous on $[0,\e]$. Moreover, the constant $A$ can be chosen uniformly for all metrics $\tilde g$ in a $\mC^1$ neighborhood of $g$. 
\end{theorem}
\begin{proof}
We can assume that $t_0=0$, so that it is enough to show that the same property holds for the integral over $Z_t^1$ in \eqref{eq:intZ1} and \cref{eq:nmeno0}, when 
\be 
h(x)\le |x|^{-k}= \left\{\begin{aligned}
    \qwe(2s^2+1)t\ewq^{-\frac{k}{2}}, 
    \quad &\text{in the case $n_+,n_-\ge 1$;}\\
    t^{-\frac{k}{2}},
    \quad &\text{in the case $n_-=0$.}
\end{aligned}\right.
\ee 
Indeed, there is at most a finite number of critical points so the integral is the sum of the contributions of each critical point. Around the points of local maximum $\vol^{m-1}(Z^1_t)=0$ for all $t\ge 0.$ 

Let us start with the case $n_+,n_-\ge 1.$ The function $\hat h$ in \cref{eq:hhat} satisfies the bound $\hat h(t,s)\le C\qwe(2s^2+1)t\ewq^{-\frac{k}{2}},$ so that we have 
\bega
\int_{Z^1_t}hdZ^1_t
&\le t^{\frac{m-1-k}{2}}\int_{0}^\frac{\e}{\sqrt{t}} (2s^2+1)^{\frac{1-k}{2}}(s^2+1)^{\frac{n_+-2}{2}}s^{n_--1}
ds,
\\
&\le t^{\frac{m-1-k}{2}}\qwe \int_{0}^1 
s^{n_--1}ds+2^{\frac{1-k}{2}}\int_1^{\frac{\e}{\sqrt{t}}}s^{m-k-2}
ds
\ewq
\\
&= t^{\frac{m-1-k}{2}}\qwe 1+\frac{2^{\frac{1-k}{2}}}{m-k-1}\tyu \frac {\e^2}t \uyt ^{\frac{m-k-1}{2}}.
\ewq
\eega
This satisfies the bound in the thesis, except for the case when $m=k+1,$ when the last equality is false and instead we have 
\be
\int_{Z^1_t}hdZ^1_t\le\dots = t^{0}\qwe 1+2^{\frac{1-k}{2}}\log \tyu  \frac{\e}{\sqrt{t}} \uyt
\ewq=O\tyu |\log t| \uyt.
\ee
In the case $n_-=0$ and $n_+=m$, \cref{eq:nmeno0} gives
\be 
\int_{Z^1_t}hdZ^1_t
\le 
t^{\frac{m-1}{2}}\int_{S^{m-1}}Ct^{-\frac k2}dS^{m-1}(x)\le C' t^{\frac{m-1-k}{2}},
\ee
for all $k$. In case $m=k-1$, this bound implies the thesis, since $t^{\frac{m-1-k}{2}}=1\le |\log t|$ for $t\in [0,\e)$. 

The last case to consider is when $n_-=m$ and $n_+=0$. In this case $\{T=\lambda + t\}$ is contained in the complement of a neighborhood of the critical point, for all $t>0$, thus the integrand is uniformly bounded for $t>0$. 

Observe that the function $\psi_t:S_+\times S_-\times (0,\e)$ is a $\mC^1$ parametrization of the submanifold $\Sigma_t:=Z_t^1\smallsetminus x^{-1}(\sqrt{t}S_+\times\{0\})$ for all $t\in [0,\e]$, including $t=0$ and $\int_{\Sigma_t^1}=\int_{Z_t^1}$. therefore the integrand in \cref{eq:hSigmat} converges almost everywhere as $t\to 0^+$ to the one corresponding to $\int_{\Sigma_0^1}h dZ$. The argument used in the previous part of the proof, now proves the dominated convergence:  $\int_{\Sigma_t^1}h dZ\to \int_{\Sigma_0^1}h dZ$.
\end{proof}
\subsection{Asymptotics of the nodal volume}\label{sec:upasy}
We consider the function $\f(t)=V(T-t)$, expressing the $m-1$ dimensional volume of $T^{-1}(t)$. Let us assume that $0$ is the only critical value of the Morse function $T\colon M\to \R$ contained in $I=[-\e ,\e ]$. The following lemma resumes what we know so far about the function $\f$.
\begin{lemma}\label{lem:Morsepiecewisecont}
One has that $\f\in \mC^0(I)\cap \mC^1(I\smallsetminus \{0\})$ and
\be\label{eq:fiprime} 
\f'(t)=\int_{Z_t}\frac{\tilde{\Delta}T}{\|dT\|^2}dZ_t.
\ee 
\end{lemma}
\begin{proof}
By \cref{thm:firstvar}, we have that for $t\neq 0$, $\f$ is continuously differentiable in a neighborhood of $t$ with derivative $\f'(t)=\langle d_{T-t}V,-1\rangle$. \cref{thm:upperbound}, with $h=1$, $k=0\le m-2$ implies that $\f$ is continuous at $0$.
\end{proof}
\subsubsection{Local expression of the derivative}\label{subsub:locexp} 
The integrand in \cref{eq:fiprime} plays the role of the function $h$ in \cref{thm:upperbound}. Using the local expression for $T$ provided by the Morse coordinate, we show that $\cref{thm:upperbound}$ can be applied with $k=2$, thanks to \cref{lem:localexpr}.
\begin{lemma}\label{lem:localexpr}
    There exists a positive constant $C>0$ such that 
    \be
\frac{|\tilde{\Delta} T(x)|}{\|d_xT\|^2}\le C\frac{1}{|x|^2}.
    \ee
\end{lemma}
\begin{proof}
In the Morse chart $x\colon O\to U$, we have that $H:=\frac12\hess(T)_0$ is:
\be  
H=\begin{pmatrix}
    \mathbbm{1}_{n_+} & 0
    \\ 0 & -\mathbbm{1}_{n_-}
\end{pmatrix}
\ee
indeed $T(x)=x^THx=(x_+)^2-(x_-)^2$. 
Hence, we have that $d_xT=2(Hx)^T$ and that $\grad T(x)=2g_x^{-1}Hx.$

Notice that the bilinear form $\hess(T)_0(v,w)=\langle \nabla_v dT)_0,w\rangle$ does not depend on the metric since $d_0T=0,$ while in general we have: 
\be
\hess(T)_x= 2H-\tyu \Gamma^i_{a,b}(x)(2Hx)_i\uyt dx^a\otimes dx^b=H+O(|x|),
\ee  
where $\Gamma^i_{a,b}=\langle dx^i, \nabla_{\de_a}\de_b\rangle\colon O\to \R$ are the Christoffel symbols of the Levi-Civita connection $\nabla$ (see \cite{leeriemann}). Since the are defined by an expression involving the derivatives of the metric $g$, the above approximation holds as soon as the metric $g$ is of class $\mC^1$. Let us denote $H_x:=\frac12\hess(T)_x$, so that $H_0=H$, and $g:=g_0.$

The Laplacian $\Delta T(x)$ is the trace of the matrix of the bilinear form $\hess(T)_x$ with respect to an orthonormal basis of $g_x$, that is: 
\be
\Delta T(x)=\mathrm{tr}( g_x^{-1}\hess(T)_x)=\mathrm{tr}( 2g_x^{-1}H_x)=\mathrm{tr}( 2g^{-1}H)+o(1).
\ee

We have thus a local formula for the integrand in \eqref{eq:fiprime}
\bega \label{eq:locexpr1}
\frac{\tilde{\Delta} T(x)}{\|d_xT\|^2} &= 
\frac{-\hess(T)(\grad T,\grad T)+\|dT\|^2\Delta T}{\|dT\|^4}\Big|_x
\\
&=\frac{1}{4x^THg_x^{-1}Hx}\tyu2\mathrm{tr}(g_x^{-1}H_x)- \frac{8x^THg_x^{-1}H_xg_x^{-1}Hx}{4x^THg_x^{-1}Hx}\uyt
\\
&=\frac{1+o(|x|^2)}{2x^THg^{-1}Hx}\tyu \mathrm{tr}(g^{-1}H)-\frac{x^THg^{-1}Hg^{-1}Hx}{x^THg^{-1}Hx}+o\tyu 1\uyt\uyt
\\
&\le C \frac{1}{|x|^2},
\eega
for some $C>0$.
\end{proof}
\subsubsection{Upper asymptotics}
From \cref{lem:localexpr} and \cref{thm:upperbound} we deduce the following behavior of $\f'$:
\begin{lemma}\label{lem:upasyMorse}
Let $M$ be a compact $\mC^2$ manifold with boundary. Let $g$ be a Riemannian metric of class $\mC^1$. Let $T\in \mC^2(M)$ be a Morse function (see \cref{def:morse}) and having $0$ as a critical value and define $\f(t)=\m H^{m-1}(T^{-1}(t))$. Then, there is a constant $C>0$ and $\e>0$ such that for all $t\in (-\e,\e)$, we have 
\be\label{eq:upas}
|\f'(t)|\le C\left\{\begin{aligned}
    & 1, \quad \text{if }m\ge 4 ;
    \\
    & |\log t|, \quad \text{if }m= 3;
    \\
    & t^{-\frac12}, \quad \text{if }m= 2  .
\end{aligned}\right.
\ee
\end{lemma}
\begin{proof}
The case $\de M=\emptyset$ is a consequence of \cref{lem:localexpr} and \cref{thm:upperbound} for $k=2$. The general case, is proven in \cref{sec:upasybound}.
\end{proof}
In particular, $\f'(t)$ is integrable and thus $\f$ is an absolutely continuous function, that is, $\f\in W^{1,1}(I)$. Moreover, if $d\ge 4$, then $\f\in \mC^1(I)$ and if $d=3$, then $\f\in W^{1,2}(I)$.
\subsubsection{Upper asymptotics for manifolds with boundary}\label{sec:upasybound} Consider a $\mC^2$ compact Riemannian manifold $M$, of dimension $m$, with boundary $\de M$. Let $T\colon W\to \R$ be a Morse function, $Z_t:=T^{-1}(t)$ and $\f(t):=\vol^{m-1}(Z_t).$

Moreover, let us assume that $I\subset \R$ is a closed interval containing $0$ in its interior and assume that $0$ is the only critical value of $T$ in $I$. 
By \cref{thm:firstvar}, we have that for $t\notin I_0$:
\be\label{eq:fiprimebd} 
\f'(t)=\langle d_{T-t}V,-1\rangle=\int_{Z_t}\frac{\tilde{\Delta}T}{\|dT\|^2}dZ_t-\int_{\de Z_t} \frac{ g(\n,\nu) }{\|d(T|_{\de M})\|} d\de Z_t,
\ee
where $\n\in \Gamma^\infty(TM|_{\de M})$ is the outward normal to the boundary and $\nu =\|dT\|^{-1}\g T$ is the normal to $Z_t$. Let $p\in W$ be a critical zero of $T$. Then, there are two cases: $p\in \inter W$ is a critical point of $T|_{\inter W}$, or $p\in \de W$ is a critical point of $T|_{\de W}$. In the first case, we have that the second integral is continuous at $t=0$, while the first behaves exactly as described in \cref{eq:upas}. Let us consider the second kind of critical points, when $p\in \de W$. Then, we have that the first integrand $\frac{\tDelta T}{\|dT\|^2}$ is bounded in a neighborhood of $p$, hence the first integral is bounded around $p$, while the second behaves as follows:
\be 
\frac{ |g(\n,\nu)| }{\|d(T|_{\de M})\|} \le \mathrm{dist}(x,p)^{-1}.
\ee
This can be easily seen by applying the discussion in \cref{subsub:locexp} to the function $f|_{\de M}$. By \cref{thm:upperbound}, with $k=1$ and $\dim \de W=m-1$, it follows that the second term in \cref{eq:fiprimebd} behaves as in \cref{eq:upas}, when $t\to 0$. We conclude that the upper asymptotics at  \cref{eq:upas} still hold, in the boundary case. This concludes the proof of \cref{lem:upasyMorse}.

\subsection{Lower asymptotics in dimension $m=2$}\label{sec:lowd2}
In the two dimensional case, we we will need to be more precise and complement \cref{lem:upasyMorse} with lower bounds. \cref{prop:3pt} and \cref{lem:notd12bd}, in this subsection, establish that $|\f'(t)|$ behaves as $\frac{1}{\sqrt{t}}$, if $m=2$, provided that some special combinations of critical points are excluded. Indeed, there can be compensation phenomena if $T$ has many critical points with the same value. For this reason, we will be very precise about signs, see \cref{lem:notd12} and the pictures below.

Consider the proof of \cref{lem:localexpr}. A consequence of the spectral theorem\footnote{By the spectral theorem, there exists an orthonormal basis of $g_0$, with respect to which the bilinear form $\frac12\hess(T)_0$ is diagonal. By rescaling the vectors of such basis, we obtain an orthogonal basis of $g_0$ with respect to which the bilinear form $\frac12\hess(T)_0$ is represented by the matrix $H$.} is that in \cref{eq:locexpr1} we can assume that the matrix $g=g_0$ is diagonal, so that
\be 
g^{-1}=\mqty[\a & 0 \\ 0 & \beta],
\quad \text{where} \quad \a=\mqty[\dmat{\a_1,\ddots,\a_{n_+}}] \text{ and }\quad \beta=\mqty[\dmat{\beta_1,\ddots,\beta_{n_-}}]
\ee
for some real numbers $\a_i,\beta_j>0$. Hence, with this choice of coordinates we have
\begin{eqnarray*}
\frac{\tilde{\Delta} T(x)}{\|d_xT\|^2}&=&\frac{(2t)^{-1} }{\sum_{i=1}^{n_+}\a_i(x_+^i)^2+\sum_{j=1}^{n_-}\beta_j(x_-^j)^2} \times \\ &&\quad \times \tyu 
\sum_{i=1}^{n_+}\a_i-\sum_{j=1}^{n_-}\beta_j -\frac{\sum_{i=1}^{n_+}\a_i^2(x_+^i)^2-\sum_{j=1}^{n_-}\beta_j^2(x_-^j)^2}{\sum_{i=1}^{n_+}\a_i(x_+^i)^2+\sum_{j=1}^{n_-}\beta_j(x_-^j)^2}
+o(1)\uyt.
\end{eqnarray*}
Thus, when $x=\psi_t(u,v,s\sqrt{t}),$ we get
\bega
\frac{\tilde{\Delta} T(\psi_t(u,v,s\sqrt{t}))}{\|d_xT\|^2}=\frac{1}{2x^Tg^{-1}x }\Bigg( 
\sum_{i=1}^{n_+}\a_i-\sum_{j=1}^{n_-}\beta_j 
+\dots
\\
\dots +\frac{
\tyu(\sum_{j=1}^{n_-}\beta_j^2|v^j|^2-\sum_{i=1}^{n_+}\a_i^2|u^i|^2\uyt s^2+\sum_{i=1}^{n_+}\a_i^2|u^i|^2
}{
\tyu \sum_{j=1}^{n_-}\beta_j|v^j|^2+\sum_{i=1}^{n_+}\a_i|u^i|^2\uyt s^2+\sum_{i=1}^{n_+}\a_i^2|u^i|^2
}
+o(1)\Bigg).
\eega

From now on, we will focus on the case $m=2$.
\begin{lemma}\label{lem:notd12}
    If $t_0$ is a critical value of $T$ with one critical point $p_0$ of index $i\in \{0,1,2\}$, then $|\f'(c+t)|=O(1)+\Theta\tyu  g_i(t-t_0) \uyt$ for $t\to t_0$, where
    \be 
\begin{aligned}
   g_0(t)= & \frac{1}{\sqrt{|t|}}1_{\{t> 0\}}; \\
   g_1(t)= & \frac{-\mathrm{sgn}(t)}{\sqrt{|t|}}=g_0(t)+g_2(t); \\
   g_2(t)= & -\frac{1}{\sqrt{t}}1_{\{t< 0\}}.
\end{aligned}
    \ee
    (See \cref{fig:compensation}).
\end{lemma}
\begin{figure}
    \centering
    \includegraphics[scale=0.6]{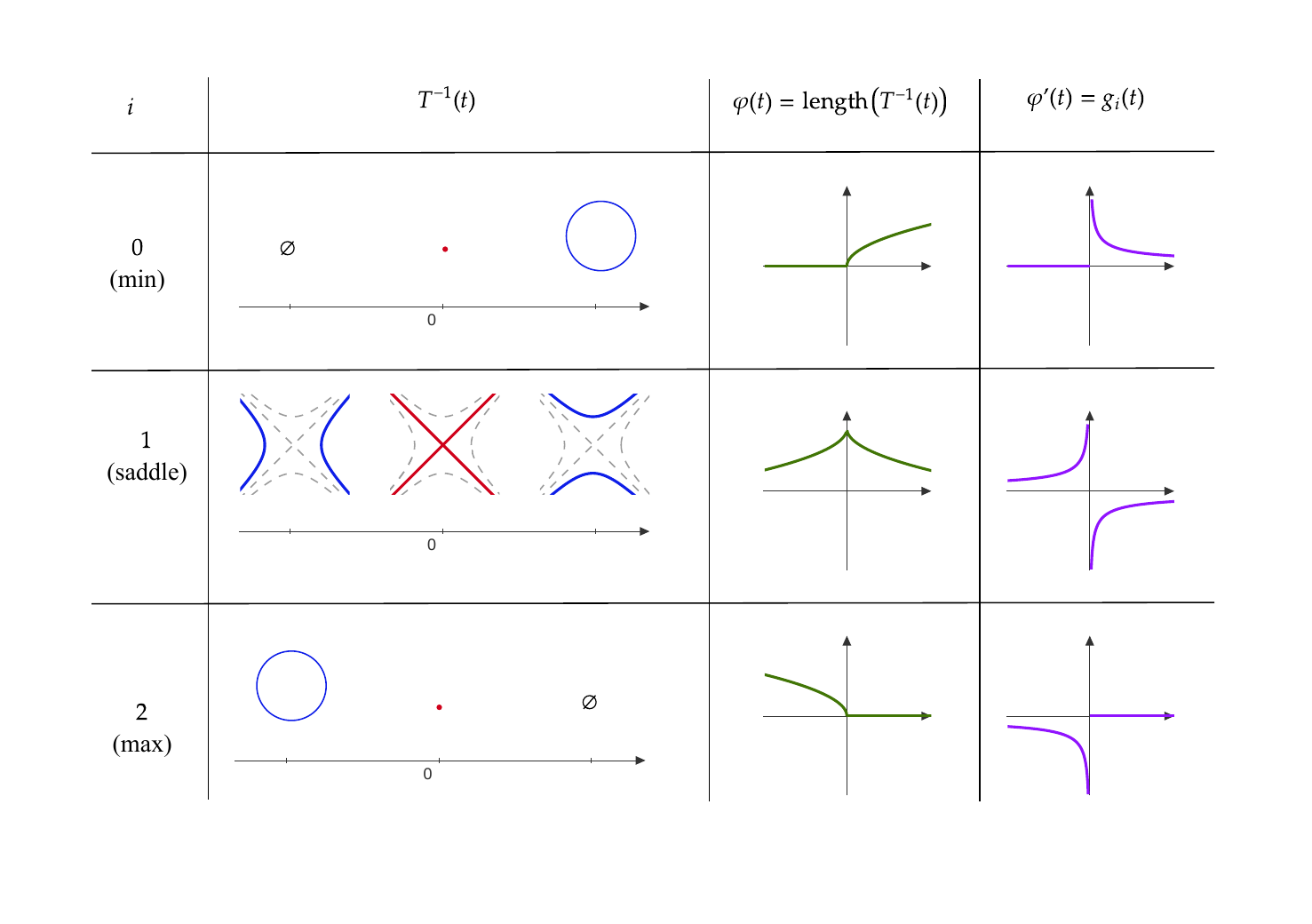}
    \caption{The picture illustrates the behavior of the length of level sets $T^{-1}(t)$ (assuming these are compact) of a Morse function $T\colon \R^2\to \R$, having only one critical point of index $i\in \{0,1,2\}$ with critical value $t=0$.}
    \label{fig:compensation}
\end{figure}
\begin{proof}
we can assume that $t_0=0$ and look at what happens as $t\to 0^+$. The case $t\to 0^-$ can be deduced by considering the Morse function $-T$. 

Let $h(x):=\frac{\tilde{\Delta} T(x)}{\|d_xT\|^2}$. Then, we have that $|\f'(t)|=O(1)+|\int_{Z_t^1}h dZ_t^1|$. We can study the second term as in \cref{eq:hSigmat}.

Let us start from the easiest case: when $i=2$, it means that $0$ is a local maximum value for $T$,thus  for $t>0$ we have $Z_t^1=\emptyset$, hence $g_2(t)=0$ for $t\ge 0$. 

In the case $n_-=0$, we have that $0$ is a local minimum value for $T$. Let us denote $u=x_+=(u_1,u_2)$ Then, for all $t\in [0,\e]$, we have that $Z_t^1=x^{-1}(\sqrt{t}S^1)$, thus $\f '(t)=O(1)$
\bega 
\int_{Z_t^1}h dZ_t^1 
&=
t^{\frac{1}{2}}\int_{S^{1}}\m{J}_g(dT(\sqrt t u)))h(\sqrt{t}u)dS^{1}(u).
\\
&=
\frac{t^{-\frac{1}{2}}}{2}\int_{S^{1}}\m{J}_g(dT(\sqrt t u)))
\frac{
\tyu 
\a_1+\a_2 -\frac{\a_1^2u_1^2+\a_2^2u_2^2}{\a_1u_1^2+\a_2u_2^2}
+o(1)\uyt
}{\a_1u_1^2+\a_2u_2^2} dS^1(u).
\\
&=
\frac{t^{-\frac{1}{2}}}{2}\int_{S^{1}}\m{J}_g(dT(\sqrt t u)))\frac{2\a_1\a_2}{\tyu\a_1u_1^2+\a_2u_2^2\uyt^2} dS^1(u)
\\
&=\Theta\tyu \frac{1}{\sqrt{t}} \uyt=\Theta\tyu g_0(t) \uyt.
\eega

Now consider the more complicated case: when $n_-=n_+=i=1$. Define $\mu(u,v,r):=\tyu\m{J}_g(dT(\psi_t(u,v,r)))\uyt$. Then, by \cref{eq:hSigmat}, we have that
\bega 
\int_{Z_t^1}h dZ_t^1 &=
\sqrt{t}\sum_{(u,v)\in S^0\times S^0}\int_{0}^{\frac{\e}{\sqrt{t}}}h(\psi_t(u,v,s\sqrt{t}))
\sqrt{2s^2+1}(s^2+1)^{-\frac12}
\mu(u,v,s\sqrt{t})
ds.
\eega 
Using the the coordinate $x=(x_+,x_-)$ discussed in \cref{subsub:locexp}, we have
\bega 
h(\psi_t(u,v,s\sqrt{t}))
&=\frac{\frac1{2t}}{(\a+\beta)s^2+\a }\tyu 
\a-\beta 
+
\frac{(\beta^2-\a^2)s^2-\a^2}{(\a+\beta)s^2+\a}
+o(1)\uyt
\\
&=\frac1{2t}\tyu 
\frac{-\a\beta}{\qwe(\a+\beta)s^2+\a\ewq^2}
+\frac{o(1)}{(\a+\beta)s^2+\a}\uyt
\\
&=\frac{-\frac{\beta}{\a}}{2t}\tyu 
\frac{1}{\qwe\nu^2s^2+1\ewq^2}
+\frac{o(1)}{\nu^2 s^2+1}\uyt
\eega
where $\nu=\sqrt{1+\frac{\beta}{\a}}$ depend only on $\hess(T)_0$ and $g_0$. By construction, there are two positive constants $c_1,c_2>0$ such that $\mu(u,v,r)\in [c_1,c_2]$. For any parameter $\delta>0$ we can choose $\e=\e(\delta)$ so small that $|o(1)|\le \delta$. Then, we will need to choose $\delta=\delta(\nu)$ small enough and, consequently, $\e=\e(\delta(\nu))$.
\bega 
\int_{Z_t^1}h dZ_t^1 &=
\frac{-\frac{\beta}{\a}}{2\sqrt{t}}\sum_{(u,v)}\int_{0}^{\frac{\e}{\sqrt{t}}}
\tyu 
\frac{1}{\qwe\nu^2s^2+1\ewq^2}
+\frac{o(1)}{\nu^2 s^2+1}\uyt
\sqrt{\frac{2s^2+1}{(s^2+1)}}
\mu(u,v,s\sqrt{t})
ds\le \dots
\\
&\dots\le
\frac{-C}{\sqrt{t}}
\tyu\tyu \int_{0}^{+\infty}
\frac{ds}{\tyu\nu^2s^2+1\uyt^2}
\uyt
-\delta\tyu
\int_{0}^{+\infty}
\frac{ds}{\nu^2 s^2+1}\uyt
\uyt
c_1\le \dots
\\
&\dots \le -\frac{1}{\sqrt{t}} C'=C'g_1(t).
\eega 
where the last inequality holds, for some positive constants $C,C'>0$, as soon as we choose $\delta(\nu)$ small enough, since both integrals in the previous line are finite. A lower bound $\f'(t)\ge -C'' \frac{1}{\sqrt{t}}$ holds as well, by \cref{thm:upperbound}, thus we have proved that $\f'(t)=\Theta(g_1(t))$, as $t\to 0^+$.
\end{proof}
\begin{proposition}\label{prop:3pt}
Let $M=\inter M\sqcup \de M$ be a compact Riemannian $\mC^2$ surface with boundary and let $T\colon M\to \R$ be a Morse function. Assume that $t_0$ is a critical value of $T$ with critical set $C\subset \inter M$. Then the function $\f'(t):=\frac{d}{dt}\vol_1(\{T=t_0+t\})$ behaves as $\frac{1}{\sqrt{t}}$ on at least one side of $0$ unless there are at least three critical points $p_0,p_1,p_2\in C$ of Morse index $0,1,2$, respectively.
\end{proposition}
\begin{proof}
The function $\f'(t)$ behaves as the sum of the local behavior around each critical point $p\in C$. In other words, if we partition $C=C_0\sqcup C_1\sqcup C_2$ according to the Morse index, then we have 
\be 
\f'(t) =O(1)+ \Theta \tyu\sum_{i=0,1,2}\#(C_i) A_i g_{i}(t)\uyt,
\ee 
as $t\to 0$. Where $A_i>0$ are some constants.
then It is clear from \cref{lem:notd12} (see also \cref{fig:compensation}) that the only way in which the above sum could behave differently than $\frac{1}{\sqrt{|t|}}$ on both sides of $t=0$ is if $\#(C_0) A_0=\#(C_2) A_2=\#(C_1) A_1$, which implies that $\#(C_i)>0$ for all $i\in \{0,1,2\}$.
\end{proof}
\subsubsection{Lower asymptotics for a two-dimensional manifold with boundary}
\cref{lem:notd12} holds in a weaker form, when the manifold has boundary.
\begin{proposition}\label{lem:notd12bd}
Let $t_0$ be a critical value of $T$ and assume that $p\in \de M$ is the only critical point in $T^{-1}(t_0)$, then $|\f'(t)|=\Theta\tyu\frac{1}{\sqrt{|t-t_0|}}\uyt$ for $t\to t_0^+$, or for $t\to t_0^-$.
\end{proposition}
\begin{proof}
    We have by \cref{eq:fiprimebd} and the discussion thereafter, we have 
    \be 
\f'(t_0+t)=
-\sum_{p\in \de Z_t} \frac{ g(\n,\nu)|_p }{\|d_p(T|_{\de M})\|} +O(1),
\ee
Assume that $p$ is a minimum of $T|_{\de M}$. Then, there is a $\mC^1$ coordinate chart $\f\colon U\to \R^2$ around $p=\f^{-1}(0,0)$ such that 
\be 
\f(U)=\kop(x,y)\colon x\in (-1,1), \ \e y\ge \e x^2\pok, \quad \text{and} \quad T\circ \f^{-1}(x,y)= y,
\ee
where $\e=-\mathrm{sign}(g(\n,\nu)|_p)\in \{-1,1\}$ 
(see \cite[Lemma de Morse $\mC^r$ \`{a} param\`{e}tres]{bromberg1993lemme}
). For simplicity, assume that $W=U$. Define $p_\pm(t)=(\pm\sqrt{t},t)$, then we have $\de Z_t=\{p_-(t),p_+(t)\}$ and $T_{p_\pm(t)}\de M=\mathrm{span}\{\de_x\pm 2\sqrt{t}\de_y\}$. 
\be 
\|d_p(T|_{\de M})\|=\left|g\tyu \g f, \frac{\de_x\pm 2\sqrt{t}\de_y}{\|\de_x\pm 2\sqrt{t}\de_y\|}\uyt\right|=
\frac{2\sqrt{t}}{\|\de_x\pm 2\sqrt{t}\de_y\|}.
\ee
We conclude that $|\f'(c+t)|\ge \frac{C}{\sqrt t}$ for some $C>0$.
\end{proof}
\begin{remark}
In the boundary case, one can formulate a suitable generalization of \cref{prop:3pt}. However, in this case the situation is less rigid, indeed the local behavior around each critical point at the boundary can take all the four forms: $\pm g_0(t), \pm g_2(t)$, depending on the index $i$ as a critical point of $T|_{\de M}$ ($i=0$ if the point is a local minimum and $i=1$ if it is a local maximum) and on the sign $s_\n$ of $dT(\n)$ at the critical point, where $\n$ is the outer normal vector to the boundary. If we denote as $s_\de$ the sign of $2i-1$, then we have:
\be 
\f'(t)=O(1)+\Theta\tyu  s_{\n} s_\de\cdot g_{2i}(t)\uyt
\ee 
(see \cref{fig:boundaryconvex} and \cref{fig:boundaryconcave}).
\begin{figure}
    \centering
    \includegraphics[scale=0.56]{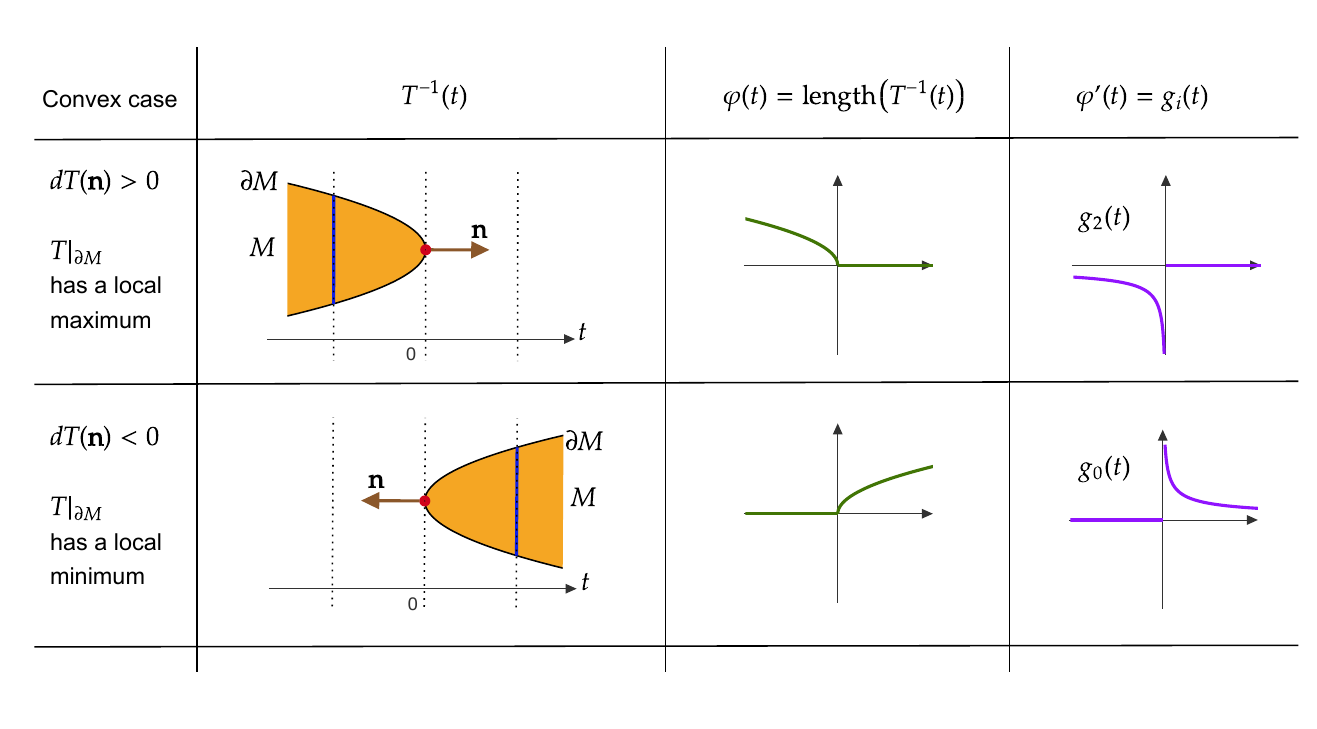}
    \caption{The picture illustrates the behavior of the length of level sets $T^{-1}(t)$ of a Morse function $T\colon M\to \R$, where $M$ is a compact domain in $\R^2$ and $T$ is the projection on the horizontal axis, $T(x,y)=x$. There are four qualitatively different cases in total, depending on the convexity of the domain near the critical point. Two are depicted above and the other two are in \cref{fig:boundaryconcave}.}
    \label{fig:boundaryconvex}
\end{figure}
\begin{figure}
    \centering
    \includegraphics[scale=0.56]{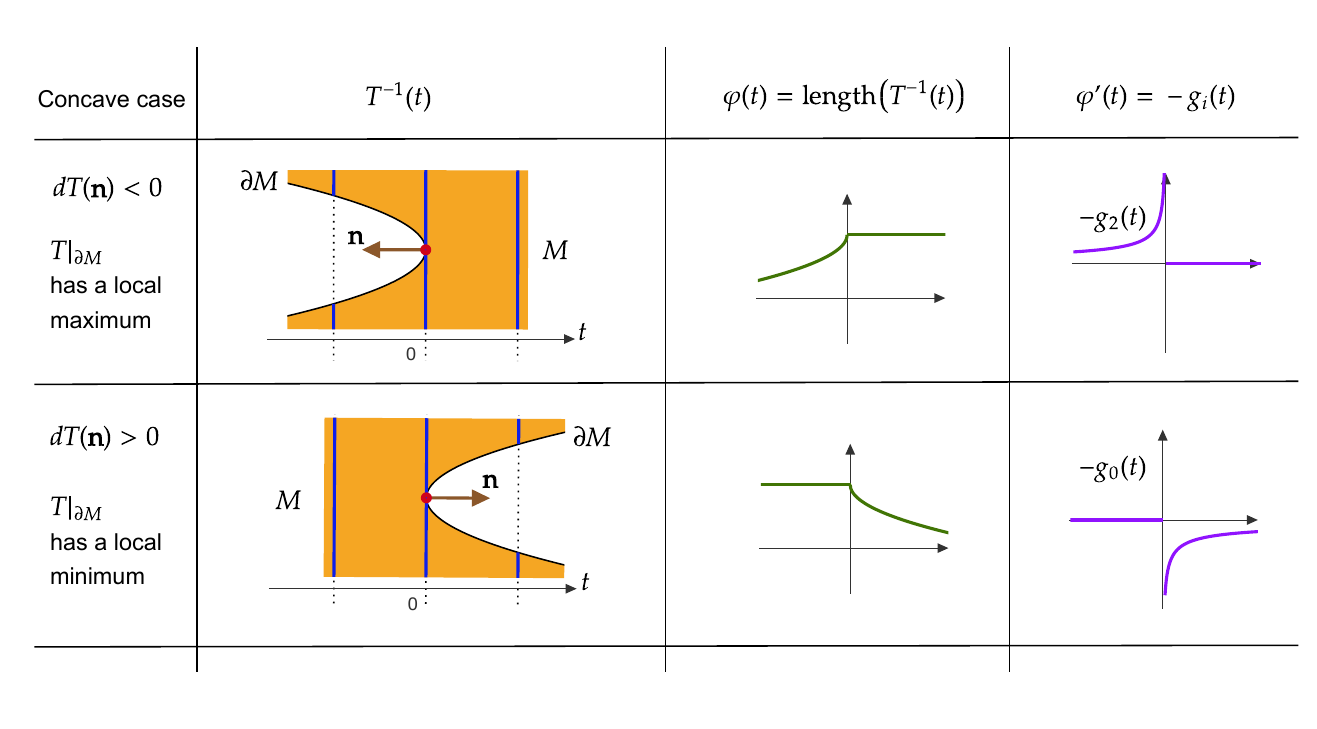}
    \caption{See the caption of \cref{fig:boundaryconvex}.}
    \label{fig:boundaryconcave}
\end{figure}
\end{remark}
\section{The nodal volume along transverse curves}\label{sec:transverse}
In this section we put together the results of \cref{sec:nondegeneracyconditions} and of \cref{sec:levelMorse} in order to complete our study of the function $\f(t)=V(f_t)$, where $(f_t)_{t\in I}$ is a transverse curve in $\mC^2(M)$, see \cref{thm:DeterministPicture} and \cref{def:noncomp} below.
Taking up the discussion started at the beginning of \cref{sec:nondegeneracyconditions}, this study is aimed at giving criteria for verifying the condition of ray-absolute continuity, (i) of \cref{def:Dspace}, that is an important part of the proof of \cref{t:mainintro}. Indeed, \cref{thm:DeterministPicture} and \cref{cor:notD12} below essentially characterize the Sobolev regularity of the restriction of the nodal volume functional $V$ to random segments $[X,X+h]$ in $\mC^2(M)$, given that, by \cref{thm:transcurve}, such segments are transverse curves with probability one.

Furthermore, the continuity part of \cref{thm:DeterministPicture}, stating that $\f$ is continuous, combined in particular with the results of \cref{sec:ACDB}, allows us to prove point (i) of \cref{t:mainintro}, with the exception of the description of the singular part of the law (\cref{e:magnificentlaw}), which will be discussed in \cref{s:singular}. \cref{thm:mainabs} states the existence of an absolutely continuous component of the law of $V(X)$ and \cref{prop:interval} ensures that its support is an interval. 
\subsection{Sobolev regularity of the nodal volume}
The following deterministic result is one of the key point of the paper. Observe that \cref{eq:thmupas} provides an intuitive justification of the results (ii),(iii) and (iv) of \cref{t:mainintro}. Indeed, it is an important ingredient of the proofs of \cref{thm:D12}, \cref{thm:D11} and \cref{thm:D12boundary}.
\begin{theorem}\label{thm:DeterministPicture} 
Let \cref{ass:1} prevail. Let $I$ be a bounded interval, let $(f_t)_{t\in I}$ be a $\mC^2$ transverse curve in $\mC^2(M)$, then $I_c=\kop t\in I\colon f_t\notin \m U\pok$ is finite. Define $\f(t)=V(f_t)$ then
$\f\in \mC^0(I)\cap \mC^1(I\smallsetminus I_c)$. There is a constant $C>0$ such that for any $t_0\in I_c$,
\be\label{eq:thmupas}
|\f'(t_0+t)|\le C\left\{\begin{aligned}
    & 1, \quad \text{if }m\ge 4 
    \\
    & |\log t|, \quad \text{if }m= 3;
    \\
    & t^{-\frac12}, \quad \text{if }m= 2.  
\end{aligned}\right.
\ee
In particular, $\f\in W^{1,1}(I)$; if $m\ge 3$, then $\f\in W^{1,k}(I)$ for all $k\in\N$. Moreover, if $m\ge 4$, then $\f\in \mC^1(I)$.
\end{theorem}
\begin{proof}
By \cref{cor:TransvToMorse}, there exists a $\mC^2$ Riemannian manifold with boundary $N$ and a Morse function $T\colon N\to \R$ such that $\f(t)=\vol^{m-1}(T^{-1}(t))$, so the statement follows from \cref{lem:Morsepiecewisecont}.
\end{proof}
\begin{remark}
One might be led to think that $\f\in \mC^0(I)\cap \mC^1(I\smallsetminus I_c)$, with $I_c$ finite, is enough to conclude that $\f$ is absolutely continuous. Indeed, $t_0,t_1\in I_0$, the following limit exists:
\be
\f(t_{1})-\f(t_{0})=\lim_{\e_0,\e_1\to 0^+}\f(t_{1}-\e_1)-\f(t_{0}+\e_0)=\lim_{\e_1,\e_2\to 0^+}\int_{t_{0}+\e_0}^{t_1-\e_1}\f'.
\ee
The existence of the rightmost limit, does not imply that the function $\f'$ is integrable in the considered interval, i.e., that $\f'\in L^1([t_0,t_{1}])$. A counterexample is the function $\f(x)=\frac{1}{x}\sin\tyu\frac{1}{x}\uyt$.
\end{remark}
\subsection{Sobolev (non)regularity in dimension $2$}\label{sec:sobd2}
\begin{definition}\label{def:noncomp}
Let $f_0,e\in E$ and let $\ell=f_0+\R e$. We say that $\ell$ is a \emph{non-compensating transverse line} if $\ell\transv \m W$ (in the sense of \cref{def:pseudotransv}) and if, whenever $f=f_0+t_0e\in \m W$, 
we have that: \begin{enumerate}[(1)]
\item\label{im} $0$ is a (strat.) Morse critical value of $f$ ;
\item\label{is} $\CZ_f\subset \inter M$, or $\CZ_f=\{p\}$ is one point;
\item\label{ii} the symmetric forms $e(p)\H_{p}f$ all have the same index for every $p\in C$.
\end{enumerate}
\end{definition}
From \cref{thm:bestia} we can deduce the following results concerning a Gaussian field $X\randin \mC^2(M)$.
\begin{corollary}\label{cor:immcor}
Let \cref{ass:1} prevail. Assume that there are no pairs of points $p,q \in \de M$ such that the random vectors $j^1_{p}X|_{\de M}, j^1_{q}X|_{\de M}$ in $J^1(\de M,\R)$ are completely correlated. Then, for every $e\neq 0 \in E$, almost surely, $X+\R e$ is a non-compensating transverse line.
\end{corollary}
\begin{proof}
This is a direct consequence of \cref{thm:bestia}.
\cref{thm:bestia} immediately implies condition (1) of \cref{def:noncomp}. Point (3) of \cref{thm:bestia} implies that $CZ_f\subset \inter{M}$ or $CZ_f\subset \de M$. Moreover, Condition (5) of \cref{thm:bestia} is satisfied by a pair of points $p,q\in \de M$ only if $p=q$, thus we deduce property (2) of \cref{def:noncomp}. Finally, condition (4) of \cref{def:noncomp} corresponds to point (3) of \cref{thm:bestia}.
\end{proof}
\begin{corollary}\label{cor:notD12}
Let \cref{ass:1} prevail with $\dim M=2$. Assume that there are no pairs of points $p,q \in \de M$ such that the random vectors $j^1_{p}X|_{\de M}, j^1_{q}X|_{\de M}$ in $J^1(\de M,\R)$ are completely correlated. Then, for every $e\neq 0 \in E$, almost surely, either $X+\R e\subset \m U$, or the function $t\mapsto V(X+te)$ is not in the space $W^{1,2}(\R)$.
\end{corollary}
\begin{proof}
Let $\f(t)=V(X+te)$. \cref{cor:immcor} implies that $X+\R e$ is a non-compensating transverse line. 
From 
\cref{prop:3pt}, since there cannot be three critical points of different index,
we know that if $\CZ\subset \inter M$, then $(\f')^2$ is not integrable in a neighborhood of a critical value $t_0$. If $\CZ=\{p\}\subset \de M$, the same is true, as it follows from \cref{lem:notd12bd}. It follows that the only situation in which $(\f')^2$ is integrable is if there are no points $t\in \R$ such that $f+te\notin \m U$.
\end{proof}

\subsection{Absolute continuity of the nodal volume}\label{sec:nons}
Let us consider the setting of \cref{sec:ACDB}, with $\mu$ being the Gaussian measure induced by $X$ on its topological support $E\subset \mC^2(M)$ and $V\colon E\to \R$ being the restriction of the nodal volume functional.
\begin{theorem}\label{thm:mainabs}
Let \cref{ass:1} prevail. If $V(X)$ is non-constant, then the law of $V(X)\randin \R$ and the Lebesgue measure are not mutually singular.
\end{theorem}
\begin{proof}
We want to apply \cref{cor:abs}. To do so, considering \cref{lem:bulinskaya} and \cref{thm:firstvar}, we only need to show that there is a connected component of $\m U\cap E \subset E$ on which $V$ is not a constant function. By contradiction, if $V$ were constant on each connected component of $\m U\cap E$, then $V(\m U\cap E)$ should be countable, since $E$ is a II-countable space. By \cref{thm:transcurve}, for any two connected components of $\m U\cap E$, there is a transverse curve in $E$ connecting one to the other. This means that any two points $f_0,f_1\in \m U\cap E$ are connected by a transverse curve $c:[0,1]\to E$, so that $c^{-1}(E\smallsetminus \m U)$ is a finite, or empty, subset $I_c\subset [0,1]$. Now, \cref{thm:DeterministPicture} ensures that $\f=V\circ c$ is continuous, hence $[V(f_0),V(f_1)]\subset \f([0,1]) = V(\m U\cap E)\cup V(I_c) $ is countable, which implies that $V(f_1)=V(f_2)$. Applying the latter argument to every pair $f_1,f_2\in \m U\cap E$, we conclude that $V$ is constant on $\m U\cap E$, which contradicts the hypothesis.
\end{proof}
\begin{proposition}\label{prop:interval}
Let \cref{ass:1} prevail. The topological support of the random variable $V(X)$ is the set $\overline{V(\m U\cap E)}$ and it is an interval. Moreover, $\overline{V(\m U\cap E)}\smallsetminus V(\m U\cap E)$ is locally finite.
\end{proposition}
\begin{proof}
Let $S\subset \R$ be the topological support of $V(X)$, that is, the smallest closed set of probability one.
$V(\m U\cap E)\subset \R$ is an event of probability one for $V(X)$, thus $S\subset  \overline{V(\m U\cap E)}$. On the other hand, since $V$ is continuous on the open set $\m U$, the set $V^{-1}(\R\smallsetminus S)\cap \m U\cap E$ is open in $E$. Clearly, it has probability zero, so it must be empty. From this, we deduce that $V(\m U\cap E)\subset S$ which implies that $S=\overline{V(\m U\cap E)}$. 

Now we prove that $\overline{V(\m U\cap E)}$ is an interval.
For any two connected components of $\m U\cap E$, there is a transverse curve connecting one to the other. This implies that any two points in $\m U\cap E$ are connected by a transverse curve. \cref{thm:transcurve} shows that a transverse curve is always contained in $\m U\sqcup \m U'\cap E$ and any point in $\m U\sqcup \m U'\cap E$ lies on some transverse curve, by \cref{cor:everytransv}; moreover, \cref{thm:DeterministPicture} states that $V$ is continuous along a transverse curve. From this, we can conclude that the set $V(\m U\sqcup \m U'\cap E)$ is an interval and that $V(\m U\sqcup \m U'\cap E)\subset \overline{V(\m U\cap E)}$. It follows that $\overline{V(\m U\sqcup \m U'\cap E)}= \overline{V(\m U\cap E)}$ is also an interval. Moreover, since a transverse curve is contained in $\m U$ for all but a locally finite set of times, we conclude that $\overline{V(\m U\sqcup \m U'\cap E)}\smallsetminus V(\m U'\cap E)$ is locally finite. 
\end{proof}
\begin{remark}\label{r:whatisused}
In the above proofs, we used only the results discussed up to \cref{sec:degMorse}, with the exception of \cref{thm:DeterministPicture}. However, we do not need the full power of the latter: the only thing that is needed is \cref{lem:Morsepiecewisecont}, ensuring the continuity of the nodal volume of the level sets of a Morse function. This is a quite direct consequence of Morse Lemma \ref{lem:CrMorse}. In particular, the proof of \cref{lem:Morsepiecewisecont} follows from a very simple sub-case ($k=0$, and $h=1$) of \cref{thm:upperbound}. Considering that, we can say that \cref{thm:mainabs} and \cref{prop:interval} rely only on the results discussed up to \cref{sec:degMorse} and on Morse Lemma.
\end{remark}
\section{Preliminary lemmas}\label{s:prelims}
This section is devoted to a study of the integrability of $\|d_XV\|_{\HX}^2$, with the purpose of checking the validity of condition (iii) of \cref{def:Dspace}, for the nodal volume random variable $V(X)$. 

Let $X\randin \mC^2(M)$ be a Gaussian satifying \cref{ass:1}. 
Let $C(p,q)=\E\kop X(p)X(q)\pok$ be its  covariance function. Then, $C\colon M\times M\to \R$ is semipositive definite, symmetric and of class $\mC^{2,2}(M\times M)$.
\subsection{Kac-Rice formula on the zero set}
\subsubsection{Degenerations are few}
Define
\be 
\Delta^X:=\kop (p,q)\in M^2: (X(p),X(q)) \text{ is degenerate}\pok
\ee
The diagonal $\Delta:=\kop (p,p)\colon p\in M\pok\subset M\times M$ is, in general, a subset of $\Delta^X$. Notice that $\Delta^X$ is closed.
\begin{lemma}\label{lem:Mp1}
    For every $p\in M,$ the set $
M_p^X=\Delta^X\cap (\{p\}\times M) $
    is finite. Moreover, there is a neighborhood $B(\Delta)$ of the true diagonal in $M\times M$ such that $B(\Delta)\cap \Delta^X=\Delta$.
\end{lemma}
\begin{proof}
It is equivalent to prove the statement for the normalized field $\E\{|X|\}^{-1}X$, so we may assume that $\E |X(p)|^2=1$ for all $p$. 
In this case, we have 
\be 
M_p^X=\kop q\in M: X(p)=\pm X(q) \text{ a.s.}\pok.
    \ee
Let $C$ be the covariance function of $X$ and let $h_p:=C(p,\cdot).$ Then $M_p^X=h_p^{-1}(\kop -1,1\pok),$ moreover $|h_p(q)|\le 1$ for all $q\in M,$ by Cauchy-Schwartz, so it is sufficient to show that $h_p$ has non-degenerate second derivative at the points $q\in M_p^X.$ This follows from:
    \bega 
d^2_qh_p=\E\{X(p)d^2_qX\}=\E\{X(q)d^2_qX\}=-\E\{d_qX (d_qX)^T\}.
    \eega
We prove the last part by contradiction: we assume that the conclusion is false, that is, we assume that there exists a sequence $(p_n,q_n)\in \Delta^X\smallsetminus \Delta$ converging to a point $(p,p)$ in the true diagonal. Notice that this means that $X(p_n)=X(q_n)$ almost surely. Then, passing to a subsequence we can assume that in a local chart around $p$ we have 
\be 
\frac{p_n-q_n}{|p_n-q_n|}\to v \in T_pM\smallsetminus \{0\}.
\quad \text{ hence }\quad 
0=\frac{X(p_n)-X(q_n)}{|p_n-q_n|}\xrightarrow{a.s.}  d_pX(v).
\ee
This contradicts the nondegeneracy of $j^1_pX$, so we conclude that such a sequence does not exist.
\end{proof}
\subsubsection{The double field is z-KROK}
Taking up the language of \cite{MathiStec}, a random field is said to be \zkrok if, roughly speaking, the Kac-Rice formula (see \cite{AzaisWscheborbook,AdlerTaylor}) for its zero set behaves well. Such concept is analogous to the hypotheses of \cite[Theorem 6.10]{AzaisWscheborbook}
\begin{lemma}\label{lem:krAlpha}
The mapping $(p,q)\mapsto (X(p),X(q))$ defines a $\mC^2$ Gaussian random map $X^{\times 2}\colon M\times M \to \R^2$ for which $(0,0)$ is a regular value almost surely. Moreover, its restriction to $M\times M\smallsetminus \Delta^X$ is \zkrok in the sense of \cite{MathiStec} and we have the following formula: let $Z=X^{-1}(0)$ and $\a\colon \mC^1(M)\times M\times M\to \R$ be a Borel function, then 
\bega\label{eq:krkrAlpha}
\E & \kop  \int_Z\int_Z\a(X,p,q)dZ(p)dZ(q)\pok 
=\dots 
\\
\dots &=\int_{M\times M\smallsetminus \Delta^X}\E\kop
\a(X,p,q)\|d_pX\|\|d_qX\|
\Bigg|
\begin{aligned}
X(p)&=0 \\ X(q)&=0
\end{aligned}
\pok \frac{dM(p)dM(q)}{2\pi\sqrt{C(p,p)C(q,q)-C(p,q)^2}}.
\eega
\end{lemma}
\begin{proof}
First, observe that since the Gaussian field $X$ satisfies the hypotheses of Bulynskaya \cref{lem:bulinskaya}, thus, with probability one, $0$ is regular value of $X$ and $Z\subset M$ is a $\mC^2$ hypersurface. This implies that $(0,0)$ is almost surely a regular value for the map $X^{\times 2},$ as well. 
Moreover, for any $(p,q)\in U^X:=M\times M \smallsetminus\Delta^X$ we have that $X^{\times 2}(p,q)$ has a density in $\R^2$.
 From this, we see that random function $X^{\times 2}|_{U^X}$ satisfies the hypotheses of \cite[Prop. 4.9]{MathiStec}, thus it is a \zkrok random field.

Observe that if $m\ge 2,$ then $Z\times Z\cap \Delta^X$ has zero $2(m-1)$-volume, because of \cref{lem:Mp1}, so that the left-hand-side of \cref{eq:krkrAlpha} can be seen as an integral over the set $\{X^{\times 2}=(0,0)\}.$ Therefore, by \cite[Thm. 6.2]{MathiStec} we obtain the wanted formula.
\end{proof}
Note that if $\a$ is defined only on the subset $\m{Z}=\{(f,p,q)\in \mC^2(M)\times M\times M: f\transv 0, f(p)=f(q)=0 \},$ then the formula still holds. To see this, observe that $\m{Z}$ is a Borel subset, so $\a$ can be extended to a global Borel $\a'$ function, for which we can apply the theorem. In the end, however, the formula depends only on the restriction $\a'|_{\m{Z}}=\a.$
\subsection{Uniform integrability of rational functions of Gaussian vectors }\label{sec:boringauss}
In this subsection we will consider the space $\m{X}$ of Gaussian vectors in $\R^{N_0}\times \R^{N_1}\times \R^{N_2}$. Elements of this space are random vectors $X=(X_0,X_1,X_2)\sim N(0,K)$, identified by their covariance matrix:
\be 
K=\begin{pmatrix}
K_{00} & K_{01} & K_{02} \\
{ } & K_{11} & K_{12} \\
{ } & { } & K_{22}
\end{pmatrix}.
\ee
Hence, $\m{X}$ can be identified with the space of symmetric semi-positive definite matrices $K$ of size $N=N_0+N_1+N_2$. For $A,B\in \N$, let $\m{X}_{A,B}\subset \m{X}$ be the open subset such that $\mathrm{rank}(K_{11})\ge A$ and $\mathrm{rank}(K_{22})\ge B.$
\begin{remark}
    For any $A\in \N$ the condition $\mathrm{rank}(K_{11})\ge A$ is equivalent to have that
    \be 
\E\kop \frac{1}{\|X_1\|^{A-1}}\pok<\infty.
    \ee
\end{remark}
At the same time we will consider the space of sub-polynomial functions $\m{F}$. An element $F\in \m{F}$ is a measurable function $F\colon \R^N\to \R$ such that there exists a constant $D>0$ such that
\be\label{eq:D}
|F(x)|\le D(1+|x|^D), \text{ for all $x\in \R^N.$}
\ee
We will call $\m{F}_D$ the subset of $\m{F}$ on which \eqref{eq:D} holds with the constant $D$. We put on $\m{F}_D$ the topology of point-wise convergence.

We need some criteria for the convergence of integrals of the form 
\be 
I_{a,b}(F,K)=\E F(X)\|X_1\|^{-a}\|X_2\|^{-b},
\ee 
with $a,b\in \N.$ 

A first idea is that $I_{a,0}$ should be continuous on $\m{F}_D\times \m{X}_{A,0}$ if $a<A$. Indeed, when $\|X_1\|=\chi_A$ is a chi random variable of parameter $A$, we have that $\E\{\chi_A^{-a}\}<\infty$ if and only if $a<A.$

A second idea is that $I_{a,b}$ should be continuous on $\m{F}_D\times \m{X}_{A,B}$ if $2a<A$ and $2b<B$. Indeed, by Cauchy-Schwartz, we have that $\E\{\chi_A^{-a}\chi_B^{-b}\}^2<\E\{\chi_A^{-2a}\}\E\{\chi_B^{-2b}\}.$

A third idea (which may be overkill) is that $I_{a,b}$ should be continuous on $\m{F}_D\times \m{X}_{A,B}$ if $pa<A$ and $qb<B$ for some conjugated exponents $p$ and $q$, that is, $\frac{1}{p}+\frac{1}{q}=1.$ Indeed, by Holder inequality, we have that $\E\{\chi_A^{-a}\chi_B^{-b}\}<\E\{\chi_A^{-pa}\}^{\frac{1}{p}}\E\{\chi_B^{-qb}\}^{\frac{1}{q}}.$
The case $b=0$ can be interpreted with $p=1$ and $q=+\infty.$
\begin{theorem}\label{thm:overkill}
    $I_{a,b}$ is continuous on $\m{F}_D\times \m{X}_{A,B}$ if $pa<A$ and $qb<B$ for some conjugated exponents $p$ and $q$, that is, $\frac{1}{p}+\frac{1}{q}=1.$ Moreover, for all $(F,K)\in \m{F}_D\times \m{X}_{N_1,N_2}$ such that $\|K\|\le D$, we have 
    \be 
I_{a,b}(F,K)\le C(D)\tyu 1+\frac{1}{\sqrt{\det K_{11}}}+\frac{1}{\sqrt{\det K_{22}}}\uyt
    \ee
\end{theorem}
\begin{corollary}\label{cor:IX}
    $I_{1,1}$ is continuous on $\m{F}_D\times \m{X}_{3,3}$.
\end{corollary}
\begin{proof}
    Set $p=q=2$, then $p\cdot 1=2<3.$
\end{proof}
\begin{proof}[Proof of \cref{thm:overkill}]
Let $K(n)\to K$ in $\m{X}_{A,B}$, as $n\to +\infty$. 
Let $X(n)\to X$ be the corresponding sequence of Gaussian vectors, which we can assume to be almost surely convergent. Let $L(n)\to L$ be a converging sequence of matrices such that $L(n)L(n)^T=K(n)$. Then\be 
I_{a,b}(F_n,K_n)=\frac{F_n(X(n)\gamma)}{|X_1(n)|^a|X_2(n)|^b}
\ee
To conclude, we will use Fatou Lemma in the form of the following argument. Young's inequality $xy\le \frac{a^p}{p}+\frac{y^q}{q}$ yields $|f_n|\le g_n,$ where:
\bega 
f_n:=\frac{F_n(X(n))}{|X_1(n)|^a|X_2(n)|^b}\quad  g_n:= |F_n(X(n))|\tyu\frac{1}{p|X_1(n)|^{ap}}+\frac{1}{q|X_2(n)|^{bq}}\uyt.
\eega 
We have that $f_n\to f=\frac{F(X)}{|X_1|^a|X_2|^b}$ and $g_n\to g=|F(X)|\tyu\frac{1}{p|X_1|^{ap}}+\frac{1}{q|X_2|^{bq}}\uyt\ge |f|$ a.s. Thus, in order to prove that $\E f_n\to \E f$, it is sufficient to show that  $\E\kop g_n\pok<+\infty$ uniformly. We can reorder the variables and change the numbers $N_0,N_1,N_2$ to have $N_1=A,$ $N_2=B$, using the inequality $\frac{1}{\|x+y\|}\le \frac{1}{\|x\|}$. 
Let $B_A\subset \R^A$ be the unit ball. Observe that 
\bega 
\E\kop\frac{|F_n(X(n))|}{|X_1(n)|^{ap}}\pok
&\le 2D\E\kop\frac{1_{B_A}(X_1(n)}{|X_1(n)|^{ap}}\pok
+\E\kop D(1+\|X_1(n)\|^D)\pok
\\
&\le \frac{2D}{(2\pi)^{\frac{A}{2}}} \int_{B^A}\frac{1}{|x|}\frac{dx}{\sqrt{\det K_{11}(n)}}+\E\kop D(1+(k \|\gamma_1\|)^D)\pok,
\eega
where $k=\sup_n \|K(n)\|<\infty$ and $\gamma_1$ is a standard Gaussian vector in $\R^A$. Thus, for some constant $C$, we have an inequality
\be 
\E\kop\frac{|F_n(X(n))|}{|X_1(n)|^{ap}}\pok\le \frac{C}{\sqrt{\det K_{11}(n)}}\int_{B^A}\frac{1}{|x|^{ap}}dx+C.
\ee
If the limit $K\in \chi_{A,B}$, then $\det K_{11}\neq 0$ and the integral is finite because $A>ap$. Repeating the same for the other term in $g_n$, we conclude.
\end{proof}
\subsection{The conditional expectation is bounded}
\subsubsection{Differentiation of the conditioning}
    
\begin{lemma}\label{lem:dd}
Let $(p_n,q_n)\in M\times M\smallsetminus \Delta$ be a sequence converging to a double point $(p,p)$. Assume that in some coordinate chart around $p$, we have
\be 
\lim_{n\to \infty}\frac{p_n-q_n}{|p_n-q_n|}=v\in T_pM
\ee 
Let $H_n\to Y$ be a converging sequence of Gaussian vectors in $\R^k$ such that $(H_n,X(p_n),X(q_n))$ are jointly Gaussian.
Then, we have the following convergence in law
\be 
\qwe H_n\ \Big|X(p_n)=X(q_n)=0\ewq \implies \qwe
H\ \Big|X(p)=d_pX(v)=0\ewq.
\ee
\end{lemma}
\begin{remark}
The above statement means that for any $\a\colon \R^N\to \R$ continuous and bounded, we have
\be 
\lim_{t\to 0}\E\{\a(H_n)|X(p_n)=X(q_n)=0\}=\E\{\a(H)|X(p)=d_pX(v)=0\}.
\ee
\end{remark}
\begin{proof}
Consider the family of non-degenerate Gaussian vectors $Y_n\randin \R^{2}$ such that 
    \bega
Y_n:= \tyu X(p_n),\tyu X(q_n)-X(p_n)\uyt \frac{1}{|q_n-p_n|}\uyt \xrightarrow[n\to \infty]{a.s.} Y:= \tyu X(p),d_pX(v) \uyt
    \eega
Notice that in particular the inverse of the covariance matrix $K_n^{-1}=\E Y_nY_n^{-T}$ is a convergent.
Moreover, the sequence of matrices defined as $C_n:=\E\kop  H_nY_n^T\pok$ converges to the matrix $C=\E\kop  HY^T\pok$.
    
Interpreting the conditioning as a projection, we can observe that for all $n\neq 0$, we have $[H_n|X(p_n)=X(q_n)=0]=[H_n|Y_n=0]$, because the $Y_n=0$ if and only if $X_n=X_n=0$. 
Furthermore, using Gaussian regression formula, we can write explicitely the sequence of conditioned Gaussian vectors (defined up to equivalence of their law, i.e., their covariance matrices) as follows
\be 
\tyu H_n \ \Big|X(p_n)=X(q_n)=0\uyt =
H_n-C_nK_n^{-1} Y_n \xrightarrow[n\to \infty]{a.s.} H-CK^{-1}Y.
\ee
The above almost sure convergence implies the convergence in law, which is what we wanted to show.
\end{proof}
\subsubsection{The conditional expectation is bounded}
\begin{lemma}\label{lem:Expbound}
Assume $\dim M\ge 4$. 
There is a constant $E>0$ and a neighborhood $B\supset \Delta$ of the diagonal, such that for all $(p,q)\in B\smallsetminus \Delta$, we have
\be 
I(p,q)=\E\kop\frac{|\Tilde{\Delta}X(p)|}{\|d_pX\|}\frac{|\Tilde{\Delta}X(q)|}{\|d_qX\|}\Bigg|
\begin{aligned}
X(p)&=0 \\ X(q)&=0
\end{aligned}
\pok\le E.
\ee
\end{lemma}
\begin{proof}
We will reduce the proof to the convergence of integrals treated in \cref{cor:IX}.
Indeed $I(p,q)$ is an integral of the form $I_{1,1}(F_{p,q},K_{p,q}),$ where $K_{p,q}$ is the covariance function of the conditioned Gaussian vector $X_{p,q}=(X_0,X_1,X_2),$ defined as
\be 
X_{p,q}=\begin{pmatrix}
    X_0
    \\
    X_1
    \\
    X_2
\end{pmatrix}
=
\qwe\begin{pmatrix}
    (\hess_pX,\hess_qX)
    \\
    d_pX
    \\
    d_qX
\end{pmatrix}
\Bigg|
\begin{aligned}
X(p)&=0 \\ X(q)&=0
\end{aligned}\ewq
=\qwe Y(X,p,q)\Bigg|
\begin{aligned}
X(p)&=0 \\ X(q)&=0
\end{aligned}
\ewq
\ee
The function $F_{p,q}$ satisfies the inequality:
\bega
F_{p,q}((H_1,H_2&),u,v):=
\\
&=|\tyu\mathrm{tr}(g_p^{-1}H_1)-\frac{1}{u^Tg_p^{-1}u}u^TH_1u\uyt \cdot \tyu \mathrm{tr}(g_q^{-1}H_2)-\frac{1}{v^Tg_q^{-1}v}v^TH_2v\uyt|
\\
&\le D|H_1|\cdot |H_2|
\eega
for some $D>2$, hence $F_{p,q}\in \m{F}_{D}$ for all $p,q\in M\times M.$ Moreover, $F_{p,q}$ depends continuously on $p,q$, in the sense of $\m{F}_{D}.$ 

Consider a sequence of pairs $(p_n,q_n)$ that approaches the supremum of $I(p,q)$. Since $M\times M$ is compact, by passing to a subsequence, we can assume that $(p_n,q_n)\to (p,q)$ and that $(q_n-p_n)\|p_n-q_n\|^{-1}\to v$, for some $v\in T_pM$. By \cref{lem:dd}, the sequence of Gaussian vectors $X_{p_n,q_n}$ converges in law to:
    \be 
X_{p,v}=
\qwe\begin{pmatrix}
    (\hess_pX,\hess_pX)
    \\
    d_pX
    \\
    d_pX
\end{pmatrix}
\Bigg|
\begin{aligned}
X(p)&=0 \\ d_pX(v)&=0
\end{aligned}
\ewq
\ee
Notice that the covariance $K_{p,v}$ of $X_{p,v}$ belongs to $\m{X}_{m-1,m-1},$ because of the nondegeneracy of $j^1_pX$. This is why we need that $m\ge 4$. Then, $K_{p,q}\in \m{X}_{3,3}$ for all $(p,q)\in M\times M$.  It follows that $(F_{p_n,q_n},K_{p_n,q_n})$ is a convergent sequence in $\m{F}_D\times \m{X}_{3,3},$ so we conclude that $\sup I_{p,q}=\lim_n I_{1,1}(F_{p_n,q_n},K_{p_n,q_n})$ must be finite, by \cref{cor:IX}.
\end{proof}
\begin{lemma}\label{lem:3DExpbound}
Assume $\dim M\ge 3$ 
Assume that for every $p\in M$ and $v\in T_pM$, we have that $X(p),d_pX(v),H_pX(v,v)$ form a non-degenerate Gaussian vector. 
There is a constant $r>0$ such that for all $q=\exp_p(v)$ with $v\in T_pM$ such that $\|v\|\le r$, we have
\be 
I(p,q)=\E\kop\frac{|\Tilde{\Delta}X(p)|}{\|d_pX\|}\frac{|\Tilde{\Delta}X(q)|}{\|d_qX\|}\Bigg|
\begin{aligned}
X(p)&=0 \\ X(q)&=0
\end{aligned}
\pok\le \frac{1}{r}\frac{1}{\|v\|}.
\ee
\end{lemma}
\begin{proof}
Let $p_n\to p$ and $q_n=\exp_{p_n}(t_nv)\to p$. We have that 
\be 
\frac{1}{t_n}[d_{p_n}X(v)|X(p_n)=X(q_n)=0]\to [H_pX(v,v)|X(p)=d_pX(v)=0].
\ee
Since the limit is non-degenerate by hypothesys, this implies that the variance of 
\be 
[ d_{p_n}X(v)|X(p_n) =X(q_n) = 0 ]
\ee 
has order at least $t_n^2$. From this and the non-degeneracy of $j^1_pX$, we deduce that the covariance matrix $K_{11}(n)$ of the Gaussian vector $[d_{p_n}X|X(p_n)=0,X(q_n)=0]$ satisfies
\be 
\frac{1}{\sqrt{\det K_{11}(n)}}=O\tyu\frac{1}{t_n}\uyt.
\ee
The same can be said for $[d_{q_n}X|X(p_n)=X(q_n)=0]$. Now, the proof can be concluded by \cref{thm:overkill}.
\end{proof}
 \begin{remark}
 The assumption essentially prevents the field to be  \emph{flat} in some but not all directions. To convince the reader that such assumption is necessary to have \cref{lem:3DExpbound}, consider a field $X\colon \R^3\to \R$ such that
 \be 
 X=\gamma_0+\gamma^Tp+\gamma_{11}\frac12(p_1^2+p_2^2)
 \ee 
 is determined by $5$ (not necessarily independent) Gaussian random variables $\gamma_0,\gamma=(\gamma_1,\gamma_2,\gamma_3),\gamma_{11}$. Then, for all $p,q$ such that $p_1=q_1=0$ and $p_2=q_2=0$, we have that  $V=[d_pX|X(p)=X(q)=0]$ is a Gaussian with a two dimensional support, at most, hence $\E\{|V|^{-2}\}=+\infty$. Moreover, we can easily compute that $[\tilde\Delta X(p)|X(p)=X(q)=0]=[\gamma_{11}|X(p)=X(q)=0]$. Thus, for all such pairs of distinct points $p,q$, we have 
 \be 
 I(p,q)=
 \E\kop\frac{|\gamma_{1,1}|^2}{\|V\|^2}\Bigg|
 \begin{aligned}
 X(p)&=0 \\ X(q)&=0
 \end{aligned}
 \pok=+\infty.
 \ee
 On the other hand, this assumption is not completely necessary, indeed if $X$ is an affine field (as above, but with $\gamma_{11}=0$), then $I(p,q)=0$.
 \end{remark}
\subsection{Estimating the Density}
\subsubsection{Pointwise value}
\begin{lemma}\label{lem:densitybound}
Then there is a constant $r>0$ such that for 
every $q=\exp_p(v)$ with $v\in T_pM$ such that $\|v\|\le r;$ we have that
\be
 r\frac{1}{\|v\|}\le \frac{1}{\sqrt{C(p,p)C(q,q)-C(p,q)^2}}\le \frac{1}{r}\frac{1}{\|v\|}.
\ee
\end{lemma}
\begin{proof}
First observe that as soon as $r$ is small eonugh, we have that $p\neq q$ belong to some neighborhood $B_r(\Delta)$ of the diagonal like that of \cref{lem:Mp1}, for which the denominator of the above expression does not vanish.
Let us consider the normalized field $p\mapsto Y(p)=C(p,p)^{-\frac12}X(p)\sim \m{N}(0,1)$, whose covariance function $K(p,q)$ satisfies the identity:
\be\label{eq:CK}
\frac{C(p,q)}{\sqrt{C(p,p)C(q,q)-C(p,q)^2}}=
\frac{K(p,q)}{\sqrt{1-K(p,q)^2}}.
\ee
Since $q$ and $p$ are assumed to be in a small enough neighborhood of the diagonal, we can assume that $C(p,q)$ and $K(p,q)$ are both bounded from below by a positive constant. Then, it is sufficient to bound the quantity $\frac{1}{\sqrt{1-K(p,q)^2}}.$
Notice also that $Y$ satisfies the hypotheses of the Lemma, in particular, the random vector $d_pY$ is non-degenerate, hence the bilinear form:
\be 
g^Y_p:=\E\kop d_pY (d_pY)^T\pok=d^2_{1,1} K(p,p)
\ee
is non-degenerate and it defines a Riemannian metric on $M$. Moreover, since $Y$ has constant variance, it follows that $d_2 K(p,p)=\E\{Y(p)d_pY\}=0,$ from which we deduce that $d^2_2 K(p,p)=-g_p^Y.$

Now, let $q(t)=\exp_p(tv),$ for some $v\in T_pM$ with $\|v\|=1.$ We have the following Taylor expansion as $t\to 0$:
\be 
K(p,q(t))=1-g_p^Y(v,v)\frac{t^2}{2}+O(t^3).
\ee
Plugging in into \eqref{eq:CK} we have
\bega
\frac{K(p,q)}{\sqrt{1-K(p,q)^2}}&=\frac{1-g_p^Y(v,v)\frac{t^2}{2}+O(t^3)}{\sqrt{1-(1-g_p^Y(v,v)\frac{t^2}{2}+O(t^3))^2}}
\\
&=
\frac{1+O(t^2)}{\sqrt{g_p^Y(v,v)t^2+O(t^3)}}
\\
&=
\frac{1}{|t|}\frac{1}{\sqrt{g_p^Y(v,v)}}(1+O(t))
\eega
\end{proof}
\subsubsection{Integral}
\begin{lemma}\label{lem:intcov}
Let $m\ge 2$. Let $B\subset M$ be such that $(B\times B)\cap \Delta^X=(B\times B)\cap \Delta$. Then, the integral below is finite:
    \be 
\int_{B\times B\smallsetminus\Delta}
\frac{1}{\mathrm{dist}(p,q)^{m-2}}\frac{dM(p)dM(q)}{\sqrt{C(p,p)C(q,q)-C(p,q)^2}}<\infty.
\ee
\end{lemma}
\begin{proof}
Let us consider the ball bundle $B_r(N\Delta):=\{(p,v): v\in T_qM, \|v\|\le r\},$ with $r>0$ given by \cref{lem:densitybound}. 
Taking a smaller $r>0$ if needed, we can assume that  the exponential map $\exp:B_r(N\Delta)\to M\times M,$ defined as $(p,v)\mapsto \exp_p(v)$, is a diffeomorphism onto its image, which we denote as $B_r(\Delta),$ that is thus a tubular neighborhood of $\Delta.$ By \cref{lem:Mp1}, we can moreover assume that $B_r(\Delta)$ does not contain elements of $\Delta^X$ other than the diagonal ones.

Let $s\colon M\times M\smallsetminus\Delta\to \R$ be the integrand function. The integral can be split into two parts: $\int_{B_r(\Delta)\smallsetminus\Delta}s+\int_{M\times M\smallsetminus B_r(\Delta)}s.$ By Cauchy-Schwartz, we have that $C(p,q)^2<C(p,p)C(q,q)$ for all $(p,q)\notin \Delta^X,$ so that the integrand is bounded on $B\times B\smallsetminus B_r(\Delta)$. 

This leaves us with the term $I_r:=\int_{B_r(\Delta)}s.$ Since the restriction of exponential map to a small enough neighborhood of the diagonal is a diffeomorphism, we can use it as a change of variables and write 
\be 
I_r=\int_{B_r(\Delta)}s(p,q) dM^2(p,q)=\int_{B_r(N\Delta)}s(p,\exp_p(v)) J(p,v) dTM(p,v),
\ee
where $J(p,v)$ is the Jacobian determinant of the exponential map, that is a $\mC^1$ function. Here we are integrating with respect to the Riemannian volume density of the canonical Riemannian metric on $TM,$ defined by the parallel transport. Notice that such metric induces the flat (i.e. constant) metric $g_p$ on each fiber $T_pM.$ Let us denote by $dT_pM$ the volume density of $T_pM,$ for any $p\in M,$ i.e., the Lebesgue measure determined by an orthonormal frame.

Denote by $J_{p,v}\pi>0$ the Jacobian of the projection map $\pi:B(N\Delta)\to M$ (it might not be $1$ when $M$ is not flat) and using the coarea formula, we obtain that
\be 
I_r= \int_{M}\int_{B_r(T_pM)}s(p,\exp_p(v))\frac{J(p,v)}{J_{p,v}\pi}dT_pM(v) dM(p).
\ee
It is easy to see that both Jacobians are bounded (because $M$ is compact) non-vanishing functions, so their ratio is bounded by a constant $b>0$, Moreover, we can write the inner-most integral in standard polar coordinates, since $T_pM$ is flat. Let $S(T_pM)$ denote the set of unit vectors in $T_pM.$ We obtain
\be\label{eq:here}
I_r\le b \int_{M}\int_{S(T_pM)}\tyu \int_{0}^r s(p,\exp_p(tv))t^{m-1}dt \uyt dS(T_pM)(v)dM(p).
\ee
Here is where we use \cref{lem:densitybound} above, which says that, for $r$ small enough, we have $s(p,\exp_p(tv)) \le \frac{1}{rt^{m-1}}$. 
Thus,
\bega 
I_r&\le \frac{b}{r} \int_{M}\int_{S(T_pM)}\tyu \int_{0}^r 1dt \uyt dS(T_pM)(v)dM(p)
\\
&\le
bk\vol(M)\vol(S^{m-1}) 
\eega
so that $I_r$ is finite when $m\ge 2.$
\end{proof}
\section{Malliavin-Sobolev regularity of the volume}\label{s:mallsob}
\subsection{Square-integrability of the volume}\label{s:L2}
In \cite{PolyAngst}, it is shown that the nodal volume of a stationary Gaussian field on $M=\mathbb{T}^m$ is in $L^{\frac{m+1}{2}-}(\PP)$. The following result ensures that the volume $V(X)=\vol_{m-1}(X^{-1}(0))\randin \R$ is a square-integrable random variable, also in the cases $m=2,3.$ 
\begin{corollary}\label{cor:volL2}
Let \cref{ass:1} prevail. Then $V(X)\in L^2(\PP)$ and
\be 
\E\kop V(X)^2\pok=\E\kop
\|d_pX\|\|d_qX\|
\Bigg|
\begin{aligned}
X(p)&=0 \\ X(q)&=0
\end{aligned}
 \pok 
 \frac{dM(p)dM(q)}{2\pi\sqrt{C(p,p)C(q,q)-C(p,q)^2}} <+\infty.
\ee
\end{corollary}
\begin{proof}
The first identity is the standard Kac-Rice formula for the second moment, see \cite[Theorem 6.9]{AzaisWscheborbook}, whose validity under \cref{ass:1} is granted by \cref{lem:krAlpha}. A direct application of \cite[Theorem 1.5]{gasstec2023}, with $p=2$, gives the boundedness. Indeed, the hypothesis of \cite[Theorem 1.5]{gasstec2023}, in such case, are equivalent to \cref{ass:1}.
\end{proof}
\subsection{Main Result}
In \cite{PolyAngst}, it is proven that the nodal volume of a stationary Gaussian field on $M=(S^1)^m$ or $M=[0,1]^m$ is in $\mathbb{D}^{\frac{m+1}{3}-}$, for $m\ge 3$. The following result is stronger in the cases $m=3,4,5$. In the other cases, the result reported in \cite{PolyAngst} should be considered stronger, in that a direct generalizaton of the proof of \cite{PolyAngst} to non-stationary Gaussian fields on an arbitrary compact manifold, with or without boundary, is relatively easy. However, the methods employed here are completely different. Let $V(X)=\vol_{m-1}(X^{-1}(0))\randin \R.$
The following result immediately implies point (ii) of \cref{t:mainintro}, when $m\ge 4$, and point (iii) of \cref{t:mainintro}, when $m=3$.

We recall the notation introduced in the statement of \cref{thm:firstvar}. Let $M$ be a $\mC^2$ compact Riemannian manifold with boundary $\de M$. Let $\n:\de M\to TM$ be the normal vector to the boundary, pointing outside the manifold. For any $f\in \mC^2(M)$ and we denote $\nu_f=\frac{\g f}{\|\g f\|}$ and
\be 
\Tilde{\Delta} f=\Delta f-\H f\tyu \frac{\g f}{\|\g f\|},\frac{\g f}{\|\g f\|}\uyt,
\ee
as functions, defined on $\{\g f\neq 0\}$. 
\begin{theorem}\label{thm:D12}
Let \cref{ass:1} prevail, with $M$ having dimension $m\ge 3$. 
\begin{enumerate}[$\bullet$]
\item Moreover, if $m=3$, we also assume that for every $(p,v)\in TM$, the three random variables $ X(p),d_pX(v),\H_pX(v,v)$ form a non-degenerate Gaussian vector.
\end{enumerate}
Then, the nodal volume $V(X) = \vol_{m-1}(X^{-1}(0))$ is in $\mathbb{D}^{1,2}.$ 
In particular, the stochastic differential $\dmlv V\randin \HX^*$ coincides almost surely with the Fr\'{e}chet differential:
\bega\label{eq:mainfrechet} 
\dmlv V=d_XV\big|_{\HX}(\cdot)&=\int_{X^{-1}(0)}\tyu\frac{\tilde\Delta X(p)}{\|d_pX\|^2}\uyt \delta_p(\cdot) \mathcal{H}^{m-1}(dp)
\\
&+
\int_{X^{-1}(0)\cap \de M}\tyu
\frac{ g_p\tyu \n,\nu_X\uyt }{\|d_p(X|_{\de M})\|}
\uyt \delta_p(\cdot) \mathcal{H}^{m-2}(dp),
\eega
and is in $L^2(\HX^*)$ with finite square norm
$\|\dmlv V\|^2_{L^2}=\E\kop \|d_XV\|^2_{\HX}\pok=\dots$
\bega\label{eq:KRDerivative}
\dots
&=\int_{M\times M}\E\kop\frac{\Tilde{\Delta}X(p)}{\|d_pX\|}\frac{\Tilde{\Delta}X(q)}{\|d_qX\|}\Bigg|
\begin{aligned}
X(p)&=0 \\ X(q)&=0
\end{aligned}
\pok
\frac{C(p,q)dM(p)dM(q)}{2\pi\sqrt{C(p,p)C(q,q)-C(p,q)^2}}
\\
&=
\int_{\de M\times \de M}\E\kop
g_p( \n,\nu_X) \cdot g_q( \n,\nu_X)
\Bigg|
\begin{aligned}
X(p)&=0 \\ X(q)&=0
\end{aligned}
\pok
\frac{C(p,q)d\de M(p)d\de M(q)}{2\pi\sqrt{C(p,p)C(q,q)-C(p,q)^2}}
\eega
Where the integrand in the latter expression is continuous almost everywhere.
\end{theorem}
\begin{proof}
We divide the proof in three steps, corresponding to the condtions ($i$),($ii$) and ($iii$) of \cref{def:Dspace}.
\subsubsection*{Step $(i)$}
The fact that $V$ is ray absolutely continuous is a direct consequence of \cref{thm:transcurve} and \cref{thm:DeterministPicture}.
\subsubsection*{Step $(ii)$}
This follows from the Fr\'{e}chet differentiability almost everywhere. Indeed, let us define the candidate stochastic differential $\dmlv V\colon E\to \HX^*$ to be the Fr\'{e}chet differential:
\bega 
\dmlv V(f):=d_fV\big|_{\HX}(\cdot)&=\int_{f^{-1}(0)}\tyu\frac{\tilde\Delta f(p)}{\|d_pf\|^2}\uyt \delta_x(\cdot) \mathcal{H}^{m-1}(dp)
\\
&+\int_{f^{-1}(0)\cap \de M}\tyu
\frac{ g_p( \n,\nu_f) }{\|d_p(f|_{\de M})\|}
\uyt \delta_p(\cdot) \mathcal{H}^{m-2}(dp),
\eega
for all $f\in \m{U}.$ Moreover, $\dmlv V$ is a continuous function on $\m{U}.$ By \cref{lem:bulinskaya}, we have that $\m{U}$ has full measure, hence the above assignement defines a measurable function $\dmlv V$ almost everywhere and thus it can be extended to a measurable function on $E$, so that it defines a random variable. The choice of the extension on $\m{W}=E\smallsetminus\m{U}$ is irrelevant. \cref{thm:firstvar} implies that we have 
\be 
\left|\frac{V(X+th)-V(X)}{t}-d_XV(h) \right|\to 0
\ee
almost surely, thus the convergence in probability holds and $\dmlv V\randin \HX^*$ is indeed the stochastic derivative of $V(X).$
\subsubsection*{Step (iii)}
We have to show that the random variable $\|\dmlv V(X)\|_{\HX}=\|d_XV\|_{\HX}$ is in $L^2$. 
By \cref{cor:expDV}, we have that, if $Z=X^{-1}(0),$ then
\bega 
 \|d_XV\|^2_{\HX}
&=\int_{Z\times Z}\frac{\Tilde{\Delta}X(p)}{\|d_pX\|^2}\frac{\Tilde{\Delta}X(q)}{\|d_qX\|^2}
C(p,q)dZ(p)dZ(q)
\\
 &+
 \int_{\de Z}\int_{\de Z} \frac{g_p(\n,\nu)g_q(\n,\nu)  }{\|d_p(X|_{\de M})\|\|d_q(X|_{\de M})\|}C(p,q)d\de Z(p)d\de Z(q)
\eega
By applying \cref{lem:krAlpha} to $X$ and to $X|_{\de M}$, we deduce a Kac-Rice formula for computing $\E\kop \|d_XV\|^2_{\HX}\pok$, corresponding exactly to that in the statement of the theorem.
The integral can be restricted to $M\times M\smallsetminus \Delta^X$ , where  $\Delta^X=\kop(p,q):(X(p),X(q)) \text{ is degenerate}\pok$, which we know to have measure zero by \cref{lem:Mp1}. 

Now, we will show that such formula is also bounded. To do so, we need two Lemmas, that are proved above: \cref{lem:Expbound} to bound the conditional expectation and \cref{lem:intcov} to bound the rest. In particular, the former is the most demanding one and its proof depends on \cref{cor:IX}, which is proved in \cref{sec:boringauss}.

Let us first consider the case when $\de M=\emptyset$.
We want to reduce to a situation such that $\Delta^X=\Delta$. To this end, we cover $M$ with a finite number of open subsets $B_1,\dots, B_N$ such that for all $i=1,\dots, N$ we have $\Delta^X\cap (B_i\times B_i)=\Delta \cap (B_i\times B_i)$. This is possible because of \cref{lem:Mp1}.
Then, we have almost surely
\bega 
\|d_XV\|^2_{\HX}
&\le 
\|C\|_{\mC^0}\sum_{i,j}\int_{Z\cap B_i\times Z\cap B_j}\left|\frac{\Tilde{\Delta}X(p)}{\|d_pX\|^2}\right|\left|\frac{\Tilde{\Delta}X(q)}{\|d_qX\|^2}
\right|dZ(p)dZ(q)
\\
&= 
\|C\|_{\mC^0}\tyu \sum_{i}\int_{Z\cap B_i}\left|\frac{\Tilde{\Delta}X(p)}{\|d_pX\|^2}
\right|dZ(p)
\uyt^2
\eega
From this we see that it is sufficient to show that the random variables defined as $\a_i={\int_{Z\cap B_i}\left|\frac{\Tilde{\Delta}X(p)}{\|d_pX\|^2}
\right|dZ(p)}$ are in $L^2$, for all $i$. Using again \cref{lem:krAlpha}, we are reduced to show the finiteness of the following integral:
\be 
\E\qwe \a_i^2
\ewq
=\int_{B\times B\smallsetminus \Delta}\E\kop\left|\frac{\Tilde{\Delta}X(p)}{\|d_pX\|}\frac{\Tilde{\Delta}X(q)}{\|d_qX\|}\right|\ \Bigg|
\begin{aligned}
X(p)&=0 \\ X(q)&=0
\end{aligned}
\pok
\frac{(2\pi)^{-1}dM(p)dM(q)}{\sqrt{C(p,p)C(q,q)-C(p,q)^2}},
\ee

When $m\ge 4$, by \cref{lem:Expbound}, we have that the conditional expectation that appears in the formula is bounded, while if $m=3$, \cref{lem:3DExpbound} implies that the expectation is bounded by $\frac{1}{\mathrm{dist}(p,q)}$. Thus, we are reduced to prove finiteness of 
\be 
I=\int_{B_i\times B_i\smallsetminus\Delta}
\frac{1}{\mathrm{dist}(p,q)}s(p,q)dM(p)dM(q), 
\quad s(p,q)=\frac{1}{\sqrt{C(p,p)C(q,q)-C(p,q)^2}}
\ee
\cref{lem:intcov} shows that $I$ is finite when $m\ge 3.$
This proves that $\E\{\|d_XV\|^2_{\HX}\}<+\infty.$


This proves that $\E\{\|d_XV\|^2_{\HX}\}<+\infty$ when $\de M=\emptyset$, but in fact, we proved that, denoting by $d_X^{\de}V$ the boundary term in \cref{eq:mainfrechet}, $\|d_XV-d_X^{\de}V\|_{\HX}\in L^2$. To prove the general case, we only have to show that $\|d_X^{\de}V\|_{\HX}\in L^2$. This can be done, using the same argument: the expectation term in the boundary integral is clearly bounded (indeed, $g_p(\n,\nu_X)\le 1$), so we can conclude, as before, with \cref{lem:intcov}, applied to $X|_{\de M}$, since $\dim \de M\ge 2$.  
\end{proof}
\begin{corollary}\label{thm:D12corner}
Let $Q$ be a $\mC^2$ manifold with corners of dimension $m\ge 3$ (for instance, $Q=[0,1]^m$). Let $X\randin \mC^2(Q)$ be Gaussian, with the property that:\begin{enumerate}[$\bullet$]
    \item For any point $p\in Q$ and any tangent vector $v\in T_pQ$, the Gaussian vector $j_p^1X:=\tyu X(p),d_pX(v)\uyt\randin \R^2$ is non-degenerate.
    \item If $m=3$, we also assume that for every $(p,v)\in TM$, we have that $X(p),d_pX(v),H_p(v,v)$ form a non-degenerate Gaussian vector.
\end{enumerate}
Then, the nodal volume $V(X)=\vol^{m-1}(X^{-1}(0))$ is in $\mathbb{D}^{1,2}.$ 
\end{corollary}
\begin{remark}
An immediate consequence is that the same statement holds true for random fields defined on any geometric object that can be written as a finite union of manifolds with corners. In particular, let $Q\subset B^m$ be a semialgebraic subset (see \cite{BCR:98}) of the unit ball $B^m$, entirely contained in its interior and let $X\randin \mC^2(B^m)$ satisfy the hypotheses of the above theorems. Then $\vol_{m-1}(X^{-1}(0)\cap Q)$ is in $\mathbb D^{1,2}$.
\end{remark}
\begin{proof}
It is possible to construct an increasing sequence of $\mC^2$ manifolds with boundary $M_n\uparrow Q$, all parametrized by a fixed compact manifold with boundary $M$, via a family of $\mC^2$ embeddings $\Phi_n\colon M\to Q$, such that $\Phi_n(M)=M_n$ and such that $\Phi_n\to \Phi$ in $\mC^\infty(M,Q)$, where $\Phi(M)=Q$. Moreover, the construction can be made in such a way that there is a partition of the boundary $\de M=N\sqcup B$, such that $\phi:=\Phi|_B$ is an embedding with image $\de Q$, the $m-1$ dimensional stratum of $Q$ and $N$ has zero measure. To see that this construction is possible, observe that it is enough to prove when $Q=[0,1]^m$, 
 in which case it can be done explicitely. Denote $\phi_n:=\Phi_n|_B$. 
 Let $V_n(f):=\vol_{m-1}(f^{-1}(0)\cap M_n)$. By dominated convergence, we have that $V_n(X)\to V(X)$ almost surely. By \cref{cor:volL2}, it follows that the convergence holds in $L^2$ as well. Exploiting the completeness of the space $\mathbb D^{1,2}$, to show that $V(X)\in \mathbb D^{1,2}$ it is thus sufficient to prove that the sequence of derivatives $d_XV_n$ is convergent in $L^2$. This is true for the non-boundary term. Let us denote by $d_X^\de V$ the boundary term:
 \bega 
d_XV_n^\de
&=
\int_{X^{-1}(0)\cap \de M_n}\tyu
\frac{ g_p\tyu \n,\nu_X\uyt }{\|d_p(X|_{\de M})\|}
\uyt \delta_p(\cdot) \mathcal{H}^{m-2}(dp)
\\
&=
\int_{(X\circ \phi_n)^{-1}(0))}\tyu
\frac{ g_{\phi_n(x)}\tyu \n,\nu_X\uyt }{\|d_x(X\circ \phi_n)\|}
\uyt\Big|_{p=\phi_n(x)} \delta_{\phi_n(x)}(\cdot) J_x\phi_n\mathcal{H}^{m-2}(dx)
 \eega
 So that applying the Kac-Rice formula to the field $(X\circ \phi,X\circ \phi_n)\colon B\times B\to \R^2$, we get
\bega 
\E \langle d_XV_n^\de - d_XV^\de\rangle_{\HX}^2&=
\\
&=
\int_{B\times B}\E\kop
g_p( \n,\nu_X) \cdot g_q( \n,\nu_X)
\Bigg|
\begin{aligned}
X(p)&=0 \\ X(q)&=0
\end{aligned}
\pok\times 
\\
&\times
\frac{C(p,q)}{2\pi\sqrt{C(p,p)C(q,q)-C(p,q)^2}}\Bigg|_{p=\phi_n(x),q=\phi(y)}\!\!\!\!\!\!\!\!\!\!\!\!\!\!\!\!\!\!dB(x)dB(y).
\eega
Since $\phi_n\to \phi$ converges to an embedding, the formula above is convergent, thus we conclude.
 \end{proof}
\subsection{The two-dimensional case}
In this section, we will prove the points (ii) and (iv) of \cref{t:mainintro}. The former, was partly proven in \cref{thm:D12}, and it only remains to show that in dimension $m\in \{2,3\}$ we have $V(X)\in \mathbb{D}^{1,1}$ assuming only \cref{ass:1}, which is the content of \cref{thm:D11} below.
The second, point (iv), is equivalent to \cref{thm:D12boundary} below, which establishes that a necessary condition for $V(X)$ to be in $\mathbb{D}^{1,2}$ is that the random curve $Z=X^{-1}(0)$ is almost surely $\mC^1$-isotopic (see \cref{def:c1iso}) to a fixed deterministic one.
\subsubsection{In dimensions $2$ and $3$, $V$ is in $\mathbb{D}^{1,1}$}
The following theorem completes the proof of point (ii) of \cref{t:mainintro}. 
\begin{theorem}\label{thm:D11}
Let \cref{ass:1} prevail. Then, $V(X)=\vol_{m-1}(X^{-1}(0))$ is in $\mathbb D^{1,1}$.
\end{theorem}
\begin{proof}
The first two steps $(i)$ and $(ii)$ in the proof of \cref{thm:D12} remain true in this setting, so it is enough to prove $(iii)$. 
Observe that since the inclusion $\HX\subset \mC^0(M)$ is continuous, it follows that 
\be 
\|h\|_{\HX}\le 1 \implies \max_{p\in M}|h(p)|\le 1.
\ee 
To prove the integrability of $d_XV$, we proceed as follows:
    \bega 
&\E\kop \|d_XV\|_{\HX}\pok \\
&=
\E\kop \sup_{\|h\|_{\HX}\le 1}\langle d_XV,h\rangle\pok
\\
&=
\E\kop \sup_{\|h\|_{\HX}\le 1}\int_{Z}\frac{\Tilde{\Delta}X(p)}{\|d_pX\|^2} h(p) dZ(p)\pok 
+
\int_{\de Z}\tyu
\frac{ g_p( \n
,\nu_X) }{\|d_p(X|_{\de M})\|}
\uyt h(p) d\de Z (p)
\\
&\le
\E\int_{Z}\frac{|\Tilde{\Delta}X(p)|}{\|d_pX\|^2}dZ(p)
+
\E\int_{\de Z}\tyu
\frac{ 1 }{\|d_p(X|_{\de M})\|}
\uyt  d\de Z (p)
\\
&=
\int_{M}\E\kop\frac{|\Tilde{\Delta}X(p)|}{\|d_pX\|}\Bigg|
X(p)=0
\pok
\frac{1}{\sqrt{2\pi C(p,p)}} dM(p)
+
\int_{\de M}
\frac{1}{\sqrt{2\pi C(p,p)}} d\de M(p).
    \eega
    By \cref{thm:overkill}, we know that $I_{1,0}$ is continuous on $\mathcal{F}_D\times \m{X}_{m,0}$, whenever $m\ge 2$. We deduce that the integral is bounded.
\end{proof}
\subsubsection{When is $V$ in $\mathbb{D}^{1,2}$?}
The idea for proving \cref{thm:D12boundary} starts from the observation that \cref{def:Dspace} implies a certain Sobolev regularity on almost every line in the space $E$. Then, we will exploit our study, done in \cref{sec:transverse}, of the behavior of $V$ along transverse curves to conclude.
\begin{lemma}\label{lem:ray}
Let $E$ be a Banach space equipped with a centered Gaussian measure $\mu=[X]$ having full-support and let $\HX\subset E$ be its Cameron-Martin Hilbert space. Let $p\in [1,+\infty]$. 
    If $V\in \mathbb{D}^{1,p}(\mu)$, then $V$ is \emph{ray-$W^{1,p}$}, that is: for every $h\in \HX$, there is a subset $N_h\subset E$, with $\mu(N_h)=1$, such that for all $x\in N_h$, the function $t\mapsto V(x+th)$ belongs to the Sobolev space $W^{1,p}(I,\R)$, for any bounded interval $I\subset \R$.
\end{lemma}
\begin{proof}
   Let $h\in \HX$. Then, we can split the space $E$ and the measure $\mu$ into a direct product $E=\R h\times E_0$ and the measure $\mu=\mu_h\otimes \mu_0$. \cref{def:Dspace} implies that for all $x$ in a full measure subset $N_h\subset E$, the function $\f_x:t\mapsto V(x+th)$ belongs to $\mathbb{D}^{1,p}(\mu_h)$. Thus, we have that $\f_x$ and its Sobolev derivative $\f_x'(t)=\langle \dmlv V(x+th),h\rangle$ are in $L^p(\mu_h)$. Here, $\mu_h$ is a Gaussian measure, hence both functions are also in $L^p(I,\R)$ for any bounded interval $I\subset \R$. 
\end{proof}

The following theorem is a reformulation of point (iv) of \cref{t:mainintro}.
\begin{theorem}\label{thm:D12boundary}
Let \cref{ass:1} prevail, with $\dim M=2$ and assume the following: 
\begin{enumerate}[$\bullet$]
    \item there are no pairs of distinct boundary points $p,q\in \de M$, such that the Gaussian vectors $ j_p^1X,j_q^1X$ are fully correlated.
\end{enumerate}
Then, if the volume random variable $V(X)=\vol_{m-1}(X^{-1}(0))$ is ray-$W^{1,2}$ then there exists a deterministic curve $Z\subset M$, such that 
\be 
\PP\kop X^{-1}(0) \text{ is $\mC^1$-isotopic to } Z\text{ in $M$}\pok =1.
\ee
If this does not hold, then $V(X)\notin \mathbb{D}^{1,2}$. 
\end{theorem}
\begin{proof}
By \cref{lem:ray} if $V(X)$ is not $ray-W^{1,2}$, then it is not in $\mathbb{D}^{1,2}$. Let us assume that $V(X)$ is $ray-W^{1,2}$. Let $f_0,f_1\in \m U$, then there exists a curve $t\mapsto \tilde f_t=f_0+th\in E$, 
for some $\tilde f_i$ close to $f_i$. 
Then, for almost every such curves we must have that the function $t\mapsto V(c(t))$ is in $W^{1,2}$. By \cref{cor:notD12}, this implies that $f_t\subset \m U$. This proves that $\m U$ is connected. Moreover, for any continuous curve $f_t\in \m U$, we have that $f_t^{-1}(0)$ is $\mC^1$-isotopic to $Z:=f_0^{-1}(0)$ in $M$ for all $t$, by \cref{lem:c1iso}. Since $\m U$ has probability one, this concludes the proof.
\end{proof}
\section{On the singular part of the law}\label{s:singular}
As announced in Section \ref{ss:motivintro}, the forthcoming Theorem \ref{thm:sing} provides a characterization of the law of random nodal volumes whenever the underlying Gaussian field has full support. We will see in Remark \ref{r:bestremarkever} that our strategy for proving such a statement answers some questions left open in \cite{PolyAngst}.

\begin{theorem}\label{thm:sing}
Let \cref{ass:1} prevail, and assume that the topological support of $X$ is $E=\mC^2(M)$. Then, the law  of the nodal volume $V(X) = \mathcal{H}^{m-1}(X^{-1}(0))$ takes the form 
\be 
\PP\kop V(X) \in I\pok=\int_I \rho(v)dv + P\cdot \delta_0(I),
\ee
for all $I\subset \R$ Borel subset, where $0<P=\PP\kop X^{-1}(0)=\emptyset \pok<1$, and $\rho:[0,+\infty[\to [0,+\infty[$ is an $L^1$ function whose topological support 
is the set of nonnegative real numbers $[0,+\infty)$, and such that $\int_0^\infty \rho(x)dx = 1-P$. If $V(X)\in \mathbb{D}^{1,2}$, then formula \eqref{e:densitivan} holds for almost every $x\in [-c,+\infty]$, with $c = \mathbb{E}(V)$ 
 and
$$
{\pi}(x) = \frac{\rho(x)}{1-P},
$$
with a function $g$ satisfying the three properties {\rm (a)--(c)} in the statement of Theorem \ref{t:introdensitivan}.
\end{theorem}
\begin{remark}\label{r:bestremarkever}
Roughly speaking, the conclusion of \cref{thm:sing} implies that the nodal volume $V(X)$ has an absolutely continuous distribution, conditionally on the event that $X$ vanishes somewhere in $M$.
A parallel study of the singular part of the law was made also in \cite[Section 4.4]{PolyAngst}. The authors prove \cite[Corollary 3]{PolyAngst} that, under the hypothesis that $M=\mathbb{T}^m$, that the field $X$ is smooth, stationary and with full support, if the kernel of a certain operator $\mathcal{L}^*$ has dimension $1$, then:
\begin{eqnarray*}
&& \PP\kop \dmlv V(X)=0 \implies \text{$X$ has constant sign (i.e., $V(X)=0$)}
\pok\\
&& = \PP\kop \{ \dmlv V(X)\neq 0\} \cup \{ V(X)=0\} \pok=1.
\end{eqnarray*}
This is equivalent to Claim \eqref{eq:finalclaim} below, which we prove here in more generality. In particular, there is no need to check if $\dim( \ker (\mathcal{L}^*))=1$, thus confirming the conjecture (formulated in \cite[end of Section 4.4]{PolyAngst}) that such a condition is indeed not necessary.
\end{remark}
\begin{proof}
The strategy of the proof is to use the Bouleau-Hirsch criterion \cite[Theorem 4]{PolyAngst}, in combination to a study of the zeroes of the Malliavin derivative $\dmlv V(X)$, which we know to exist, at least in the $L^1$ sense, for all $m\ge 2$ by \cref{thm:D11}. The key idea is to look at the second variation of the volume, which for minimal hypersurfaces is expressed by a standard differential geometric formula, see \cref{eq:secondvar} below. Then, we will apply \cref{lem:blind} to the set $\m V\subset E$ of zeroes of the Malliavin derivative.

To improve its readability, the proof is divided into eight steps.
\begin{enumerate}[wide]
\item Given that $X$ has full support, there is a positive probability that $X>1$ on the whole manifold and a positive probability that $X(p)\cdot X(q)<-1$ for a pair of distinct points $p,q$. It follows that $0<P=\PP\kop X^{-1}(0)=\emptyset \pok<1$. Therefore the law $\mu$ of the nodal volume certainly is a sum $\mu=\mu_0+P\delta_0$, for a positive measure $\mu_0$, that can be defined as follows
\be 
\mu_0(I)=\E\qwe  1_I\tyu V(X) \uyt\cdot 1_{V^{-1}]0,+\infty[}(X)\ewq.
\ee
\item Consider the Bouleau-Hirsch criterion as reported in \cite[Theorem 4]{PolyAngst}. In order to include the two dimensional case, we will apply the theorem with $p=1$, relying on the fact that $V(X)\in \mathbb D^{1,1}$ by \cref{thm:D11}. In our setting, \cite[Theorem 4]{PolyAngst} states that there exists a density function $\rho_0\colon [0,+\infty[\to [0,+\infty[$, integrable and such that for every $I\subset \R$ Borel subset,
\be 
\E\qwe 1_I\tyu V(X)\uyt\cdot \|\dmlv V(X)\|_{\HX}\ewq=\int_I\rho_0(t)dt.
\ee
An immediate consequence is that the measure $\nu$, defined as
\be 
\nu(I)=\E\qwe 1_I\tyu V(X)\uyt\cdot  1_{\|\dmlv V(X)\|_{\HX}>0}\ewq,
\ee
is absolutely continuous with respect to the Lebesgue measure. Let $\rho$ be its density function. We will conclude the proof by showing that $\nu=\mu_0$, that is equivalent to prove that:
\be\label{eq:finalclaim}
\textbf{Claim:}\quad \PP\kop V(X)> 0,\ \dmlv V(X)=0\pok =0.
\ee
The rest of the proof is devoted to prove Claim \eqref{eq:finalclaim}.
\item Since $X$ has full support $\mathcal{C}^2(M)=\overline{\HX}$, the almost sure identity $\dmlv V(X)= d_XV|_{\HX}=0$ implies that $Z=X^{-1}(0)$ is either empty or a minimal hypersurface, that is, a hypersurface whose mean curvature $H_Z$, defined as in \eqref{eq:mean}, vanishes identically. 
\item The second order differential of $V$ in the direction $h\in \mC^2(M)$, computed at point $f\in \m U$, such that $Z=f^{-1}(0)$ is a $\mC^2$ minimal hypersurface, can be obtained from the general formula for the second variation of the volume \cite[Equation at the bottom of page 8]{li2012geometric}, reasoning as in \cref{prop:boundarypuff}:
\be\label{eq:secondvar}
d^2_fV(h,h):=\frac{d^2}{ds^2}\Big|_{s=0} V(f+sh)=\int_Z \|d \psi\|^2-\psi^2\tyu\|\mathrm{II}\|^2+\mathrm{Ric}(\nu,\nu)\uyt dZ,
\ee
where $\psi=-\|df\|^{-1}h|_Z
\in \mC^1(Z)$ is determined by the change of variation formula given in \cref{eq:changeofvariation}; $\mathrm{II}=\|df\|^{-1}\H f
$ is the second fundamental form of $Z$ (\cref{eq:II}); $\mathrm{Ric}$ is the \emph{Ricci curvature} (see \cite{leeriemann}) of the ambient manifold $M$ and $\nu =\|d f\|^{-1}\g f$. 
\item Fix $f$ as above, let $Z=f^{-1}(0)$ be minimal, and consider the subspace $\Psi_Z\subset \mC^1(Z)$ defined as
\be 
\Psi_Z:=\kop \psi=-\|df\|^{-1}h|_Z: h\in \HX \pok.
\ee
Recall that $f\in \m U$ if and only if $\|df\|$ has no zeroes on $Z$, thus $\|df\||_Z$ is a $\mC^1$ function on $Z$.
Since $\HX$ is dense in $\mC^2(M)$, we have that $\Psi_Z$ is dense in $\mC^1(Z)$.
\item Assume that $d^2_fV(h,h)=0$ for all $h\in \HX$, then due to the density of $\Psi_Z$ in $\mC^1(Z)$, we would have that the right hand side of \cref{eq:secondvar} vanishes for all $\psi\in \mC^1(Z)$, indeed, since $\|\mathrm{II}\|^2$ and $\mathrm{Ric}(\nu,\nu)$ are continuous functions on $Z$, the latter expression is continuous with respect to $\psi\in \mC^1(Z)$. In particular, by Stokes-Green divergence formula, we have that 
\be \label{eq:green}
\int_Z\psi\tyu \Delta^{(Z)}\psi +
\tyu\|\mathrm{II}\|^2+\mathrm{Ric}(\nu,\nu)\uyt\psi
\uyt
dZ=0, \quad \forall \psi \in \mC^2(Z)
\ee
where $\Delta^{(Z)}$ is the Laplace-Beltrami operator of the $\mC^2$ compact Riemannian manifold $Z$. Set $C:=\tyu\|\mathrm{II}\|^2+\mathrm{Ric}(\nu,\nu)\uyt$, a continuous function on $Z$, and notice that \cref{eq:green} is an identity of the form $A(\psi,\psi)=0$ for a symmetric bilinear form $A$, thus it implies that $A=0$, that is,
\be\label{eq:absurd} 
\Delta^{(Z)}\psi +
C\psi=0, \quad \forall \psi \in \mC^2(Z).
\ee
Clearly this is an absurd statement, unless $Z$ is empty. This proves that if $f\in \m U$ is such that $Z=f^{-1}(0)\neq \emptyset$ is minimal, then there exists at least one $h\in \HX$ (depending on $f$) such that $d^2_fV(h,h)\neq 0$.
\item Observe that $d^2_fV(h,h)$ is continuous with respect to $f,h\in \mC^2(M)$, so that if $d_{f_0}(h,h)\neq 0$, then there is a neighborhood $U_{f_0}$ of $f_0$ such that $d_{f}(h,h)\neq 0$ for all $f\in U_{f_0}$. Let $\delta>0$ and $O_{f_0}$ be a smaller neighborhood of $f_0$ in $U_{f_0}$, such that 
\be 
f\in O_{f_0} 
\text{ and }|s|<\delta\implies f+sh\in U_{f_{0}}.
\ee 
Then for all $f\in O_{f_0}$, the function $s\mapsto V(f+sh)$ has no flexes in $(-\delta,\delta)$, meaning that its derivative is strictly monotone, so that it can have at most one zero. In particular, we proved the following:
``Define $\m V:=\kop f \in \m U\colon V(f)\neq 0, d_fV|_{\HX}=0\pok$. 
For every $f_0\in \m V$, there is an open set $O_{f_0}\subset \m U$, $\delta>0$ and $h\in \HX$ such that for every $f\in O_{f_0}$, the set
\be 
S(f,\delta,h)=\kop s\in (-\delta,\delta): f+sh \in \m V\pok
\ee
is finite.'' By \cref{lem:blind} applied with $e=0$, recalling that $\PP\kop X\in \m U\pok=1$, we conclude that Claim \eqref{eq:finalclaim} holds and the proof of the first part of the statement is concluded. In particular, the support of the law of $V(X)$ and thus of $\rho$, must necessarily contain the image $V(\m U)$, which is the the whole positive real line $[0,+\infty)$. 
\item Now assume that $V = V(X)\in \mathbb{D}^{1,2}$. To prove the final claim, we start by recalling the standard estimates: $\mathbb{E}\big[ \langle {\rm D}_{\mathscr{M}}V, - {\rm D}_{\mathscr{M}}L^{-1}V\rangle_{\mathcal{H}_X} \, |\, \bar{V}\big]\geq 0$, a.s.-$\mathbb{P}$, and
$$
\mathbb{E} \Big\{ \mathbb{E}\big[ \langle {\rm D}_{\mathscr{M}}V, - {\rm D}_{\mathscr{M}}L^{-1}V\rangle_{\mathcal{H}_X} \, |\, \bar{V}\big]  \Big\} \leq \mathbb{E} \|D_\mathscr{M} V\|^2_{\mathcal{H}_X}  <\infty,
$$
where the last relation follows from the fact that $V\in \mathbb{D}^{1,2}$, and we have used Cauchy-Schwarz and the bound $\mathbb{E} \|D_\mathscr{M}L^{-1} V\|^2_{\mathcal{H}_X} \leq \mathbb{E} \|D_\mathscr{M}  V\|^2_{\mathcal{H}_X} $; see  \cite[Proposition 2.9.4 and Lemma 5.3.7]{nourdinpeccatibook}. Now denote by $g$ a version (with respect to the law of $\bar{V}$) of the mapping appearing in \eqref{e:npkernel} and observe that the law of $\bar{V}$ is characterized as follows: for every Borel set $I$,
$$
\mathbb{P}[\bar{V}\in I] = P \delta_{-c}(I) + (1-P) \int_I \bar{\pi}(x) \,dx,
$$
where $\bar{\pi}(x) := \pi(x+c)$, in such a way that the support of $\bar{\pi}$ coincides with the interval $[-c ,+ \infty]$ (note that, necessarily, $0\in (-c, +\infty)$). Considering smooth test functions with support contained in $(-c, +\infty)$, and reasoning as in \cite[Proof of (3.17)]{nourdinviens}, one deduces that, necessarily, for $dx$-almost every $x\in (-c, +\infty)$,
$$
g(x) = \frac{\int_x^{\infty} x\, \bar{\pi}(x) dx }{\bar{\pi}(x)} :=  \frac{\varphi(x) }{\bar{\pi}(x)},
$$
and the conclusion is deduced as in \cite[pp. 2294-2295]{nourdinviens}, by exploiting the fact that the function $\varphi(x) $ defined above is strictly positive and continuous on the interval $(-c, \infty)$ (this follows from the following facts: (i) $(1-P)\varphi(-c) =cP>0 $, (ii) $\varphi$ is strictly increasing on $(-c, 0)$ and strictly decreasing on $(0,\infty)$, and (iii) as $x \to (\infty)$, $\varphi(x)$ converges to zero). 
\end{enumerate}
\end{proof}

%
%
\begin{appendix}
\section{Definition of ray absolutely continuous}\label{apx:ray}
We will prove that the definition \emph{ray absolute continuity} that we gave in \cref{def:Dspace} is equivalent to that of \cite[Def. 5.2.3]{bogachev}. 
\begin{lemma}
\begin{enumerate}[$(i)$]
The following statements are equivalent:
    \item For every $h\in \HX,$ there is a set $N_h\subset E,$ with $\mu(N_h)=0,$ such that for all $x\in E\smallsetminus N_h,$ the function $t\mapsto V(x+th)$ coincides $dt$-almost everywhere with an absolutely continuous function $t\mapsto \f_{x,h}(t)$.
    \item For every $h\in \HX,$ there exists a function $V_h:E\to \R$ such that $V=V_h$ $\mu$-a.e. and, for every $x\in E,$ the mapping $t\mapsto V_h(x+th)$ is absolutely continuous.
\end{enumerate}
\end{lemma}
\begin{proof}
Since $h$ is fixed in both statement, we can assume that $E=\R h \oplus E_0,$ and that $\mu=[\gamma h+X_0]$ where $X_0\randin E_0$ is a full-support Gaussian random element with Cameron-Martin space $H_0=h^\perp$ dense in $E_0,$ and $\gamma\sim N(0,\sigma^2)$ is an independent nondegenerate Gaussian random variable.

$(i)\implies (ii).$ By Tonelli Theorem, if $\mu(N_h)=0,$ there must be some $t_0\in \R$ such that $\mu((t_0h+E_0)\cap N_h)=0.$ Without loss of generality, we can assume that $t_0=0,$ for simplicity. 

Now, we define a function $V_h\colon E \to \R$ as follows. If $x\in E_0\cap N_h,$ then $V_h(x+th):=0$ for all $t\in \R;$ if $x\in E_0\smallsetminus N_h,$ then $V_h(x+th):=\f_{x,h}(t)$. The function $V_h$ is thus a well defined measurable function and it has the property that $t\mapsto V_h(x+th)$ is always absolutely continuous, being either equal to $\f_{x,h}$ or to the zero function. Finally, we have that for $x$ in the full $[X_0]$-measure set $E_0\smallsetminus N_h$, then $V_h(x+th)=\f_{x,h}(t)=V(x+th)$ for almost every $t\in \R,$ therefore the two functions $V$ and $V_h$ must coincide on a full measure subset of $E$.

$(ii)\implies (i).$ Let us denote by $N\subset E,$ the set of all $x$ such that $V(x)\neq V_h(x)$. By hypotheses, we know that $\mu(N)=0.$
By the Cameron-Martin theorem, the measures $[X]$ and $[X+th]$ are absolutely continuous one with respect to the other, therefore we have that $\mu(N+th)=0$ for all $t\in \R.$ Using Tonelli theorem to exchange the order of integration, we have that
\be 
0=\int_\R \mu(N+th) dt=\int_E \tyu\int_\R 1_N(x+th)dt\uyt d\mu(x).
\ee
This identity says that the set $N_h=\{\int_\R 1_N(x+th)dt\neq 0\}$ has measure zero for $\mu.$ Moreover, if $x\in E\smallsetminus N_h,$ then $1_N(x+th)=0$ for almost every $t\in \R,$ which means that the function $t\mapsto V(x+th)$ coincides almost everywhere with the absolutely continuous function $t\mapsto \f(t)=V_h(x+th).$
\end{proof}
\section{Proof of Theorem \ref{thm:bestia}}\label{apx:compensation}
Let the partitions $E=\m U\sqcup \m W$ and $\m W=\m U_1 \sqcup \m W_1$ be defined as in \cref{sec:degMorse}. 
Recall that $C=\CZ_f\neq \emptyset$ if and only if $X+t_0e\in \m W$.
Notice that, the points \ref{bestia:im} and \ref{ie} of \cref{thm:bestia} together are equivalent to say that $X+t_0e\in \m U_1$ and that $t\mapsto X+t_0e$ is a transverse curve around $t_0$. As a consequence, we have that the set
\be 
\{t\in\R\colon X+te \notin \m U\}=\{t\in\R\colon X+te \in \m U_1\}
\ee 
is locally finite.
We will see (\cref{prop:3pt}) that the behavior of the nodal volume $V(X+te)$ around $t_0$ is dictated by the index $\lambda$ and all the signs $e(p)$ and, when $S=\de M$, also by the sign of $b(p)(X+t_0e)$.

Define for $\ell=1,\dots,5$ a subset $\m W_\ell\subset E$ as follows:
\be 
\m W_\ell:=\kop f\in E\colon C:=\CZ_f\neq \emptyset \text{ and ($\ell$) in Theorem \ref{thm:bestia} is false }
\pok.
\ee
The subset $\m W_1$ is indeed the one that was defined above.
To be precise, the subsets $\m W_2,\m W_4\subset E$ depend on $e$, while the others are well defined subsets of $E$. Even $\m W_5$ is the same for all Gaussian measures with support $E$. Define $\m W_\times(e):=\cup_{\ell=\in 1}^4\m W_\ell$, for every $e\neq 0\in E$. Then we have a new partition $E=\m U\sqcup \m W$ and $\m W=\m U_\times (e) \sqcup \m W_\times(e) $, where $W_\times(e)\subset \m W_1$.

For the proof of \cref{thm:bestia}, we will need some lemmas. 
\begin{lemma}
Let $V$ be a finite dimensional vector space and let $A\colon E\to V$ be a continuous linear map. Then the Gaussian vector $A X\randin V$ is non-degenerate if and only if $A$ is surjective. Moreover, given a family $A_i\colon E\to V_i$ of surjective continuous linear maps, then the Gaussian vectors $A_iX$ are fully correlated if and only if  the kernels $\ker(A_i)$ are the same for all $i$.
\end{lemma}
\begin{proof}
The first part is straightforward. For the second observe that the all the kernels coincide with a space $K$ if an only if there is a common orthonormal basis $h_1,\dots, h_d\in \HX$ of $E_0=(K\cap \HX)^\perp$, so that $a_i:=A_i|_{E_0}\colon E_0\to V_i$ is an isomorphism for all $i$ and thus the Gaussian vector $A_iX=a_i\circ a_{i_0}^{-1} \tyu A_{i_0}X\uyt$
is completely determined by $A_{i_0}X$, for fixed $i_0$.
\end{proof}
\begin{lemma}
Let $V_1,V_2$ be two finite dimensional vector spaces, let $H_1,H_2$ be two nondegenerate quadratic forms on $V_1,V_2$, respectively. Let $A_i\colon E\to V_i$, for $i=1,2$, be two surjective continuous linear maps. Assume that for all $h\in E$, we have that 
\be 
H_1(A_1(h)^{\otimes 2})=H_2(A_2(h)^{\otimes 2}).
\ee
Then $\ker{A_1}=\ker{A_2}$, $\dim V_1=\dim V_2$ and the quadratic forms $H_1,H_2$ have the same index.
\end{lemma}
\begin{proof}
That the kernels must coincide, it is obvious.
Arguing as in the previous proof, we can show that there is a finite dimensional space $E_0$ such that $a_i:=A_i|_{E_0}\colon E_0\to V_i$ is an isomorphism. From this we deduce immediately that $\dim V_1=\dim V_2$ and that the matrices of $H_i$ with respect to the bases $a_i(h_1),\dots, a_i(h_d)$ are the the same matrix for $i=1,2$. Therefore, their canonical diagonal form is the same and the lemma is proved.
\end{proof}
The following technical lemma is a very efficient tool to prove that certain events have zero probability when more standard transversality arguments (like those on which \cref{lem:Ptransv}, \cref{lem:bulinskaya} and  \cref{thm:transcurve} rely on) are not available. It is a key ingredient in the proof of \cref{thm:bestia} and we will employ it again in the proof of \cref{thm:sing}, to prove the validity of \cref{e:magnificentlaw} in point (i) of \cref{t:mainintro}.
\begin{lemma}
\label{lem:blind}
Let $\m V\subset E$ and fix $e\in E$. Assume that for every $f_0\in \m V$ there is a neighborhood  $O_{f_0}$ of $f_0$ in $E$, $\delta>0$ and $h\in \HX\smallsetminus \R e$ such that for all $f\in O_{f_0}$, the set
\be 
T(f,\delta,h):=\kop(s,t)\in (-\delta,\delta)^2\colon f+sh+te \in \m V\pok
\ee
has the property that $T(f,\delta,h)\cap \{s\}\times \R $ is empty for almost every $s\in(-\e,\e)$. 
Then, $\PP\tyu\exists t\in \R: X+te\in \m V\uyt=0$.
\end{lemma}
\begin{remark}
In particular, if $T(f,\delta,h)$ is finite, then clearly $T(f,\delta,h)\cap \{s\}\times \R $ is empty for almost every $s\in(-\e,\e)$. Notice that, when $e=0$, we have that \be 
T(f,\delta,h)=\kop s\in (-\delta,\delta)\colon f+sh\in \m V\pok\times (-\delta,\delta)
\ee 
and the conclusion is that $\PP\kop X\in \m V\pok=0$.
\end{remark}
\begin{proof}
The space $E$ is a closed subspace of $\mC^2(M)$, thus it has a countable basis $\m B$ of open sets, for this reason, it is enough to prove the statement for $\m V\cap O_{f_0}$.

It is not restrictive to assume from the beginning that $\|h\|_{\HX}=1$ and that $c:=\langle h, e \rangle_{\HX}\le 1$. Let $e=e_0+ch$.
Using the splitting $\HX=h^\perp\oplus_\perp \R h$ we have that $X=X_h+\gamma h$ for some $\gamma\sim \m N(0,1)$ and some $X_h$ Gaussian vector supported on $h^\perp$, independent from $\gamma$.

Observe that the set $\Theta_{f_0}:=\{(t,f)\in \R\times E : f+te\in O_{f_0}\}$ is an open subset of $\R\times E$, therefore it can be covered by a countable family of open subsets $(\beta_{n})_{n\in \N}$  of rectangular form: 
\be 
\beta_n=(t_n-\e_n,t_n+\e_n)\times \hat B_n, \text{ where } \hat B_n=B_n\oplus (s_n-\e_n,s_n+\e_n)h
\ee 
for some $t_n, s_n\in \R$, $0<\e_n<\frac 12 \delta$, where $\hat B_n$ is a an open subset of $E$ and $B_n$ is an open subset of $h^\perp$ such that $\hat B_n+(-\e_n,\e_n)h+t_ne
\subset O_{f_0}$.  Define $X_n:=
X+(s_n-\gamma)h+t_n e$, so that 
\bega 
X+te 
 =X_n+(\gamma-s_n)h+(t-t_n) e.
 \eega
Observe that, if $X\in \hat B_n$, then $|\gamma-s_n|<\e_n$ and hence $X_n\in O_{f_0}$.
Therefore, we have 
\begin{multline}
\PP\tyu\exists t\in \R: X+te\in \m V\cap O_{f_0} \uyt \le \sum_{n\in \N}\PP\tyu\exists t\in (t_n-\e_n,t_n+\e_n): X\in \hat B_n, \ X+te\in \m V\uyt 
\\ 
\le\sum_{n\in \N}\PP\tyu\exists t\in (-\delta,\delta): X_n\in O_{f_0};\ |\gamma-s_n|<\delta; \ X_n+(\gamma -s_n) h +t e \in \m V \uyt
\\
\le
\sum_{n\in \N} \E\tyu 1_{O_{f_0}}(X_n)\int_{-\delta}^{\delta}
\# \Big( T(X_n,\delta,h)\cap \{s\}\times \R\Big)  \cdot 
\frac{e^{-\frac{(s+s_n)^2}{2}}}{\sqrt{2\pi}} ds \uyt=0.
\end{multline}
indeed, the hypothesis implies that if $f\in O_{f_0}$ and $|s|<\delta$, then $T(f,\delta,h)\cap \{s\}\times \R $ is empty for almost every $s\in(-\e,\e)$. 
\end{proof}
\begin{proof}[Proof of \cref{thm:bestia}]
Define for $\ell=1,\dots,5$ a subset $\m W_\ell\subset E$ as follows:
\be 
\m W_\ell:=\kop f\in E\colon C:=\CZ_f\neq \emptyset \text{ and ($\ell$) is false }
\pok.
\ee
The subset $\m W_1$ is indeed the one that was defined above.
To be precise, the subsets $\m W_2,\m W_4\subset E$ depend on $e$, while the others are well defined subsets of $E$. Even $\m W_5$ is the same for all Gaussian measures with support $E$. We must prove that $\PP\{\exists t \in\R : X+te\in \cup_{\ell=1}^5\m W_\ell\}=0$.

We already know  by \cref{thm:transcurve} that, almost surely, the curve $t\mapsto X+te$ is transverse to $\m W$, which means  (
as e pointed out in the discussion at the beginning of this appendix) that $\PP\{\exists t \in\R : X+te\in \m W_1\cup \m W_2\}=0$. Moreover observe that, by definition, $\m W_\ell\subset \m W$ for all $\ell$. So, we have $\m V_\ell(e):=\m W_\ell\smallsetminus  \m (W_1\cup \m W_2)$, $\m V(e):=\m V_3(e)\cup\m V_4(e) \cup \m V_5 (e)\subset \m U_1\smallsetminus \m W_2(e)$.
To conclude, we must prove that
\be 
\PP\{\exists t \in\R : X+te\in \m V(e) \}=0.
\ee

We will apply \cref{lem:blind}
Let $f_0\in \m V(e)$. Then, $\CZ_{f_0}\subset M$ is a non-empty finite subset and the conditions \ref{bestia:im}, \ref{ie} hold true. 
Assume that $\CZ_{f_0}\cap \inter M=\{P_1(f_0),\dots,P_k(f_0)\}$ and $\CZ_{f_0}\cap  \de M=\{y_{1},\dots,y_{h}\}$. For any elements $\seg p=(p_1,\dots,p_k)$ of $M^k$, we denote $\delta_{\seg p}f:=\tyu f(p_1),\dots,f(p_k)\uyt \in \R^k$ and $d_{\seg p}f:=(d_{p_1}f,\dots, d_{p_k}f)\in (T^* M)^k$. We will use analogous notations whenever we deal with tuples of points.
Consider the function 
\bega 
\psi:E\times M^{k}\times (\de M)^{h} &\to (T^*M)^k\times (T^*\de M)^h
\\
\psi(f,\seg p,\seg q)&:=\tyu d_{\seg p}f,d_{\seg q} f |_{\de M}\uyt.
\eega
It is clear that $\psi$ is a $\mC^1$ function in the Banach sense, where $E$ is naturally endowed with the $\mC^2$ Banach structure. Let us compute the differential of $\psi$ at the point $(f_0,\seg P,\seg Q)$ in the direction $(0,\seg{\dot p}, \seg{\dot q})\in E\times (T_{\seg P}M)^k\times (T_{\seg Q}\de M)^h$. Using the Levi-Civita connection, we can represent it as 
\bega 
\langle d_{(f_0,\seg P,\seg Q,0)}\psi, (0,\seg{\dot p}, \seg{\dot q},0)\rangle=\tyu \H_{\seg P}f(\seg{\dot p},\cdot ), \H_{\seg Q}f|_{\de M}(\seg{\dot q},\cdot )\uyt\in (T_{\seg P}^*M)^k\times (T_{\seg Q}^*\de M)^h.
\eega
By \ref{bestia:im}, the above linear map is surjective, so that by the Implicit Function Theorem, we have the following. There is a neighborhood $O\subset E$ of $f_0$ and a $\mC^1$ function
\bega 
\tyu \seg P (\cdot),\seg  Q(\cdot) \uyt: O &\to M^k\times \de M^h, 
\quad \text{ s.t. }\\
\psi^{-1}(\seg 0)\cap \tyu O\times M^k\times \de M^h\uyt
&=\kop\tyu f, \seg P (f),\seg Q (f) \uyt :f\in O\pok.
\eega

In other words, every $f\in O$ has exactly one critical point in a neighborhood $B_i$ of $P_i(f_0)$ in $\inter{M}$ and exactly one in an neighborhood $\de B_{k+j}$ of $Q_j(f_0)$ in $\de M$ and they are $P_i(f)$ and $Q_j(f)$, respectively, for every $i=1,\dots,k$ and $j=1,\dots,h$. However, each of them might or might not be a critical zero. Let $\seg x(f):=(\seg P(f),\seg Q(f))\in M^{k+h}$ and define $\phi_a\colon O\to \R$ such that $\phi_a(f):=f(x_a(f))$, for $a\in \{1,\dots, k+h\}$.  
Observe that if $f \in \m V(e)$, then $f$ must have two or more critical zeroes. This means that 
\bega 
\m V(e) \cap O =\bigcup_{\ell=3,4,5}\bigcup_{
\ 1\le a<b\le k+h} \m V_{a,b}^\ell(O,e), \ \text{where }
\\
 \m V_{a,b}^\ell(O,e):=\kop f\in O\colon \phi_a(f)=\phi_b(f)=0\text{ and } C=\{x_a(f),x_b(f)\} \text{ violates } (\ell)  \pok
\eega
Now, we fix $a,b$ and focus our attention on $\m V_{a,b}(O,e)=\cup_{\ell=3,4,5}\m V_{a,b}^\ell(O,e)$. Assume that $f_0\in \m V_{a,b}(O,e)$.
We want to show the following claim: 

\begin{claim}\label{claim}
There exist $O_{f_0}$ small enough neighborhood of $f_0$ in $E$, $\delta>0$ and $h\in \HX$ such that for all $f\in O_{f_0}$, the set below is finite:
\bega 
T(f,\delta , h):&=\kop (s,t) \in  (-\delta,\delta)^2: f+s h+te\in V_{a,b}(O,e)\pok
\eega
\end{claim}

Notice that there are three cases: when $b\le k$, $a\ge k+1$ and $a\le k<b $. We start by considering the latter situation, corresponding to having both a critical zero $P(f):=x_a(f)\in \inter{M}$ in the interior and one  $Q(f):=x_b(f)\in \de M$ in the boundary. 
By construction, for every $f\in O$ and $\dot f\in E$ we have
\bega\label{eq:second} 
0=\langle d_{f}\tyu d_{P(f)}f \uyt , \dot f \rangle &=
 \H_{P(f)}f \tyu d_fP(\dot f),\cdot \uyt 
+ d_{P(f)}\dot f 
\\
0=\langle d_{f}\tyu d_{Q(f)}f|_{\de M} \uyt , \dot f \rangle &=
 \H_{Q(f)}f|_{\de M} \tyu d_fQ(\dot f),\cdot \uyt 
+ d_{Q(f)}\dot f|_{\de M}.
\eega 
Keeping in mind the above identities, we compute the first and second differentials of $\phi_a$ and $\phi_b$ at $f\in O$, along any curve $s\mapsto f(s)\in E$ such that $f(0)=f$.

\bega\label{eq:dfpf1}
d_{f}\phi_a(\dot f)&=\langle d_{P(f)}f , d_fP(\dot f)\rangle + \dot f \tyu P(f) \uyt = \dot f \tyu P(f) \uyt
\\
d_{f}\phi_b(\dot f)&=\langle d_{Q(f)}f |_{\de M}, d_fQ(\dot f)\rangle + \dot f \tyu Q(f) \uyt = \dot f \tyu Q(f) \uyt
\eega
From this we see that , $\phi_a$ and $\phi_b$ are of class $\mC^2$ and
\bega\label{eq:dfpf2}
\H_{f}\phi_a(\dot f^{\otimes 2})&=\langle d_{P(f)}\dot f, d_fP(\dot f)\rangle =
\ddot{f}(P(f))
- \qwe\H_{P(f)}f \ewq^{-1}\tyu  d_{P(f)}\dot f ^{\otimes 2}\uyt
\\
\H_{f}\phi_b(\dot f^{\otimes 2})&=\langle d_{Q(f)}\dot f, d_fQ(\dot f)\rangle =
\ddot{f}(Q(f))
- \qwe\H_{Q(f)}f|_{\de M} \ewq^{-1}\tyu  d_{Q(f)}\dot f|_{\de M} ^{\otimes 2}\uyt
\eega

Denote $p_0:=P(f_0)\in \inter{M}$ and $q_0=Q(f_0)\in\de M$, that, we recall, are two of the critical zeroes of $f_0$. Let $O'\times (-\e,\e)\subset E\times \R$ and be an open neighborhood of $(f_0,0)$ such that $(f,t)\in O'\times (-\e,\e)\implies f+te\in O$. 

Since $\frac{\de \phi_{a}}{\de t}|_{t=0}(f_0+te)=e(p_0)\neq 0$, by \ref{ie}, we can apply the Implicit Function Theorem to deduce that if we take $O'$ and $\e>0$ small enough, there exists a $\mC^2$ function $\tau \colon O'\to (-\e,\e)$ such that $\phi_{a}^{-1}(0)\cap \kop f+te: f\in O', t\in (-\e,\e)
 \pok=\kop f+\tau(f)e\colon f\in O'\pok$. Define a real valued function $\f(f):=\phi_{b}(f+\tau(f)e)$, for all $f\in O'$, then 
\bega 
T(f,\delta,h)&\subset  \kop (s,t)\in (\delta,\delta)^2\colon \phi_a(f+sh+te)=0=\phi_b(f+sh+te)\pok
\\
&= \kop (s,\tau(f+sh))\in (\delta,\delta)^2\colon \f(f+sh)=0\pok.
\eega
By construction, $\f$ is of class $\mC^2$, hence, given any $h\in \HX$, we have a Taylor expansion 
\be 
\f(f+sh)=\f(f)+d_f\f(h)s+\frac 12\H_f\f(h,h)s^2+o(s^2),
\ee 
uniformly for all $f\in O'$ (after replacing $O'$ with a smaller neighborhood, if needed),
hence if \cref{claim} doesn't hold, we must have that $\f(f_0)=d_{f_0}\f(h)=\H_{f_0}\f(h,h)=0$. We already know that $\f(f_0)=f_0(p_0)=0$ because $p_0\in \CZ_{f_0}$. By differentiating the identity $\phi_a(f_0+sh+\tau(f_0+sh)e)=0$ with respect to $s$, using the formulas \cref{eq:dfpf1} and \cref{eq:dfpf2}, and denoting $\dot{\tau}=d_{f_0}\tau(h)$ and $\ddot{\tau}=\H_{f_0}\tau(h^{\otimes 2})$, we get the following two identities:
\bega 
0&=h(p_0)+\dot{\tau} e(p_0)
\\
0&=e(p_0)\ddot{\tau} - \qwe\H_{p_0}f \ewq^{-1}\tyu  (d_{p_0}h +\dot{\tau} d_{p_0}e)^{\otimes 2}\uyt
\eega 
Therefore, for all $h$ such that $h(p_0)=0$, the following two equations are satisfied:
\bega
d_{f_0}\f(h)&=h(q_0)
\\
\H_{f_0}\f(h,h)&=e(q_0)\ddot\tau  - \qwe\H_{q_0}f|_{\de M} \ewq^{-1}\tyu  (d_{q_0}h|_{\de M}) ^{\otimes 2}\uyt,
\eega
From $d_{f_0}\f(h)=0$ for all $h\in\HX$, since also $e(q_0)\neq 0$ by \ref{ie}, we conclude that $ K:=\ker(\delta_{p_0})=\ker(\delta_{q_0})$ as subspaces of $E$. From $\H_{f_0}\f(h,h)=0$, we obtain the following for all $h\in K$:
\be\label{eq:finalmix} 
\frac{1}{e(p_0)}\qwe\H_{p_0}f \ewq^{-1}\tyu  (d_{p_0}h )^{\otimes 2}\uyt
=
\frac{1}{e(q_0)} \qwe\H_{q_0}f|_{\de M} \ewq^{-1}\tyu  (d_{q_0}h|_{\de M}) ^{\otimes 2}\uyt
\ee
This is possible only if $\ker(d_{q_0}|_{\de M})\cap \ker(\delta_{q_0})=\ker(d_{p_0})\cap \ker(\delta_{p_0})$, but that is impossible because those spaces have different codimensions, due to the assumption that $j^1_pX$ is non-degenerate for every $p\in M$. This is the contradiction. 
So \cref{claim} holds in the case $a\le k<b$. 

Let us consider the case $a<b\le k$, thus $f_0\notin \m V_{a,b}^3(O,e)$. Then $q_0=x_b(f)\in \inter{M}$ is a critical point of $f|_{\inter{M}}$. Arguing exactly as in the previous case, we arrive in the end to deduce that $K:=\ker(\delta_{p_0})=\ker(\delta_{q_0})$ and that for every $h\in K$ the identity  \cref{eq:finalmix} is replaced by 
\be\label{eq:finalint} 
\frac{1}{e(p_0)}\qwe\H_{p_0}f \ewq^{-1}\tyu  (d_{p_0}h )^{\otimes 2}\uyt
=
\frac{1}{e(q_0)} \qwe\H_{q_0}f \ewq^{-1}\tyu  (d_{q_0}h) ^{\otimes 2}\uyt
\ee
Again, this implies that \be 
\ker (j^1_{p_0})=\ker(d_{p_0})\cap \ker(\delta_{p_0})=\ker(d_{q_0})\cap \ker(\delta_{q_0})=\ker (j^1_{q_0}),
\ee
that is equivalent to say that $j^1_{p_0}X$ and $j^1_{q_0}X$ are completely correlated random vectors, which means that $f_0\notin \m V_{a,b}^5(e)$. Moreover, \cref{eq:finalint} also implies that the two bilinear forms $\frac{1}{e(p_0)}\qwe\H_{p_0}f \ewq^{-1}$ and $\frac{1}{e(q_0)}\qwe\H_{q_0}f \ewq^{-1}$ have the same index, that is equivalent (by inverting the corresponding matrices) to say that $\ref{bestia:ii}$ holds for $C=\{p_0,q_0\}$ and function $f_0$, thus $f_0\notin \m V_{a,b}^4(e)$. We showed that $f_0\notin \m V_{a,b}(e)=\cup_{\ell=3,4,5}\m V_{a,b}^\ell(e)$, which is a contradiction since we started by assuming that $f_0\in \m V_{a,b}(e)$. Thus, we proved \cref{claim} in the case $a<b\le k$.
The case $k<a<b$ is completely analogous to the previous case, indeed we can apply the previous argument to $X|_{\de M}$ on the closed manifold $\de M$. Finally,  \cref{claim} is proven.
\end{proof}
\end{appendix}

%


\bibliographystyle{abbrv}
\bibliography{Nodal_Volumes_as_differentiable_functionals_of_Gaussian_fields}       

\end{document}